\documentclass[11pt]{article}

\usepackage{graphicx}

\usepackage{bbm}

\usepackage{amsmath}
\usepackage{amsthm}
\usepackage{amsfonts}
\usepackage{amssymb}
\usepackage{xcolor}
\usepackage{tcolorbox}

\newcommand{\eee}{{\rm e}}

\newcommand{\me}{\mathbb{E}}
\newcommand{\mn}{\mathbb{N}}
\newcommand{\mr}{\mathbb{R}}
\newcommand{\mz}{\mathbb{Z}}
\renewcommand{\P}{\mathbb{P}}
\DeclareMathOperator{\1}{\mathbbm{1}}

\newcommand{\E}{\mathbb{E}}
\newcommand{\Q}{\mathbb{Q}}
\newcommand{\var}{{\rm Var \,}}

\newcommand{\dd}{{\rm d}}

\newcommand{\mmp}{\mathbb{P}}

\newtheorem{thm}{Theorem}[section]
\newtheorem{lemma}[thm]{Lemma}

\newtheorem{cor}[thm]{Corollary}

\newtheorem{assertion}[thm]{Proposition}
\theoremstyle{definition}

\theoremstyle{remark}
\newtheorem{rem}[thm]{Remark}

\begin{document}
\title{Laws of the iterated and single logarithm for sums of independent indicators, with applications to the Ginibre point process and Karlin's occupancy scheme}\date{}
\author{Dariusz Buraczewski\footnote{Mathematical Institute, University of Wroclaw, Poland;\newline e-mail address: dariusz.buraczewski@math.uni.wroc.pl} \ \ Alexander Iksanov\footnote{Faculty of Computer Science and Cybernetics, Taras Shevchenko National University of Kyiv, Ukraine; e-mail address:
iksan@univ.kiev.ua} \ \ and \ \ Valeriya Kotelnikova\footnote{Faculty of Computer Science and Cybernetics, Taras Shevchenko National University of Kyiv, Ukraine and Mathematical Institute, University of Wroclaw, Poland; e-mail address: valeria.kotelnikova@unicyb.kiev.ua}}
\maketitle

\begin{abstract}
We prove a law of the iterated logarithm (LIL) for an infinite sum of independent indicators parameterized by $t$ and monotone in $t$ as $t\to\infty$. It is shown that if the expectation $b$ and the variance $a$ of the sum are comparable, then the normalization in the LIL includes the iterated logarithm of $a$. If the expectation grows faster than the variance, while the ratio $\log b/\log a$ remains bounded, then the normalization in the LIL includes the single logarithm of $a$ (so that the LIL becomes a law of the single logarithm). Applications of our result are given to the number of points of the infinite Ginibre point process in a disk and the number of occupied boxes and related quantities in Karlin's occupancy scheme.
\end{abstract}

\noindent Key words: Ginibre point process; independent indicators; infinite occupancy; law of the iterated logarithm

\noindent 2020 Mathematics Subject Classification: Primary:
60F15, 60G50
\hphantom{2020 Mathematics Subject Classification: } Secondary: 60C05, 60G55

\section{Main results}
The main purpose of the article is to prove a law of the iterated logarithm (LIL) for an infinite sum of independent indicators. Although this problem is interesting on its own, the present work has originally been motivated by an application of the aforementioned LIL to the number of points of the infinite Ginibre point process in a disk and the number of occupied boxes and related objects in Karlin's occupancy scheme.

Let $(A_1(t))_{t\geq 0}$, $(A_2(t))_{t\geq 0},\ldots$ be independent families of events defined on a common probability space $(\Omega, \mathcal{F}, \mmp)$. The events $A_k(t)$ and $A_k(s)$, with $t\neq s$, are dependent. Formally, this fact will be secured by assumption (A4) stated below. Put $X(t):=\sum_{k\geq 1}\1_{A_k(t)}$ and assume that, for each $t\geq 0$, $X(t)<\infty$ almost surely (a.s.), and moreover $b(t):=\me X(t)=\sum_{k\geq 1}\mmp (A_k(t))<\infty$. Then also $$a(t):={\rm Var}\,X(t)=\sum_{k\geq 1}\mmp(A_k(t))(1-\mmp(A_k(t)))\leq b(t)<\infty.$$ Actually, according to Lemma \ref{lem:ineq} below, $\me \eee^{\theta X(t)}<\infty$ for all $\theta>0$.

We start by formulating a central limit theorem for $X(t)$. The following assumption is of principal importance for all our results:

\begin{tcolorbox}[colback=green!5,colframe=green!40!black]
(A1) $\lim_{t\to\infty}a(t)=\infty$.
\end{tcolorbox}

\begin{assertion}\label{prop:clt}
Suppose (A1). Then $(X(t)-\me X(t))/({\rm Var}\,X(t))^{1/2}$ converges in distribution as $t\to\infty$ to a random variable with the standard normal distribution.
\end{assertion}

Our next result states that the central limit theorem is accompanied by the convergence of exponential moments of all orders.
\begin{assertion}\label{main:exp}
Suppose (A1). Then, for each $\theta \in\mr$,
\begin{equation*}\label{eq:main_statement}
\lim_{t\to\infty} \me \exp\Big(\theta \, \frac{X(t)-\me X(t)}{({\rm Var}\,X(t))^{1/2}}\Big)= \me \exp(\theta \, {\rm Normal}(0,1))=\exp(\theta^2/2),
\end{equation*}
where ${\rm Normal}(0,1)$ denotes a random variable with the standard normal distribution.

Also, for each $p>0$,
\begin{equation}\label{eq:sec_statement}
\lim_{t\to\infty} \me \Big(\frac{|X(t)-\me X(t)|}{({\rm Var}\,X(t))^{1/2}}\Big)^p= \me\, |{\rm Normal}(0,1)|^p=\frac{2^{p/2}\Gamma((p+1)/2)}{\pi^{1/2}},
\end{equation}
where $\Gamma$ is the Euler gamma function.
\end{assertion}

The proofs of Propositions \ref{prop:clt} and \ref{main:exp} will be given in Section \ref{sec:CLT}. These propositions are by no means new and only given as a precursor of the LIL to be stated below. Not intending to provide a detailed survey of results related to the two propositions, we only mention that a local central limit theorem for $X(t)$ and, more generally, the Edgeworth expansions for $\mmp\{X(t)=k\}$ as $t\to\infty$, which hold uniformly in $k\in\mn_0:=\{0,1,2,\ldots\}$, follow from Theorem 11.1 in \cite{Dolgopyat+Hafouta:2022};
a precise two-sided Berry-Esseen type estimate for $X(t)$ can be found in formula (17) of \cite{Mattner+Schulz:2018}; and an approximation of the distribution of $X(t)$ by two terms of the Edgeworth expansion is given in \cite{Volkova:1996}.

The central limit theorem given in Proposition \ref{prop:clt} suggests the form of a LIL to be stated next. To formulate it, we need a set of further assumptions (in addition to (A1)). As usual, $x(t)\sim y(t)$ as $t\to\infty$ will mean that $\lim_{t\to\infty}(x(t)/y(t))=1$.

\begin{tcolorbox}[colback=green!5,colframe=green!40!black]
	(A2) There exists a nondecreasing function $a_0$ such that $a(t)\sim a_0(t)$ as $t\to\infty$.
\end{tcolorbox}

\begin{tcolorbox}[colback=green!5,colframe=green!40!black]
	(A3) There exists $\mu^\ast\geq 1$ such that $b(t)=O((a(t))^{\mu^\ast})$ as $t\to\infty$.
\end{tcolorbox}
\noindent Necessarily, $\mu^\ast\geq 1$, for $a(t)\leq b(t)$ for $t>0$. Put
\begin{equation}\label{eq:infim}
	\mu:=\inf\{\mu^\ast\ge1: b(t)=O((a(t))^{\mu^\ast})\}.
\end{equation}
\begin{rem}
There are two situations: either the infimum in \eqref{eq:infim} is attained, that is, the infimum coincides with the minimum; or the infimum is not attained. The first situation occurs if, for instance, $b(t)~\sim~ c(a(t))^\beta$ for some $c>0$ and $\beta\geq 1$, in which case $\mu=\beta$. The second situation occurs if, for instance, $b(t)~\sim~ (a(t))^\beta f(t)$ for some $f$ diverging to $\infty$ and satisfying
$f(t)=o((a(t))^\varepsilon)$ for all $\varepsilon > 0$, in which case, again, $\mu=\beta$.
\end{rem}

We write that a function defined on $[0,\infty)$ has some property eventually, if it has the property on $[T,\infty)$ for some $T\geq 0$. Recall that a positive measurable function $\ell$, defined on some neighborhood of $\infty$, is called slowly varying at $\infty$, if $\lim_{t\to\infty}(\ell(\lambda t)/\ell(t))=1$ for all $\lambda>0$.

\begin{tcolorbox}[colback=green!5,colframe=green!40!black]
(A3 cont.) If $\mu=1$, we assume that either $b$ is eventually continuous
or $${\lim\inf}_{t\to\infty}(\log b(t-1)/\log b(t))>0$$ and that
\begin{equation}\label{eq:infim2}
	b(t)/a(t)=O(f_{q}(a(t))),\quad t\to\infty,
\end{equation}
where $f_{q}(t)=(\log t)^{q}\mathcal{L}(\log t)$ for some $q\geq 0$ with $\mathcal{L}$ slowly varying at $\infty$
and if $q>0$, $b(t)/a(t)\neq O(f_s(a(t)))$ for $s\in (0,q)$.
\end{tcolorbox}

\noindent In particular, if $q=0$ and $\mathcal{L}(t)$ converges as $t\to\infty$ to a positive constant, $b/a$ is a bounded function.

\begin{tcolorbox}[colback=green!5,colframe=green!40!black]
	(A4) For each $k\in\mn:=\{1,2,\ldots\}$ and $t>s>0$, $A_k(s)\subseteq A_k(t)$.
\end{tcolorbox}

\noindent In particular, $b$ is a nondecreasing function.

In view of (A1) and $a(t)\leq b(t)$ for $t>0$, we infer $\lim_{t\to\infty}b(t)=\infty$. For each $\varrho\in (0,1)$, put
\begin{equation}\label{eq:mutheta}
\mu_\varrho:=\mu+\varrho\quad \text{if}~\mu>1\quad\text{and}\quad q_\varrho:=q+\varrho \quad \text{if}~\mu=1.
\end{equation}
Assuming (A3), fix any $\kappa\in (0,1)$, any $\varrho\in (0,1)$ and put
\begin{equation}\label{tn}
t_n=t_n(\kappa, \mu):=\inf\{t>0: b(t)>v_n(\kappa,\mu)\}
\end{equation}
for $n\in\mn$, where $v_n(\kappa, 1)=v_n(\kappa, 1, q, \varrho)=\exp(n^{(1-\kappa)/(q_\varrho+1)})$ and $v_n(\kappa,\mu)=v_n(\kappa, \mu, \varrho)=n^{\mu_\varrho(1-\kappa)/(\mu_\varrho-1)}$ for $\mu>1$. Plainly, the sequence $(t_n)_{n\in\mn}$ is nondecreasing with $\lim_{n\to\infty} t_n=+\infty$.
\begin{rem}
In the situation that the infimum in \eqref{eq:infim} is attained, we could have defined $v_n(\kappa,\mu)$ and $t_n(\kappa, \mu)$ with $\mu$ replacing $\mu_\varrho$ if $\mu>1$ and $0$ replacing $q_\varrho$ if $\mu=1$. This possibility will become clear after an inspection of the proof. For the time being, this can be informally justified as follows. If the infimum is attained, then $b(t)\leq {\rm const}\,(a(t))^\mu$ for large\footnote{The phrase `for each $n$ ($t$) large enough' means that there exists $n_0\in\mn$ ($t_0>0$) and we take any $n\geq n_0$ ($t\geq t_0$). Analogously, `for each positive $\kappa$ sufficiently close to $0$' means that there exists $\kappa_0>0$ and we take any $\kappa\in (0,\kappa_0)$.} $t$,
whereas if the infimum is not attained and $\mu>1$ ($\mu=1$), then $b(t)\leq {\rm const}\,(a(t))^{\mu_\varrho}$ ($b(t)\leq (\log a(t))^{q_\varrho} a(t)$) for each $\varrho\in (0,1)$ and large $t$. Our formulation is intended to simplify the subsequent presentation and avoid dealing with multiple cases.
\end{rem}

\begin{tcolorbox}[colback=green!5,colframe=green!40!black]
(A5) For each positive $\kappa$ sufficiently close to $0$ and for each $n$ large enough, there exists $A>1$ and a partition $t_n=t_{0,\,n}<t_{1,\,n}<\ldots<t_{j,\,n}=t_{n+1}$ with $j=j_n$ satisfying $$1\leq b(t_{k,\,n})-b(t_{k-1,\,n})\leq A,\quad 1\leq k\leq j$$
and, for all $\varepsilon>0$, $\big(j_n\exp(-\varepsilon (a(t_n))^{1/2})\big)$ is a summable sequence.
\end{tcolorbox}

\begin{rem}\label{suff}
A sufficient condition for (A5) is that $b$ is eventually strictly increasing and eventually continuous. Indeed, one can then choose a partition that satisfies, for large $n$, $b(t_{k,\,n})-b(t_{k-1,\,n})=1$ for $k\in\mn$, $k\leq j-1$ and $b(t_{j,\,n})-b(t_{j-1,\,n})\in [1,2)$. As a consequence,
$$j_n=\lfloor v_{n+1}(\kappa, \mu)-v_n(\kappa,\mu)\rfloor,$$ and, by Lemma \ref{lem:aux1}(b) below, $j_n=o(a(t_n))$ as $n\to\infty$. The sequence $\big(j_n\exp(-\varepsilon (a(t_n))^{1/2})\big)$ is summable because $a(t_n)$ grows at least polynomially fast.
\end{rem}

Assuming (A1) and (A3), fix any $\gamma>0$ and put
\begin{equation}\label{eq:tau}
\tau_n=\tau_n(\gamma,\mu):=\inf\{t>0: a(t)>w_n(\gamma,\mu)\}
\end{equation}
for large $n\in\mn$ with $\mu$ as given in \eqref{eq:infim}. Here, with $q$ as given in \eqref{eq:infim2}, $w_n(\gamma, 1)=w_n(\gamma, 1, q)=\exp(n^{(1+\gamma)/(q+1)})$ if $\mu=1$ and $w_n(\gamma, \mu)=n^{(1+\gamma)/(\mu-1)}$ if $\mu>1$.

\ \smallskip Now we formulate our remaining assumptions.

\begin{tcolorbox}[colback=green!5,colframe=green!40!black]
	(B1) The function $a$ is eventually continuous or $\lim_{t\to\infty}(\log a(t-1)/\log a(t))=1$ if $\mu=1$ and $\lim_{t\to\infty}(a(t-1)/a(t))=1$ if $\mu>1$.
\end{tcolorbox}
\begin{tcolorbox}[colback=green!5,colframe=green!40!black]
	(B2.1)
	There exist $s_0>0$, $\varsigma_0>0$ and a family $(R_{\varsigma}(t))_{0<\varsigma < \varsigma_0, t>s_0}$ of sets of positive integers satisfying the following two conditions: for each $\gamma>0$ close to 0 and all $0<\varsigma < \varsigma_0$ there exists $n_0=n_0(\varsigma, \gamma)\in\mn$ such that the sets $R_\varsigma(\tau_{n_0}(\gamma,\mu))$, $R_\varsigma(\tau_{n_0+1}(\gamma,\mu)),\ldots$ are disjoint; and
	\begin{equation}\label{eq:onevar}
		\lim_{t\to\infty}\frac{{\rm Var}\Big(\sum_{k\in 	R_\varsigma(t)}\1_{A_k(t)}\Big)}{{\rm Var}\,X(t)}=1-\varsigma.
	\end{equation}
\end{tcolorbox}
\begin{tcolorbox}[colback=green!5,colframe=green!40!black]
	(B2.2)
	There exist $s_0>0$ and a family $(R_0(t))_{t>s_0}$ of sets of positive integers satisfying the following two conditions: for each $\gamma>0$ close to $0$ there exists $n_0=n_0(\gamma)\in\mn$ such that the sets $R_0(\tau_{n_0}(\gamma,\mu))$, $R_0(\tau_{n_0+1}(\gamma,\mu)),\ldots$ are disjoint; and
	\begin{equation}\label{eq:one}
		\lim_{t\to\infty}\frac{{\rm Var}\Big(\sum_{k\in R_0(t)}\1_{A_k(t)}\Big)}{{\rm Var}\,X(t)}=1.
	\end{equation}
\end{tcolorbox}

We are ready to state a LIL, which is the main result of the present article.
\begin{thm}\label{thm:main}
Suppose (A1)-(A5), (B1) and either (B2.1) or (B2.2). Then, with $\mu\geq 1$ and $q\geq 0$ as defined in \eqref{eq:infim} and \eqref{eq:infim2}, respectively, $$\limsup_{t\to\infty}\frac{X(t)-\me X(t)}{(2(q+1){\rm Var}\,X(t)\log\log
{\rm Var}\,X(t))^{1/2}}=1\quad\text{{\rm a.s.}}$$ if $\mu=1$ and $$\limsup_{t\to\infty}\frac{X(t)-\me X(t)}{(2(\mu-1){\rm Var}\,X(t)\log {\rm Var}\,X(t))^{1/2}}=1\quad\text{{\rm a.s.}}$$ if $\mu>1$.
\end{thm}

Theorem \ref{thm:main} will be proved in Section \ref{sec:main}.

\begin{rem}
A perusal of the proof of Theorem \ref{thm:main} reveals that the corresponding lower limits in Theorem \ref{thm:main} are equal to $-1$ a.s.
\end{rem}

\section{A short overview of relevant literature}

For history up to 1985 of development of LILs for various stochastic models we refer to an excellent survey \cite{Bingham:1986}.

Now we provide an incomplete list of relatively recent articles, which are concerned with LILs for infinite sums of deterministically weighted independent and identically distributed random variables. Let $\xi_1$, $\xi_2,\ldots$ be independent copies of a random variable $\xi$ with $\me \xi=0$ and ${\rm Var}\,\xi\in (0,\infty)$. A LIL for a random geometric series $\sum_{k\geq 0}b^k\eta_{k+1}$ as $b\to 1-$ was proved in Theorem 1.1 of \cite{Bovier+Picco:1993} with the help of classical methods and in Theorem 2.1 of \cite{Zhang:1997} via a strong approximation argument. See also \cite{Fu+Huang:2016, Gao+Gao+Xia:2024, Stoica:2003} and Theorem 1(iii) in \cite{Kiesel:1996} for various extensions. A LIL for a random Dirichlet series $\sum_{k\geq 2}k^{-1/2-s}(\log k)^\alpha \xi_k$ as $s\to 0+$ was proved in \cite{Aymone+Frometa+Misturini:2020} in the case $\mmp\{\xi=\pm 1\}=1/2$ and $\alpha=0$ and in \cite{Buraczewski etal:2023} and \cite{Iksanov+Kostohryz:2025} in the cases $\alpha>-1/2$ and $\alpha=-1/2$, respectively. One may expect that the asymptotic behavior of the random Dirichlet series with $\alpha=-1/2$ mimics that of $\sum_p p^{-1/2-s}\xi_p$, where $\sum_p$ denotes the summation over the prime numbers.

The most relevant to our setting is the article \cite{Lai+Wei:1982}. It proves a LIL for $\sum_{i\in\mz}a_{n,i}\eta_i$ as $n\to\infty$, where $(a_{n,i})_{n\geq 1, i\in\mz}$ is a sequence of reals satisfying certain conditions, and $\ldots, \eta_{-1}$, $\eta_0$, $\eta_1,\ldots$ are independent random variables with $\me \eta_i=0$, ${\rm Var}\,\eta_i=\sigma^2\in (0,\infty)$ and $\sup_{i\in\mz}\,\me |\eta_i|^r<\infty$ for some $r>2$. We use in one fragment of our proof the argument worked out in that paper, see Remark \ref{rem:laiwei} for more details. Four LILs for $\sum_{i\geq 0}a_{n,i}\eta_i$ as $n\to\infty$, where the weights $(a_{n,i})_{n\geq 1, i\geq 0}$ relate to the four particular summability methods, can be found in Theorem 2 of~\cite{Bingham+Stadtmueller:1990}.

After the present manuscript has been submitted for publication, Iksanov and Kotelnikova published the article \cite{Iksanov+Kotelnikova:2024}, in which Theorem \ref{thm:main} was proved under less restrictive assumptions. Namely, the requirement of monotonicity in assumptions (A2) and (A4) was removed and replaced by a weaker condition. Although the argument given in \cite{Iksanov+Kotelnikova:2024} followed, for the most part, that of the present paper, the technical details were more complicated at places.

\section{An application to the Ginibre point processes}

Let $\mathbb{C}$ and ${\rm Leb}$ denote the set of complex numbers and the Lebesgue measure on $\mathbb{C}$, respectively. As usual, we write $\bar z$ for the complex conjugate of $z\in\mathbb{C}$. Let $\Xi$ be the infinite Ginibre point process on $\mathbb{C}$, that is, a determinantal point process with kernel $C(z,w)=\eee^{z\bar w}$ for $z,w\in \mathbb{C}$ with respect to the measure $\mu$ defined by $\mu({\rm d}z):=\pi^{-1}\eee^{-|z|^2}{\rm Leb}({\rm d}z)$ for $z\in \mathbb{C}$. This means that $\Xi$ is a simple point process such that, for any $k\in\mn$ and any mutually disjoint Borel subsets $B_1,\ldots, B_k$ of $\mathbb{C}$ $$\me \prod_{i=1}^k \Xi (B_i)=\int_{B_1\times\ldots\times B_k} {\rm det}(C(z_i, z_j))_{1\leq i,j\leq k}\,\mu({\rm d}z_1)\ldots\mu({\rm d}z_k).$$ See \cite{Hough+Krishnapur+Peres+Virag:2009} for detailed information on determinantal point processes and particularly Sections 4.3.7 and 4.7 for a discussion of the Ginibre point processes.

For each $t\geq 0$, let $\Xi(D_t)$ denote the number of points of $\Xi$ in the disk $D_t:=\{z\in \mathbb{C}: |z|<t^{1/2}\}$. According to an infinite version of Kostlan's result \cite{Kostlan:1992}, stated as Theorem 1.1 in \cite{Fenzl+Lambert:2021}, the process
\begin{equation}\label{eq:distr}
(\Xi(D_t))_{t\geq 0}~~\text{has the same distribution as}~~N=(N(t))_{t\geq 0}=\Big(\sum_{k\geq 1}\1_{\{\Gamma_k\leq t\}}\Big)_{t\geq 0},
\end{equation}
where $\Gamma_1$, $\Gamma_2,\ldots$ are independent random variables and $\Gamma_k$ has a gamma distribution with parameters $k$ and $1$, that is, $$\mmp\{\Gamma_k\in {\rm d}x\}=\frac{1}{(k-1)!}x^{k-1}\eee^{-x}\1_{(0,\infty)}(x){\rm d}x,\quad x\in\mr.$$

The process $N$ can be thought of as a decoupled version of a Poisson process $\pi$ on $[0,\infty)$ of unit intensity. The expectations of the two processes coincide $\me N(t)=\me \pi_t=t$ for $t\geq 0$, but their variances behave differently. It is shown in Proposition B.1 of \cite{Fenzl+Lambert:2021} that
\begin{equation}\label{eq:var11}
{\rm Var}\,N(t)\sim (t/\pi)^{1/2},\quad t\to\infty,
\end{equation}
whereas ${\rm Var}\,\pi_t=t$ for $t\geq 0$. In view of \eqref{eq:distr}, \eqref{eq:var11} and Proposition \ref{prop:clt}, $(t/\pi)^{-1/4}(\Xi(D_t)-t)$ converges in distribution to a random variable with the standard normal distribution. Weak convergence of finite-dimensional distributions of $(\Xi(D_{ut}))_{u\geq 0}$, properly normalized and centered, to those of a Gaussian white noise was proved in Proposition 1.3 of \cite{Fenzl+Lambert:2021}. A functional central limit theorem for $(\Xi(D_{u+t}))_{u\geq 0}$, properly normalized and centered, was obtained in Proposition 1.4 of the same paper.

An application of Theorem \ref{thm:main} gives, with some additional technical efforts, a LIL for $\Xi(D_t)$.
\begin{thm}\label{thm:Ginibre}
The number of points of $\Xi$ in the disk $D_t$ satisfies
$$\limsup_{t\to\infty}\frac{\Xi(D_t)-t}{t^{1/4}(\log t)^{1/2}}=\frac{1}{\pi^{1/4}}\quad\text{{\rm a.s.}}$$
\end{thm}

Theorem \ref{thm:Ginibre} will be deduced from Theorem \ref{thm:main} in Section \ref{sec:Ginibre}.

\section{An application to Karlin's occupancy scheme}

Let $(p_k)_{k\in\mn}$ be a discrete probability distribution with $p_k>0$ for infinitely many $k$. We are interested in an infinite balls-in-boxes scheme which is defined as follows. There are infinitely many balls and an infinite array of boxes $1$, $2,\ldots$. The balls are allocated, independently of each other, over the boxes with probability $p_k$ of hitting box $k$, $k\in\mn$. The scheme is commonly referred to as Karlin's occupancy scheme due to Karlin's article \cite{Karlin:1967}, which offered the first systematic investigation of the scheme. A survey of many results available for the scheme up to 2007 can be found in \cite{Gnedin+Hansen+Pitman:2007}.

In a {\it deterministic version} of Karlin's occupancy scheme the $n$th ball is thrown at time $n\in\mn$. For $j,n\in\mn$, denote by $\mathcal{K}_j(n)$ and $\mathcal{K}_j^\ast(n)$ the number of boxes containing at time $n$ at least $j$ balls and exactly $j$ balls, respectively. In particular, $\mathcal{K}_1(n)$ denotes the number of occupied boxes at time $n$.

Denote by $(S_k)_{k\in\mn}$ a standard random walk with independent increments having an exponential distribution of unit mean. Put $\pi(t):=\#\{k\in\mn: S_k\le t\}$ for $t\geq 0$, so that $\pi:=(\pi(t))_{t\ge 0}$ is a Poisson process on $[0,\infty)$ of unit intensity. It is assumed that $\pi$ is independent of random variables which represent the indices of boxes hit by the balls $1$, $2,\ldots$

In a {\it Poissonized version} of Karlin's occupancy scheme the $n$th ball is thrown at time $S_n$ for $n\in\mn$, so that there are $\pi(t)$ balls at time $t\geq 0$. For $j\in\mn$ and $t\ge 0$, denote by $K_j(t)$ and $K_j^*(t)$ the number of boxes at time $t$ in the Poissonized scheme containing at least $j$ balls and exactly $j$ balls, respectively.
In particular, $K_1(t)$ denotes the number of occupied boxes at time $t$. For $i\in\mn$, denote by $Y_i$ the index of a box that the $i$th ball falls into. Also, for $k\in\mn$ and $t\geq 0$, denote by $\pi_k(t)$ the number of balls which fall into the box $k$ in the Poissonized scheme at time $t$. Then $\pi_k(t)=\#\{i\le\pi(t):Y_i=k\}$ and $\pi(t)=\sum_{k\geq 1}\pi_k(t)$ for all $t\geq 0$. The reason behind introducing the {\it Poissonized version} is the thinning property of Poisson processes. Since the variables $Y_1$, $Y_2,\ldots$ are independent, and also independent of $\pi$, the property ensures that the processes $(\pi_1(t))_{t\geq 0}$, $(\pi_2(t))_{t\geq 0},\ldots$ are independent, and $(\pi_k(t))_{t\geq 0}$ is a Poisson process of intensity $p_k$. Thus, the numbers of balls which hit distinct boxes are independent, and
for each $j\in\mn$, the variable $K_j(t)$ is an infinite sum of independent indicators. This is in contrast with the deterministic version, in which the aforementioned independence is absent. This justifies a common approach which is exploited when analyzing the deterministic version. First, the scheme is Poissonized. Then the result in focus, or rather its counterpart, is proved for the Poissonized scheme. Finally, a transition, called de-Poissonization, is made from the Poissonized to the deterministic scheme. As far as the quantities introduced above are concerned, a relative simplicity of the approach is secured by the equalities $K_j(t)=\mathcal{K}_{j}(\pi(t))$ for $t\geq 0$ a.s. and $K_{j}(S_n)=\mathcal{K}_j(n)$ for $n\in\mn$ a.s. The latter equality is not as useful as the former, for the process $(K_j(t))_{t\geq 0}$ and the variable $S_n$ are dependent.

Put $$\rho(t):=\#\{k\in\mn: 1/p_k\le t\},\quad t>0,$$ so that $\rho$ is the counting function of the sequence $(1/p_k)_{k\in\mn}$. Observe that $\rho(t)=0$ for $t\in(0,1]$. Typically, assumptions on the distribution $(p_k)_{k\in\mn}$ are formulated in terms of $\rho$ rather than the distribution itself. Regular variation of $\rho$ at $\infty$ of index $\alpha\in [0,1]$ is a standard condition, sometimes referred to as Karlin's condition. This means that $\rho(t)\sim t^\alpha L(t)$ as $t\to\infty$ for some $L$ slowly varying at $\infty$. Comprehensive information about slowly and regularly varying functions can be found in Section 1 of \cite{Bingham+Goldie+Teugels:1989}. The Karlin scheme exhibits different behaviors for $\alpha=0$, $\alpha\in(0,1)$ and $\alpha=1$. In view of this the three cases are treated separately in Theorems \ref{thm:Karlin0}, \ref{thm:Karlin} and \ref{thm:Karlin1} (all the theorems will be proved by an application of Theorem \ref{thm:main}).

To be more precise, Theorem \ref{thm:Karlin0} is concerned with the situation in which $\rho$ belongs to a proper subclass of slowly varying functions. We say that $\rho$ belongs to de Haan’s class $\Pi$ with the auxiliary function $\ell$ (notation $\rho\in\Pi$ or $\rho\in\Pi_\ell$) if, for all $\lambda>0$,
\begin{equation}\label{eq:deHaan}
\lim_{t\to\infty}\frac{\rho(\lambda t)-\rho(t)}{\ell(t)}=\log \lambda.
\end{equation}
Detailed information about the class $\Pi$ can be found in Section 3 of \cite{Bingham+Goldie+Teugels:1989} and in \cite{Geluk+deHaan:1987}. In particular, according to Theorem 3.7.4 in \cite{Bingham+Goldie+Teugels:1989}, $\Pi$ is a proper subclass of the class of slowly varying functions. We note in passing that the definition of a general function belonging to the class $\Pi$ does not require the precise growth rate imposed on $\ell$, nor even divergence of $\ell$ to $\infty$. For instance, $t\mapsto \log t$ is a function of the class $\Pi$ with the auxiliary function $\ell\equiv 1$.
We shall use the notation $\rho\in \Pi_{\ell,\,\infty}$ if $\rho$ belongs to the de Haan class $\Pi$ with the auxiliary function $\ell$ satisfying $\lim_{t\to\infty}\ell(t)=\infty$.

\begin{thm}\label{thm:Karlin0}
Assume that $\rho\in\Pi_\ell$. If $\ell$ in \eqref{eq:deHaan} satisfies
\begin{equation}\label{eq:slowly}
\ell(t)~\sim~ (\log t)^\beta l(\log t),\quad t\to\infty
\end{equation}
for some $\beta>0$ and $l$ slowly varying at $\infty$, then, for each $j\in\mn$,
\begin{equation}\label{eq:LILkar}
\limsup_{t\to\infty}\frac{K_j(t)-\me K_j(t)}{({\rm Var}\,K_j(t)\log {\rm Var}\,K_j(t))^{1/2}}=\Big(\frac{2}{\beta}\Big)^{1/2}\quad\text{{\rm a.s.}}
\end{equation}
If $\ell$ in \eqref{eq:deHaan} satisfies
\begin{equation}\label{eq:slowly2}
\ell(t)~\sim~\exp(\sigma(\log t)^\lambda),\quad t\to\infty
\end{equation}
for some $\sigma>0$ and $\lambda\in (0,1)$, then, for each $j\in\mn$,
\begin{equation}\label{eq:LILkar1}
\limsup_{t\to\infty}\frac{K_j(t)-\me K_j(t)}{({\rm Var}\,K_j(t)\log\log {\rm Var}\,K_j(t))^{1/2}}=\Big(\frac{2}{\lambda}\Big)^{1/2}\quad\text{{\rm a.s.}}
\end{equation}
If $\rho\in \Pi_{\ell,\,\infty}$ (in particular, in both cases discussed above), then
\begin{equation*}
\me K_j(t)~\sim~ \rho(t),\quad t\to\infty
\end{equation*}
and
\begin{equation*}
{\rm Var}\,K_j(t)~\sim~ \Big(\log 2- \sum_{k=1}^{j-1}\frac{(2k-1)!}{(k!)^2 2^{2k}}\Big)\ell(t),\quad t\to\infty.
\end{equation*}
\end{thm}
\begin{rem}\label{rem:deHaan}

For the time being, working in the present setting with slowly varying $\rho\notin \Pi$ is a challenging task, for even the large-time asymptotics of $t\mapsto\var K_1(t)$ is not known.
\end{rem}

\begin{thm}\label{thm:Karlin}
Assume that, for some $\alpha\in (0,1)$ and some $L$ slowly varying at $+\infty$,
\begin{equation*}
\rho(t)~\sim~ t^\alpha L(t),\quad t\to\infty.
\end{equation*}
Then, for each $j\in\mn$,
\begin{equation}\label{eq:LILkar2}
\limsup_{t\to\infty}\frac{K_j(t)-\me K_j(t)}{({\rm Var}\,K_j(t)\log\log {\rm Var}\,K_j(t))^{1/2}}=2^{1/2}\quad\text{{\rm a.s.}},
\end{equation}
\begin{equation*}
\me K_j(t)~\sim~ \frac{\Gamma(j-\alpha)}{(j-1)!}t^\alpha L(t)\quad\text{and}\quad {\rm Var}\,K_j(t)~\sim~ c_j t^\alpha L(t),\quad t\to\infty,
\end{equation*}
where $\Gamma$ is the Euler gamma function and
\begin{equation*}
c_j:=\frac{1}{(j-1)!}\Big(\frac{1}{2^{j+1-\alpha}}\sum_{i=0}^{j-1}\frac{\Gamma(i+j-\alpha)}{i!2^i}- \Gamma(j-\alpha)\Big)>0.
\end{equation*}
\end{thm}

\begin{thm}\label{thm:Karlin1}
Assume that, for some $L$ slowly varying at $+\infty$,
\begin{equation*}\label{eqassump1}
\rho(t)~\sim~ tL(t),\quad t\to\infty.
\end{equation*}
Then, for each $j\geq 2$, relation \eqref{eq:LILkar2} holds,
\begin{equation*}
\me K_j(t)~\sim~ \frac{1}{j-1}tL(t)\quad \text{and}\quad {\rm Var}\,K_j(t)~\sim~ \frac{(2j-3)!}{((j-1)!)^2 2^{2j-3}}tL(t),\quad t\to\infty.
\end{equation*}

\noindent If $j=1$ and, for each small enough $\gamma>0$,
\begin{equation}\label{eq:exotic}
\lim_{n\to\infty}\frac{\hat L(\exp((n+1)^{1+\gamma}))}{\hat L(\exp(n^{1+\gamma}))}=0,
\end{equation}
where $\hat L(t):=\int_t^\infty y^{-1}L(y){\rm d}y$ is well-defined for large $t$, then relation \eqref{eq:LILkar2} holds with $j=1$. If \eqref{eq:exotic} does not hold, then
\begin{equation}\label{eq:LILkar200}
\limsup_{t\to\infty}\frac{K_1(t)-\me K_1(t)}{({\rm Var}\,K_1(t)\log\log {\rm Var}\,K_1(t))^{1/2}}\leq 2^{1/2}\quad\text{{\rm a.s.}}
\end{equation}
In any event $${\rm Var}\,K_1(t)~\sim~ \me K_1(t)~\sim~ t \hat L(t),\quad t\to\infty.$$
\end{thm}

Theorems \ref{thm:Karlin0}, \ref{thm:Karlin} and \ref{thm:Karlin1} will be proved in Section \ref{sec:Karlin} by an application of Theorem \ref{thm:main}.

\begin{rem}
The convergence of the integral defining $\hat L$ which is not obvious will be justified in the paragraph preceding Lemma \ref{lem:alone}.

Slowly varying $\hat L$ satisfying \eqref{eq:exotic} do exist, for instance, $\hat L(t)\sim \exp(-\log t/\log\log t)$ as $t\to\infty$. In general, such functions should exhibit a very fast decay. In particular, for slowly varying $\hat L$ of moderate decay like $\hat L(t)\sim (\log t)^{-\beta}$ for some $\beta>0$ or fast decay like $\widehat L(t)\sim \exp(-\sigma (\log t)^\lambda)$ for some $\sigma>0$ and $\lambda\in (0,1)$, relation \eqref{eq:exotic} fails to hold.

We do not claim that relation \eqref{eq:LILkar2}, with $j=1$, fails to hold whenever so does \eqref{eq:exotic}.  As explained in Remark \ref{rem:counterex} the first parts of conditions (B2.1) and (B2.2) of Theorem \ref{thm:main}, with the sets $R_\varsigma(t)$ and $R_0(t)$ chosen in a natural way, do not hold, so that one half of Theorem \ref{thm:main} stated as Proposition \ref{lilhalf2} below is not applicable in this case.
\end{rem}

Our next purpose is to obtain versions of Theorems \ref{thm:Karlin0}, \ref{thm:Karlin} and \ref{thm:Karlin1} for the {\it deterministic version} of Karlin's occupancy scheme.
It is known (see Lemma 1 in \cite{Gnedin+Hansen+Pitman:2007}) that the means $\me \mathcal{K}_j(n)$ and $\me K_j(n)$ satisfy
\begin{equation}\label{eq:means}
\lim_{n\to\infty} |\me\mathcal{K}_j(n)-\me K_j(n)|=0.
\end{equation}
The validity of this relation does not require any specific conditions like regular variation of the counting function $\rho$. However, we are not aware of a counterpart of this result for variances. To fill this gap, we obtain in Proposition \ref{equivvar} formula \eqref{eq:ratio} which is crucial for de-Poissonization. Although \eqref{eq:ratio} is not a precise analogue of \eqref{eq:means}, it serves our purposes.
In the case $j=1$ relation \eqref{eq:ratio} was proved in Lemma 4 of \cite{Gnedin+Hansen+Pitman:2007} under the sole assumption that either $\lim_{n\to\infty} \var \mathcal{K}_1(n)=\infty$ or $\lim_{n\to\infty} \var K_1(n)=\infty$. The case $j\geq 2$ of Proposition \ref{equivvar} seems to be new.

\begin{assertion}\label{equivvar}
Assume that either $\rho\in \Pi_{\ell,\,\infty}$ or $\rho$ is regularly varying at $\infty$ of index $\alpha\in (0,1]$. Then, for $j\in\mn$,
\begin{equation}\label{eq:ratio}
\lim_{n\to\infty} \frac{\var \mathcal{K}_j(n)}{\var K_j(n)}=1.
\end{equation}
\end{assertion}

Given next are LILs for $\mathcal{K}_j(n)$, $j\in\mn$ a deterministic scheme version of $K_j(t)$, $j\in\mn$. The results will be derived from Theorems \ref{thm:Karlin0}, \ref{thm:Karlin} and \ref{thm:Karlin1} with the help of a de-Poissonization technique.
\begin{thm}\label{thm:depoiss}
Under the assumptions of Theorems \ref{thm:Karlin0}, \ref{thm:Karlin} or \ref{thm:Karlin1}, for $j\in\mn$, limit relations \eqref{eq:LILkar}, \eqref{eq:LILkar1} and \eqref{eq:LILkar2} hold true with $\mathcal{K}_j(n)$, $\me \mathcal{K}_j(n)$ and $\var \mathcal{K}_j(n)$ replacing $K_j(t)$, $\me K_j(t)$ and $\var K_j(t)$, and $n\to\infty$ replacing $t\to\infty$.
\end{thm}

\section{Proofs of the main results}

\subsection{Auxiliary results}

For $k\in\mn$ and $t\geq 0$, put $X^\ast(t):=X(t)-\me X(t)$ and $\eta_k(t):=\1_{A_k(t)}-\mmp(A_k(t))$. Note that
\begin{equation}\label{eq:reprX}
X^\ast(t)=\sum_{k\geq 1} \eta_k(t),\quad t\geq 0
\end{equation}
and that $\eta_1(t)$, $\eta_2(t),\ldots$ are independent centered random variables.
\begin{lemma}\label{lem:ineq}
For any $\vartheta\in\mr$ and any $t\geq 0$,
\begin{equation}\label{eq:ss1}
\me \exp(\vartheta X^\ast(t))\leq \exp(2^{-1}\vartheta^2\exp(|\vartheta|)a(t))
\end{equation}
and
\begin{equation}\label{eq:ss100}
\me \exp(\vartheta X(t))\leq \exp((\exp(\vartheta)-1)b(t)).
\end{equation}
\end{lemma}
\begin{proof}
Using \eqref{eq:reprX}, $\me \eta_k(t)=0$ and $\eee^x\leq 1+x+(x^2/2)\eee^{|x|}$ for $x\in\mr$ we infer, for $\vartheta\in\mr$ and $t\geq 0$,
\begin{equation*}
\me \exp(\vartheta X^\ast(t))=\prod_{k\geq 1}\me \exp (\vartheta \eta_k(t))\leq \prod_{k\geq 1} \big(1+2^{-1}\vartheta^2 \me \eta_k^2(t) \exp(|\vartheta| |\eta_k(t)|)\big).
\end{equation*}
In view of $|\eta_k(t)|\leq 1$ a.s., we conclude that $\exp(|\vartheta| |\eta_k(t)|)\leq \exp(|\vartheta|)$. This in
combination with $\sum_{k\geq 1}\me \eta^2_k(t)=a(t)$ and $\eee^x\geq 1+x$ for $x\in\mr$ yields, for $\vartheta\in\mr$ and $t\geq 0$,
\begin{equation*}
\me \exp(\vartheta X^\ast(t))
\leq \prod_{k\geq 1}(1+2^{-1}\vartheta^2 \exp(|\vartheta|) \me \eta_k^2(t))\leq \exp(2^{-1}\vartheta^2\exp(|\vartheta|)a(t)).
\end{equation*}
Thus, inequality \eqref{eq:ss1} does indeed hold.

Inequality \eqref{eq:ss100} follows from
\begin{multline*}
\me \exp(\vartheta X(t))=\prod_{k\geq 1}\me \exp (\vartheta \1_{A_k(t)})= \prod_{k\geq 1}(1+(\eee^\vartheta-1)\mmp(A_k(t)))\leq \exp\Big((\eee^\vartheta-1)\sum_{k\geq 1}\mmp(A_k(t))\Big)\\=
\exp((\eee^\vartheta-1)b(t)).
\end{multline*}
\end{proof}

Corollary \ref{crude} is an immediate consequence of \eqref{eq:ss100}.
\begin{cor}\label{crude}
Suppose (A4). Then, for $\vartheta \in\mr$ and $t\ge s>0$,
$$\me \exp\big(\vartheta (X(t)-X(s))\big)\le \exp((\eee^{\vartheta}-1)(b(t)-b(s))).$$
\end{cor}
For each $B\geq 0$ and each $D>1$, put
\begin{equation}\label{eq:defg}
g_{1,\,B}(t):=(B+1)\log\log t,\quad t>\eee\quad\text{and}\quad g_D(t):=(D-1)\log t,\quad t>1.
\end{equation}

\begin{lemma}\label{lem:aux1}
Fix any $\varrho\in (0,1)$ and any $\kappa\in (0,1)$,
suppose (A1) and (A3) and let $q_\varrho$ and $\mu_\varrho$ be as defined in \eqref{eq:mutheta}.

\noindent (a) If $\mu$ in \eqref{eq:infim} is equal to $1$, then $\exp(-g_{1,\,q_\varrho}(a(t_n(\kappa, 1))))=O(n^{-(1-\kappa)})$ as $n\to\infty$, and if $\mu>1$, then $\exp(-g_{\mu_\varrho}(a(t_n(\kappa, \mu))))=O(n^{-(1-\kappa)})$.

\noindent (b) There exists an integer $r\geq 2$ such that $\big(\big((v_{n+1}(\kappa,\mu)-v_n(\kappa, \mu))/a(t_n)\big)^r\big)_{n\in\mn}$ is a summable sequence.
\end{lemma}
\begin{proof}

\noindent (a) By (A3) and the definition of $t_n$, for $\mu=1$,
\begin{equation}\label{eq:ineq10}
\exp(n^{(1-\kappa)/(q_\varrho+1)})\leq b(t_n(\kappa, 1))=O(a(t_n(\kappa, 1))f_q(a(t_n(\kappa, 1)))),\quad n\to\infty
\end{equation}
and, for $\mu>1$,
\begin{equation}\label{eq:ineq11}
n^{\mu_\varrho(1-\kappa)/(\mu_\varrho -1)}\leq b(t_n(\kappa, \mu))=O((a(t_n(\kappa, \mu)))^{\mu_\varrho}),\quad n\to\infty.
\end{equation}
Since $\lim_{t\to\infty}(\log f_q(t)/\log t)=0$ we infer $$\exp(-g_{1,\,q_\varrho}(a(t_n(\kappa, 1))))=(\log a(t_n(\kappa, 1)))^{-(q_\varrho+1)}=O(n^{-(1-\kappa)}),\quad n\to\infty.$$ Also, for $\mu>1$, $$\exp(-g_{\mu_\varrho}(a(t_n(\kappa, \mu))))=(a(t_n(\kappa, \mu)))^{-(\mu_\varrho-1)}=O(n^{-(1-\kappa)}),\quad n\to\infty.$$

\noindent (b) We first show that if (A3) holds with $\mu=1$, then
\begin{equation}\label{eq:ineqlog}
\log a(t_n(\kappa, 1))=O(n^{(1-\kappa)/(q_\varrho+1)}),\quad n\to\infty.
\end{equation}
Indeed, if $b$ is eventually continuous, then $b(t_n(\kappa, 1))=v_n(\kappa, 1)$ for large enough $n$, whence $\log a(t_n(\kappa, 1))\leq \log b(t_n(\kappa, 1))=n^{(1-\kappa)/(q_\varrho+1)}$ for large $n$. Assume now that ${\lim\inf}_{t\to\infty}(\log b(t-1)/\log b(t))>0$. This in combination with $\log b(t_n(\kappa, 1)-1) \leq n^{(1-\kappa)/(q_\varrho+1)}$ and $\log a(t_n(\kappa, 1))\leq \log b(t_n(\kappa, 1))$ yields \eqref{eq:ineqlog}.

Observe that, as $n\to\infty$,
\begin{multline*}
v_{n+1}(\kappa,1)-v_n(\kappa, 1)=\exp((n+1)^{(1-\kappa)/(q_\varrho+1)})-\exp(n^{(1-\kappa)/(q_\varrho+1)})\\~\sim~((1-\kappa)/(q_\varrho+1))n^{((1-\kappa)/(q_\varrho+1))-1}\exp(n^{(1-\kappa)/(q_\varrho+1)})
\end{multline*}
and, for $\mu>1$, $$v_{n+1}(\kappa,\mu)-v_n(\kappa, \mu)=(n+1)^{\mu_\varrho(1-\kappa)/(\mu_\varrho-1)}-n^{\mu_\varrho(1-\kappa)/(\mu_\varrho-1)}~\sim~ (\mu_\varrho(1-\kappa)/(\mu_\varrho-1))n^{(1-\mu_\varrho\kappa)/(\mu_\varrho-1)}.$$

Recall that $\lim_{t\to\infty}t^{-\rho_\ast}L_\ast(t)=0$ for any $\rho_\ast>0$ and any $L_\ast$ slowly varying at $\infty$. This implies that $f_q(t)=O((\log t)^{q_\varrho})$ as $t\to\infty$. Using the latter estimate together with \eqref{eq:ineq10} and \eqref{eq:ineqlog} we infer
\begin{multline}\label{eq:1a}
\frac{1}{a(t_n(\kappa, 1))}=O((\log a(t_n(\kappa, 1)))^{q_\varrho}\exp(-n^{(1-\kappa)/(q_\varrho+1)}))\\=O(n^{q_\varrho(1-\kappa)/(q_\varrho+1)}\exp(-n^{(1-\kappa)/(q_\varrho+1)})),\quad n\to\infty.
\end{multline}
If $\mu>1$, inequality \eqref{eq:ineq11} entails
\begin{equation}\label{eq:1a1}
\frac{1}{a(t_n(\kappa,\mu))}=O(n^{-(1-\kappa)/(\mu_\varrho-1)})),\quad n\to\infty.
\end{equation}
Summarizing, for $\mu\geq 1$, $$\frac{v_{n+1}(\kappa,\mu)-v_n(\kappa, \mu)}{a(t_n)}=O(n^{-\kappa}),\quad n\to\infty.$$ Now the claim of part (b) is justified by choosing any integer $r\geq 2$ satisfying $r\kappa>1$.
\end{proof}

\begin{lemma}\label{lem:equality}
Fix any $\gamma>0$ and suppose (A1) and (A3). If $\mu$ in \eqref{eq:infim} is equal to $1$, then $\exp(-g_{1,\,q}(a(\tau_n(\gamma, 1))))=O(n^{-(1+\gamma)})$ as $n\to\infty$, and if $\mu>1$, then $\exp(-g_\mu (a(\tau_n(\gamma,\mu))))=O(n^{-(1+\gamma)})$ as $n\to\infty$. Assuming additionally that (B1) holds, $$\exp(-g_{1,\,q}(a(\tau_n(\gamma, 1))))~\sim~n^{-(1+\gamma)}\quad \text{and}\quad \exp(-g_\mu(a(\tau_n(\gamma, \mu))))~\sim~ n^{-(1+\gamma)},\quad n\to\infty$$ in the cases $\mu=1$ and $\mu>1$, respectively.
\end{lemma}
\begin{proof}
The proof of the first part is similar to the proof of Lemma \ref{lem:aux1}(a). The proof of the second part is similar to the proof of \eqref{eq:ineqlog}. Hence, both are omitted.
\end{proof}

Lemma \ref{bill} can be found in Lemma 2 of \cite{Longnecker+Serfling:1977}.
\begin{lemma}\label{bill}
Let $\xi_1$, $\xi_2,\ldots$ be random variables. Fix any $N\in\mn$ and assume that
\begin{equation}\label{bill1}
\me |\xi_{l+1}+\ldots+ \xi_m|^{\lambda_1} \leq (u_{l+1}+\ldots+u_m)^{\lambda_2},\quad 0\leq l<m\leq N
\end{equation}
for some $\lambda_1>0$, some $\lambda_2>1$ and some nonnegative numbers $u_1,\ldots, u_N$. Then
\begin{equation*}\label{billconcl}
\me (\max_{1\leq m\leq N}\,|\xi_1+\ldots+\xi_m|)^{\lambda_1} \leq A_{\lambda_1,\,\lambda_2} (u_1+\ldots+u_N)^{\lambda_2}
\end{equation*}
for a positive constant $A_{\lambda_1,\,\lambda_2}$.
\end{lemma}

\subsection{Proofs of Propositions \ref{prop:clt} and \ref{main:exp}}\label{sec:CLT}

\begin{proof}[Proof of Proposition \ref{prop:clt}]
This follows by a standard application of the Lindeberg-Feller theorem. One of the conditions that has to be checked is
\begin{equation}\label{eq:lindeb}
\lim_{t\to\infty}\sum_{k\geq 1}\me \frac{(\eta_k(t))^2}{a(t)}\1_{\{|\eta_k(t)|>\varepsilon (a(t))^{1/2}\}}=0
\end{equation}
for all $\varepsilon>0$. As $|\eta_k(t)|=|\1_{A_k(t)}-\mmp(A_k(t))|\leq 1$ a.s.\ and $a$ diverges to infinity, the indicator \newline $1_{\{|\eta_k(t)|>\varepsilon (a(t))^{1/2}\}}$ is equal to $0$ for large $t$, which entails \eqref{eq:lindeb}.
\end{proof}
\begin{proof}[Proof of Proposition \ref{main:exp}]
Fix any $\theta \in\mr$. By Proposition \ref{prop:clt} together with Fatou's lemma,
$$
\liminf_{t\to\infty} \me\exp\Big(\frac{\theta \, X^\ast(t)}{(a(t))^{1/2}}\Big) \geq \me \exp(\theta \, {\rm Normal}(0,1)).
$$
Hence, we are left with proving that
\begin{equation}\label{eq:limsup}
\limsup_{t\to\infty} \me \exp\Big(\frac{\theta \, X^\ast(t)}{(a(t))^{1/2}}\Big)\le \me \exp(\theta \, {\rm Normal}\,(0,1))=\exp(\theta^2/2).
\end{equation}
Using inequality \eqref{eq:ss1} with $\vartheta=\theta/(a(t))^{1/2}$ we obtain
\begin{equation*}
\me \exp\Big(\frac{\theta \, X^\ast(t)}{(a(t))^{1/2}}\Big)\leq \exp\Big(\frac{\theta^2}{2}\exp\Big(\frac{|\theta|}{(a(t))^{1/2}}\Big)\Big).
\end{equation*}
Recalling (A1) and letting $t$ tend to $\infty$ yields \eqref{eq:limsup}.

Relation \eqref{eq:sec_statement} is a consequence of uniform integrability of $(|X^\ast(t)|/(a(t))^{1/2}))_{t\geq s_1}$ for some $s_1
>0$, which is implied by \eqref{eq:limsup}, and Proposition \ref{prop:clt}.
\end{proof}

\subsection{Proof of Theorem \ref{thm:main}}\label{sec:main}

Theorem \ref{thm:main} follows from the two propositions given next.
\begin{assertion}\label{lilhalf1}
Suppose (A1)-(A5). Then, with $\mu\geq 1$ and $q\geq 0$ as defined in \eqref{eq:infim} and \eqref{eq:infim2}, respectively, 	
\begin{equation*}\label{lil12}
\limsup_{t\to\infty}\frac{X^\ast(t)}{(2(q+1)a(t) \log\log
a(t))^{1/2}}\leq 1\quad\text{and}\quad \limsup_{t\to\infty}\frac{X^\ast(t)}{(2(\mu-1)a(t)\log a(t))^{1/2}}\leq 1\quad \text{{\rm a.s.}}
\end{equation*}
in the cases $\mu=1$ and $\mu>1$, respectively.
\end{assertion}

\begin{assertion}\label{lilhalf2}
Suppose (A1), (A3), (B1) and either (B2.1) or (B2.2). Then, with $\mu\geq 1$ and $q\geq 0$ as defined in \eqref{eq:infim} and \eqref{eq:infim2}, respectively,
\begin{equation}\label{lil11}
\limsup_{t\to\infty}\frac{X^\ast(t)}{(2(q+1)a(t) \log\log
a(t))^{1/2}}\geq 1\quad\text{and}\quad \limsup_{t\to\infty}\frac{X^\ast(t)}{(2(\mu-1)a(t)\log a(t))^{1/2}}\geq 1\quad \text{{\rm a.s.}}
\end{equation}
in the cases $\mu=1$ and $\mu>1$, respectively.
\end{assertion}
\begin{proof}[Proof of Proposition \ref{lilhalf1}]
In view of (A2) we can and do assume that $a$ is eventually nondecreasing. It suffices to show that, for each $\varrho\in (0,1)$ and each positive $\kappa$ sufficiently close to $0$,
\begin{equation}\label{princip}
\limsup_{n\to\infty}\frac{\sup_{u\in [t_n,\,t_{n+1}]}\,X^\ast(u)}{(2 a(t_n) h_\varrho(a(t_n)))^{1/2}}\leq 1+\kappa\quad\text{a.s.},
\end{equation}
where $t_n=t_n(\kappa,\mu)$ is as defined in \eqref{tn}, $h_\varrho=g_{1,\,q_\varrho}$ if $\mu$ in \eqref{eq:infim} is equal to $1$ and $h_\varrho=g_{\mu_\varrho}$ if $\mu>1$ (see \eqref{eq:defg} for the definitions of $g_{1,\,q_\varrho}$ and $g_{\mu_\varrho}$). Indeed, if this is true, then, for large enough $n$, $$\frac{X^\ast(t)}{(2 a(t) h_\varrho (a(t)))^{1/2}}\leq \frac{\sup_{u\in [t_n,\,t_{n+1}]}\,X^\ast(u)}{(2 a(t_n)h_\varrho (a(t_n)))^{1/2}}\leq 1+\kappa\quad \text{a.s.},$$ whenever $t\in [t_n,\,t_{n+1}]$ (the first inequality is the only place of our proof in which condition (A2) is used).

Relation \eqref{princip} will be obtained in two steps. First, we shall prove in Lemma \ref{inter1} that
\begin{equation}\label{princip1}
{\limsup}_{n\to\infty}\frac{X^\ast(t_n)}{(2 a(t_n)h_\varrho (a(t_n)))^{1/2}}\leq 1+\kappa\quad \text{{\rm a.s.}}
\end{equation}
Second, we shall show that
\begin{equation}\label{princip2}
	\lim_{n\to\infty}\frac{\sup_{u\in [t_n,\, t_{n+1}]}|X^\ast(u)-X^\ast(t_n)|}{(a(t_n)h_\varrho(a(t_n)))^{1/2}}=0\quad\text{a.s.}
\end{equation}

\begin{lemma}\label{inter1}
Suppose (A1) and (A3). Then relation \eqref{princip1} holds for any $\kappa\in (0, (\sqrt{5}-1)/2)$.
\end{lemma}
\begin{proof}
Fix any $\kappa\in (0,(\sqrt{5}-1)/2)$.
We claim that there exists a positive $\rho=\rho(\kappa)$ satisfying
\begin{equation}\label{eq:choiceofrho}
(1-\kappa)(1+\kappa)^2(2-\exp(2(1+\kappa)\rho))>1.
\end{equation}
To prove this, observe that $(1-\kappa)(1+\kappa)^2>1$. Choosing positive $\rho$ sufficiently close to $0$, we can make $2-\exp(2(1+\kappa)\rho)$ as close to $1$ as we wish and particularly ensure that $2-\exp(2(1+\kappa)\rho)>(1-\kappa)^{-1}(1+\kappa)^{-2}$.

Fix any $\theta\in\mr$. Inequality \eqref{eq:ss1} with $\vartheta=\theta/(2a(t)h_\varrho(a(t)))^{1/2}$ reads
\begin{equation*}
\me \exp\Big(\frac{\theta X^\ast(t)}{(2a(t)h_\varrho(a(t)))^{1/2}}\Big)\leq \exp\Big(\frac{\theta^2}{4h_\varrho(a(t))}\exp\Big(\frac{|\theta|}{(2a(t)h_\varrho(a(t)))^{1/2}}\Big)\Big).
\end{equation*}
Since $h_\varrho(a(t))/a(t)\leq 2 \rho^2$ for large enough $t$ and $\rho$ as in \eqref{eq:choiceofrho}, we infer, for such $t$,
\begin{equation}\label{eq:7}
\me \exp\Big(\frac{\theta X^\ast(t)}{(2a(t)h_\varrho(a(t)))^{1/2}}\Big)\leq \exp\Big(\frac{\theta^2}{4h_\varrho(a(t))}\exp\Big(\frac{\rho |\theta|}{h_\varrho(a(t))}\Big)\Big).
\end{equation}
By Markov's inequality, for large $n$,
\begin{multline*}
\mmp\Big\{\frac{X^\ast(t_n)}{(2a(t_n)h_\varrho(a(t_n)))^{1/2}}>1+\kappa\Big\}\leq \eee^{-(1+\kappa)\theta} \me \exp\Big(\theta \frac{X^\ast(t_n)}{(2a(t_n)h_\varrho(a(t_n)))^{1/2}}\Big)\\\leq \exp\Big(-(1+\kappa)\theta+\frac{\theta ^2}{4h_\varrho(a(t_n))}\exp\Big(\frac{\rho |\theta|}{h_\varrho(a(t_n))}\Big)\Big).
\end{multline*}
Putting $\theta=2(1+\kappa)h_\varrho(a(t_n))$ we obtain
\begin{multline*}
\mmp\Big\{\frac{X^\ast(t_n)}{(2a(t_n)h_\varrho(a(t_n)))^{1/2}}>1+\kappa\Big\}\leq \exp(-(1+\kappa)^2(2-\exp(2(1+\kappa)\rho))h_\varrho(a(t_n)))\\=O\Big(\frac{1}{n^{(1-\kappa)(1+\kappa)^2(2-\exp(2(1+\kappa)\rho))}}\Big).
\end{multline*}
Here, we have used $\exp(-h_\varrho(a(t_n)))=O(n^{-(1-\kappa)})$ as $n\to\infty$, see Lemma \ref{lem:aux1}(a).

Hence, by \eqref{eq:choiceofrho}, $$\sum_{n\geq n_1}\mmp\Big\{\frac{X^\ast(t_n)}{(2a(t_n)h_\varrho(a(t_n)))^{1/2}}>1+\kappa\Big\}<\infty$$ for some $n_1\in\mn$ large enough, and an appeal to the Borel-Cantelli lemma completes the proof of Lemma \ref{inter1}.
\end{proof}

Next, we shall prove a result which is stronger than \eqref{princip2}:
\begin{equation}\label{eq:3}
	\lim_{n\to\infty}\frac{\sup_{u\in [t_n,\, t_{n+1}]}|X^\ast(u)-X^\ast(t_n)|}{(a(t_n))^{1/2}}=0\quad\text{a.s.}
\end{equation}

\noindent {\sc Proof of \eqref{eq:3}.} Using a partition $t_n=t_{0,\,n}<\ldots<t_{j,\,n}=t_{n+1}$ defined in (A5) we obtain
\begin{multline*}
\sup_{u\in [t_n,\, t_{n+1}]}|X^\ast(u)-X^\ast(t_n)|\\=\max_{1\leq k\leq j}\sup_{v\in
[0,\,t_{k,\,n}-t_{k-1,\,n}]}|(X^\ast(t_{k-1,\,n})-X^\ast(t_n))+(X^\ast(t_{k-1,\,n}+v)-X^\ast(t_{k-1,\,n}))|\\\leq \max_{1\leq k\leq j}|X^\ast(t_{k-1,\,n})-X^\ast(t_n)|\\+\max_{1\leq k\leq j}\sup_{v\in
[0,\,t_{k,\,n}-t_{k-1,\,n}]}|(X^\ast(t_{k-1,\,n}+v)-X^\ast(t_{k-1,\,n}))|\quad\text{a.s.}
\end{multline*}
We first show that
\begin{equation}\label{eq:315}
\lim_{n\to\infty}\frac{\max_{1\leq k\leq j}|X^\ast(t_{k-1,\,n})-X^\ast(t_n)|}{(a(t_n))^{1/2}}=0\quad\text{a.s.}
\end{equation}
To this end, we need two additional auxiliary results.
\begin{lemma}\label{ineq1}
Suppose (A4). Let $r\in\mn$ and $t,s\geq 0$. Then
\begin{equation}\label{eq:5}
\me (X^\ast(t)-X^\ast(s))^{2r}\leq C_r \max(|b(t)- b(s)|^r, |b(t)-b(s)|)
\end{equation}
for a positive constant $C_r$ which does not depend on $t$ and $s$.
\end{lemma}
\begin{proof}
In view of the representation
\begin{equation*}
X^\ast(t)-X^\ast(s)=\sum_{k\geq 1}(\1_{A_k(t)}-\mmp(A_k(t))-\1_{A_k(s)}+\mmp(A_k(s)))=:\sum_{k\geq 1}\eta_k(s,t),
\end{equation*}
the variable $X^\ast(t)-X^\ast(s)$ is an infinite sum of independent centered random variables with finite $(2r)$th moment. If $r\geq 2$, then, by Rosenthal's inequality (Theorem 3 in \cite{Rosenthal:1970}),
$$\me (X^\ast(t)-X^\ast(s))^{2r}\leq C_r \max\Big(\Big(\sum_{k\geq 1}\me (\eta_k(s,t))^2\Big)^r, \sum_{k\geq 1}\me (\eta_k(s,t))^{2r}\Big).$$ If $r=1$, then $$\me (X^\ast(t)-X^\ast(s))^2= \sum_{k\geq 1}\me (\eta_k(s,t))^2,$$ so that the preceding inequality trivially holds with $C_1=1$.
For any event $A$,
\begin{equation*}
\me (\1_A-\mmp(A))^{2r}=(1-\mmp(A))^{2r}\mmp(A)+(\mmp(A))^{2r}(1-\mmp(A))\leq (1-\mmp(A))\mmp(A)\leq \mmp(A).
\end{equation*}
Since, under (A4), $\eta_k(t,s)=\1_{A_k(t\vee s)\backslash A_k(t\wedge s)}-\mmp(A_k(t\vee s)\backslash A_k(t\wedge s))$ we conclude that in this case, for $k\in\mn$, $$\me (\eta_k(t,s))^{2r}\leq \mmp(A_k(t\vee s)\backslash A_k(t\wedge s))=|\mmp(A_k(t))-\mmp(A_k(s))|.$$ This proves \eqref{eq:5}.
\end{proof}

\begin{lemma}\label{billappl}
Suppose (A4) and (A5). Then, for any integer $r\geq 2$, there exists a positive constant $A_r$ such that
\begin{equation}\label{ineq4}
\me (\max_{1\leq k\leq j}\,|X^\ast(t_{k-1,\,n})-X^\ast(t_n)|)^{2r}\leq A_r (v_{n+1}(\kappa,\mu)-v_n(\kappa,\mu))^r.
\end{equation}
Here, $j$ and $(t_{k,\,n})_{0\leq k\leq j}$ are as defined in (A5), and $v_n(\kappa,\mu)$ is as defined in \eqref{tn}.
\end{lemma}
\begin{proof}
We intend to apply Lemma \ref{bill} with $\lambda_1=2r$, $\lambda_2=r$, $N=j-1$ and $\xi_k:=X^\ast(t_{k,\,n})-X^\ast(t_{k-1,\,n})$ for $k\in\mn$. Let $0\leq l<m\leq j-1$. Recall that (A5) ensures that $b(t_{m,\,n})-b(t_{l,\,n})=\sum_{k=l+1}^m (b(t_{k,\,n})-b(t_{k-1,\,n}))\geq 1$. Using this and Lemma \ref{ineq1} we obtain
\begin{multline}
\me |\xi_{l+1}+\ldots+\xi_m|^{2r}=\me(X^\ast(t_{m,\,n})-X^\ast(t_{l,\,n}))^{2r}\leq C_r \max((b(t_{m,\,n})-b(t_{l,\,n}))^r, b(t_{m,\,n})-b(t_{l,\,n}))\\=C_r (b(t_{m,\,n})-b(t_{l,\,n}))^r=\Big(\sum_{k=l+1}^m u_k\Big)^r,\label{ineq3}
\end{multline}
that is, inequality \eqref{bill1} holds with $u_k:=C_r^{1/r}(b(t_{k,\,n})-b(t_{k-1,\,n}))$, $k\in\mn$, $k\leq j-1$, where $C_r$ is the constant defined in Lemma \ref{ineq1}. Now \eqref{ineq4} is an immediate consequence of Lemma \ref{bill} and the definition of $t_n$:
\begin{multline*}
\me (\max_{1\leq k\leq j}\,|X^\ast(t_{k-1,\,n})-X^\ast(t_n)|)^{2r}=\me (\max_{1\leq k\leq j-1}|\xi_1+\ldots+\xi_k|)^{2r}\leq A_{2r,\,r}\Big(\sum_{k=1}^{j-1} u_k\Big)^r\\=A_{2r,\,r} C_r (b(t_{j-1,\,n})-b(t_n))^r\leq A_{2r,\,r} C_r (v_{n+1}(\kappa,\mu)-v_n(\kappa,\mu))^r.
\end{multline*}
\end{proof}
\begin{rem}\label{rem:laiwei}
We have learned the idea of using Lemma \ref{bill} in the present context from the proof of Theorem 4 (i) in \cite{Lai+Wei:1982}. The cited result is concerned with an upper half of a LIL for a rather general discrete-time model and as such cannot be used to prove \eqref{eq:3}.
\end{rem}

An application of Markov's inequality yields, for any $r>0$ and all $\varepsilon>0$,
\begin{multline*}
\mmp\Big\{\max_{1\leq k\leq j}|X^\ast(t_{k-1,\,n})-X^\ast(t_n)|>\varepsilon (a(t_n))^{1/2}\Big\}\leq \frac{\me (\max_{1\leq k\leq j}\,|X^\ast(t_{k-1,\,n})-X^\ast(t_n)|)^{2r}}{\varepsilon^{2r}(a(t_n))^r}.
\end{multline*}
In view of Lemma \ref{billappl} and Lemma \ref{lem:aux1}(b), there exists an integer $r\geq 2$ such that the right-hand side forms a summable sequence. This in combination with the Borel-Cantelli lemma ensures \eqref{eq:315}.

To complete the proof of \eqref{eq:3} it is enough to show that
$$\lim_{n\to\infty}\frac{\max_{1\leq k\leq j}\sup_{v\in [0,\, t_{k,\,n}-t_{k-1,\,n}]}|X^\ast(t_{k-1,\,n}+v)-X^\ast(t_{k-1,\,n})|}{(a(t_n))^{1/2}}=0\quad \text{a.s.}$$ Condition (A4) ensures that the process $(X(t))_{t\geq 0}$ is a.s.\ nondecreasing, and the function $b$ is nondecreasing.
Hence,
\begin{multline*}
\sup_{v\in [0,\,t_{k,\,n}-t_{k-1,\,n}]}|X^\ast(t_{k-1,\,n}+v)-X^\ast(t_{k-1,\,n})|\leq \sup_{v\in [0,\,t_{k,\,n}-t_{k-1,\,n}]}(X(t_{k-1,\,n}+v)-X(t_{k-1,\,n}))\\+\sup_{v\in [0,\,t_{k,\,n}-t_{k-1,\,n}]}(b(t_{k-1,\,n}+v)-b(t_{k-1,\,n}))\leq X(t_{k,\,n})-X(t_{k-1,\,n})+b(t_{k,\,n})-b(t_{k-1,\,n})\quad\text{a.s.}
\end{multline*}
According to (A1) and (A5)
\begin{equation*}
\frac{\max_{1\leq k\leq j}\,(b(t_{k,\,n})-b(t_{k-1,\,n}))}{(a(t_n))^{1/2}}\leq \frac{A}{(a(t_n))^{1/2}}~\to 0,\quad n\to\infty.
\end{equation*}
Finally, for all $\varepsilon>0$,
\begin{multline*}
\mmp\big\{\max_{1\leq k\leq j}(X(t_{k,\,n})-X(t_{k-1,\,n}))>\varepsilon (a(t_n))^{1/2}\big\}\le \sum_{k=1}^{j} \mmp\big\{X(t_{k,\,n})-X(t_{k-1,\,n})>\varepsilon (a(t_n))^{1/2}\big\}\\
\le \eee^{-\varepsilon (a(t_n))^{1/2}} \sum_{k=1}^{j}  \me \eee^{X(t_{k,\,n})-X(t_{k-1,\,n})}\le \eee^{-\varepsilon (a(t_n))^{1/2}} \sum_{k=1}^{j} \exp((\eee-1)(b(t_{k,\,n})-b(t_{k-1,\,n})))\\\leq \exp(A(\eee-1))j \eee^{-\varepsilon (a(t_n))^{1/2}}.
\end{multline*}
We have used Markov's inequality for the second inequality, Corollary \ref{crude} for the third and (A5) for the fourth. The right-hand side is summable in $n$ by another appeal to (A5) and thereupon
$$\lim_{n\to\infty}\frac{\max_{1\leq k\leq j}(X(t_{k,\,n})-X(t_{k-1,\,n}))}{(a(t_n))^{1/2}}=0\quad\text{a.s.}$$
by the Borel-Cantelli lemma. This completes the proof of \eqref{eq:3} and that of Proposition \ref{lilhalf1}.
\end{proof}

To unify the subsequent proofs for the cases $\mu=1$ and $\mu>1$, we introduce an additional notation. Put $h_0:=g_{1,\,q}$ if $\mu$ defined by \eqref{eq:infim} is equal to $1$ (here, $q$ is as defined in \eqref{eq:infim2}) and $h_0:=g_\mu$ if $\mu>1$.
\begin{lemma}\label{lem:l1}
Suppose (A1) and (A3) and let $\mu\geq 1$ be as given in \eqref{eq:infim}. For large $t>0$ and some $\varsigma\in [0,1)$, let $R_\varsigma(t)$ denote a set of positive integers satisfying \eqref{eq:onevar} if $\varsigma\in (0,1)$ and \eqref{eq:one} and $R_\varsigma(t)\neq \mn$ if $\varsigma=0$. Then
\begin{equation*}\label{eq:inter}
{\lim\inf}_{n\to \infty} \frac 1{(2a(\tau_n)h_0(a(\tau_n)))^{1/2}} \sum_{k\in (R_\varsigma(\tau_n))^c} \eta_k(\tau_n)\geq -(\varsigma(1-\gamma^2))^{1/2} \quad\text{{\rm a.s.}},
\end{equation*}
where $\tau_n=\tau_n(\gamma,\mu)$, with any $\gamma\in (0, (\sqrt{5}-1)/2)$, is as defined in \eqref{eq:tau}, and $(R_\varsigma(\tau_n))^c:=\mn\backslash R_\varsigma(\tau_n)$ denotes the complement of $R_\varsigma(\tau_n)$ in $\mn$.
\end{lemma}
\begin{proof}
Fix any $\gamma\in (0, (\sqrt{5}-1)/2)$. Arguing in the same way as we did when justifying \eqref{eq:choiceofrho}, we prove that there exists a positive $\rho=\rho(\gamma)$ satisfying
\begin{equation}\label{eq:choiceofrho2}
(1-\gamma^2)(1+\gamma)(2-\exp(2(1-\gamma^2)^{1/2}\rho))>1.
\end{equation}
For large $t$, put $a^\ast_\varsigma(t):={\rm Var}\,\Big(\sum_{k\in (R_\varsigma(t))^c}\1_{A_k(t)}\Big)=\sum_{k\in (R_\varsigma(t))^c} \me (\eta_k(t))^2$ and then $$Z_{1,\,\varsigma}(t):=\frac 1{(2a^\ast_\varsigma(t)h_0(a(t)))^{1/2}} \sum_{k\in (R_\varsigma(t))^c} \eta_k(t).$$ In view of \eqref{eq:onevar} or \eqref{eq:one} we have to show that
\begin{equation*}
\limsup_{n\to \infty}\, (-Z_{1,\,\varsigma}(\tau_n))\leq (1-\gamma^2)^{1/2} \quad\text{{\rm a.s.}}
\end{equation*}
According to the Borel-Cantelli lemma, the latter is ensured by
\begin{equation}\label{eq:w1}
\sum_{n\ge n_1} \P\{-Z_{1,\,\varsigma}(\tau_n)>(1-\gamma^2)^{1/2}\} <\infty
\end{equation}
for some $n_1$ large enough. To this end, we obtain a counterpart of \eqref{eq:7}, for $u\in\mr$ and $t$ large enough to ensure $h_0(a(t))/a^\ast_\varsigma(t)\le 2\rho^2$, where $\rho$ is as in \eqref{eq:choiceofrho2},
\begin{equation*}\label{eq:eZ}
\E {\rm e}^{u (-Z_{1,\,\varsigma}(t))} \le \exp \Big( \frac{u^2}{4h_0(a(t))}\exp \Big( \frac{\rho|u|}{h_0(a(t))} \Big)\Big).
\end{equation*}
Invoking the Markov inequality with $u =2(1-\gamma^2)^{1/2} h_0(a(\tau_n))$ and the first part of Lemma \ref{lem:equality} yields, for large $n$,
\begin{multline*}
\P\{-Z_{1,\,\varsigma}(\tau_n)>(1-\gamma^2)^{1/2}\} \le {\rm e}^{-u(1-\gamma^2)^{1/2}} \E {\rm e}^{u(-Z_{1,\,\varsigma}(\tau_n))}\\ \le \exp\big(-(1 - \gamma^2)(2-\exp(2(1-\gamma^2)^{1/2}\rho))h_0(a(\tau_n))\big)=O\Big(\frac{1}{n^{(1-\gamma^2)(1+\gamma)(2-\exp(2(1-\gamma^2)^{1/2}\rho))}}\Big).
\end{multline*}
In view of \eqref{eq:choiceofrho2} this entails \eqref{eq:w1}.
\end{proof}

\begin{lemma}\label{lem:new}
Suppose (A1), (A3), (B1) and either (B2.1) or (B2.2)
and let $\mu\geq 1$ be as given in \eqref{eq:infim}. For sufficiently small $\delta >0$, pick $\gamma\in (0, (\sqrt{5}-1)/2) >0$ satisfying $(1+\gamma)(1-\delta^2/8)<1$. Then
$$\limsup_{n\to \infty} \frac 1{(2a(\tau_n)h_0(a(\tau_n)))^{1/2}} \sum_{k\ge 1} \eta_k(\tau_n) \ge 1-\delta \quad\text{{\rm a.s.}},$$ where $\tau_n=\tau_n(\gamma,\mu)$
is as defined in \eqref{eq:tau}.
\end{lemma}
\begin{proof}
We intend to show that
\begin{equation}\label{eq:relat}
\limsup_{n\to \infty} \frac 1{(2a(\tau_n)h_0(a(\tau_n)))^{1/2}} \sum_{k\ge 1} \eta_k(\tau_n) \ge -(\varsigma (1-\gamma^2))^{1/2}+(1-\varsigma)^{1/2}
(1-\delta)\quad\text{{\rm a.s.}}
\end{equation}
with either any $\varsigma>0$ close to $0$ which appears in condition (B2.1) or $\varsigma=0$ if condition (B2.2)
holds. In the former case the claim of the lemma follows on sending $\varsigma$ to $0+$.

For large $t$, put $a_\varsigma(t):={\rm Var}\,\Big(\sum_{k\in R_\varsigma(t)}\eta_k(t)\Big)$. In view of Lemma \ref{lem:l1} and $$\limsup_{n\to\infty}\,(x_n+y_n)\geq \liminf_{n\to\infty}\, x_n+\limsup_{n\to\infty}\, y_n$$ which holds true for any numeric sequences $(x_n)$ and $(y_n)$, relation \eqref{eq:relat} follows if we can check that
\begin{equation}\label{eq:ss9}
\limsup_{n\to\infty}\, Z_{2,\,\varsigma}(\tau_n) \ge 1-\delta\quad\text{a.s.},
\end{equation}
where, for large $t$, $$Z_{2,\,\varsigma}(t)= \frac 1{(2a_\varsigma(t)h_0(a(t)))^{1/2}}  \sum_{k\in R_\varsigma(t)}\eta_k(t).$$ We shall prove that, for large $t$,
\begin{equation}\label{eq:r3}
\P\{Z_{2,\,\varsigma}(t)>1-\delta \}\ge  3^{-1} \eee^{-(1-\delta^2/8) h_0(a(t))}.
\end{equation}
By Lemma \ref{lem:equality}, conditions (A1) and (B1) ensure that $\eee^{-(1-\delta^2/8) h_0(a(\tau_n))}\sim n^{-(1+\gamma)(1-\delta^2/8)}$ as $n\to\infty$.
As a consequence, $$\sum_{n\ge n_1}\P\{Z_{2,\,\varsigma}(\tau_n)>1-\delta\}= \infty$$ for some $n_1 \ge n_0$ large enough. In view of the first part of (B2.1) or (B2.2), the random variables $Z_{2,\,\varsigma}(\tau_{n_0})$, $Z_{2,\,\varsigma}(\tau_{n_0+1}),\ldots$ are independent. Hence, divergence of the series entails \eqref{eq:ss9} by the converse part of the Borel-Cantelli lemma.
		
As a preparation for the proof of \eqref{eq:r3}, consider the event
$$U^{(\varsigma)}_t: = \big\{ 1-\delta < Z_{2,\,\varsigma}(t) \le 1 \big\} = \big\{(1-\delta)g(t)<  W_\varsigma(t)/(a_\varsigma(t))^{1/2} \le g(t)\big\},$$
where $g(t): =\big(2h_0(a(t)) \big)^{1/2}$ and $$W_\varsigma(t):= \sum_{k\in R_\varsigma(t)} \eta_k(t).
$$
Recall that $s_0$ is a positive number appearing in the definition of $R_\varsigma(t)$, see assumption (B2.1). Given $u\in\mr$ and large $t$ (that is, $t>s_0$ such that $a(t)>0$ and $a_\varsigma(t)>0$) introduce a new probability measure $\Q_{t,u}$ on $(\Omega,\mathcal{F})$ by the equality
$$\Q_{t,u}(A) = \frac{\E \left({\rm e}^{u  W_\varsigma (t)/ (a_\varsigma(t))^{1/2}}\1_{A}\right)}{\E {\rm e}^{u  W_\varsigma(t)/(a_\varsigma(t))^{1/2}}}=
\frac{\int_{A}{\rm e}^{u  W_\varsigma(t)/(a_\varsigma(t))^{1/2}}{\rm d}\mmp}{\int_{\Omega}{\rm e}^{u W_\varsigma(t)/(a_\varsigma(t))^{1/2}}{\rm d}\mmp},\quad A\in\mathcal{F}.
$$
Then, for $u\ge 0$,
\begin{multline}\label{eq:ss10}
\left(\E {\rm e}^{u  (W_\varsigma(t) /(a_\varsigma(t))^{1/2} -g(t))}\right)\Q_{t,u}(U^{(\varsigma)}_t)={\rm e}^{-ug(t)}\int_{U^{(\varsigma)}_t}{\rm e}^{uW_\varsigma(t)/(a_\varsigma(t))^{1/2}}{\rm d}\mmp\\\leq {\rm e}^{-ug(t)}\int_{U^{(\varsigma)}_t}{\rm e}^{ug(t)}{\rm d}\mmp=\mmp(U^{(\varsigma)}_t)\leq \P\{Z_{2,\,\varsigma}(t)>1-\delta\}.
\end{multline}
We shall show that, upon choosing an appropriate positive $u=u(t)=O((h_0(a(t)))^{1/2})$, the expectation on the left-hand side is bounded by ${\rm e}^{-(1-\delta^2/8)h_0(a(t))}$ from below and also prove that $\Q_{t,u(t)}(U^{(\varsigma)}_t)\geq 1/3$, thereby deriving~\eqref{eq:r3}.
		
First, we check that, with $u(t)=O((h_0(a(t)))^{1/2})$, $u(t)\in\mr$ (thus, for the time being, we do not assume that $u(t)$ takes positive values),
\begin{equation}\label{eq:p1}
\E  {\rm e}^{u(t)  W_\varsigma(t)/(a_\varsigma(t))^{1/2}} = {\rm e}^{(u(t))^2/2 + (u(t))^2 h(t)}, \quad t\to \infty
\end{equation}
for some function $h$ satisfying $\lim_{t\to \infty} h(t) = 0$. To this end, put $$\xi_k(t):= \frac{\eta_k(t)}{(a_\varsigma(t))^{1/2}}$$
and recall that $\eta_k(t):=\1_{A_k(t)}-\mmp(A_k(t))$. Since $|\eta_k(t)|\le 1$ a.s., $$|u(t) \xi_k(t)|=O\Big(\Big(\frac{h_0(a(t))}{a_\varsigma(t)}\Big)^{1/2}\Big)= o(1), \quad t\to \infty\quad\text{a.s.}$$ uniformly in $k$. Recalling the asymptotic expansions
$${\rm e}^x = 1 + x + x^2/2 + o(x^2)\quad\text{and}\quad\log(1+x)=x+O(x^2),\quad x\to 0,$$
we infer
\begin{multline}\label{eq:ss3}
\E \eee^{u(t) W_\varsigma(t)/(a_\varsigma(t))^{1/2}} = \prod_{k\in R_\varsigma(t)} \E \exp( u(t) \xi_k(t))\\ = \prod_{k\in R_\varsigma(t)} \E \big(1+u(t)\xi_k(t)+ (u(t))^2 (\xi_k(t))^2\big(1/2 + o(1)\big)\big)\\= \exp \Big(\sum_{k\in R_\varsigma(t)} \log \big(1+u(t)^2\E (\xi_{k}(t))^2 \big(1/2 + o(1)\big) \big)\Big)\\ = \exp \Big(  (u(t))^2\big(1/2 + o(1)\big)\sum_{k\in R_\varsigma(t)} \E (\xi_k(t))^2+(u(t))^4 O \Big( \sum_{k\in R_\varsigma(t)} (\E (\xi_k(t))^2 )^2\Big)\Big).
\end{multline}
Since, for large $t$,
\begin{equation*}
\sum_{k\in R_\varsigma(t)} \E (\xi_k(t))^2=\frac{\sum_{k\in R_\varsigma(t)} \E (\eta_k(t))^2}{a_\varsigma(t)}=1
\end{equation*}
and $\me (\eta_k(t))^2 \le 1$, we conclude that
\begin{equation*}
(u(t))^2 \sum_{k\in R_\varsigma(t)} (\E (\xi_k(t))^2)^2= O(h_0(a(t)))\frac{\sum_{k\in R_\varsigma(t)} \E (\eta_k(t))^2}{(a_\varsigma(t))^2}=o(1), \quad t\to\infty
\end{equation*}
having utilized \eqref{eq:onevar} or \eqref{eq:one} for the last equality. Now \eqref{eq:ss3} entails \eqref{eq:p1}.
		
Observe that, with $u\in\mr$ fixed, formula \eqref{eq:p1} reads $$\lim_{t\to \infty}\E \eee^{u  W_\varsigma(t)/(a_\varsigma(t))^{1/2}}= \eee^{u^2/2},$$ which implies a central limit theorem $W_\varsigma(t)/(a_\varsigma(t))^{1/2}{\overset{{\rm d}}\longrightarrow} {\rm Normal} (0,1)$ as $t\to \infty$. Here, as before, ${\rm Normal} (0,1)$ denotes a random variable with the standard normal distribution.
		
We are ready to prove \eqref{eq:r3}. Recall that
$g(t)=\big(2h_0(a(t)) \big)^{1/2}$ and put
$$u(t)= (1-\delta/2) g(t)=O((h_0(a(t)))^{1/2}).$$ Formula \eqref{eq:p1} implies that
\begin{equation}\label{eq:p11}
\E {\rm e}^{u(t) (W_\varsigma(t)/(a_\varsigma(t))^{1/2}-g(t))}={\rm e}^{-(1-\delta^2/4) h_0(a(t)) + o(h_0(a(t)))} \ge {\rm e}^{-(1-\delta^2/8) h_0(a(t))}
\end{equation}
for large $t$. Next, we intend to show the $\Q_{t,u(t)}$-distribution of $W_\varsigma(t)/(a_\varsigma(t))^{1/2}-(u(t)  + 2u(t) h(t))$ converges weakly as $t\to \infty$ to the $\mmp$-distribution of ${\rm Normal} (0,1)$. To this end, we prove convergence of the moment generating functions. Let $\E_{\Q_{t,u(t)}}$ denote the expectation with respect to the probability measure $\Q_{t,u(t)}$. Using \eqref{eq:p1} we obtain, for $r\in\mr$,
\begin{align*}
&\hspace{-0.4cm}\E_{\Q_{t,u(t) }} {\rm e}^{r(W_\varsigma(t)/(a_\varsigma(t))^{1/2}-(u(t)  + 2 u(t) h(t)))}= \frac{\E {\rm e}^{(r+u(t) )  W_\varsigma(t)/(a_\varsigma(t))^{1/2}}}{\E {\rm e}^{u(t)   W_\varsigma(t)/(a_\varsigma(t))^{1/2}}}{\rm e}^{-r(u(t)  + 2u(t) h(t))}\\
&= \exp\big( (1/2)(r+u(t) )^2  + (r+u(t))^2 h(t) - (1/2) u^2(t)  -   u^2(t) h(t) -   r(u(t)  + 2u(t)h(t))
\big)\\	&=\exp\big((1/2+h(t))r^2\big)\to \exp((1/2)r^2)=\me \exp(r\, {\rm Normal} (0,1)),\quad t\to \infty.
\end{align*}
The weak convergence ensures that
\begin{multline*}
{\limsup}_{t\to \infty} \Q_{t,u(t)} \big\{ W_\varsigma(t)/(a_\varsigma(t))^{1/2}\leq (1-\delta)g(t)\big\}\\
\le \lim_{t\to \infty}\Q_{t,u(t)} \big\{W_\varsigma(t)/(a_\varsigma(t))^{1/2}\leq u(t) + 2u(t)h(t) \big\}
=\mmp\{{\rm Normal}\,(0,1)\leq 0\}= 1/2.
\end{multline*}
Since $\lim_{t\to \infty} (g(t) - (u(t) + 2u(t)h(t)))=+\infty$, we also have $$\lim_{t\to \infty} \Q_{t,u(t)} \big\{W_\varsigma(t)/(a_\varsigma(t))^{1/2}\leq g(t)\big\}=1.$$
Summarizing,
\begin{equation}\label{eq:p12}
\Q_{t,u(t)} (U^{(\varsigma)}_t)=\Q_{t,u(t)} \big\{W_\varsigma(t)/(a_\varsigma(t))^{1/2} \leq g(t)\big\}-\Q_{t,u(t)} \big\{W_\varsigma(t)/(a_\varsigma(t))^{1/2}\leq (1- \delta)g(t)\big\}\ge 1/3
\end{equation}
for large $t$. Now \eqref{eq:r3} follows from \eqref{eq:ss10}, \eqref{eq:p11} and \eqref{eq:p12}. The proof of Lemma \ref{lem:new} is complete.
\end{proof}

\begin{proof}[Proof of Proposition \ref{lilhalf2}]
Given sufficiently small $\delta>0$ pick $\gamma\in (0, (\sqrt{5}-1)/2)$ satisfying $(1+\gamma)(1-\delta^2/8)<1$. Then
\begin{equation*}
{\limsup}_{t\to\infty}\frac{X^\ast(t)}{(2a(t)h_0(a(t)))^{1/2}}\\\geq {\limsup}_{n\to \infty} \frac{X^\ast(\tau_n)}{(2a(\tau_n)h_0(a(\tau_n)))^{1/2}} \ge 1-\delta  \quad\text{{\rm a.s.}},
\end{equation*}
where the second inequality is secured by Lemma \ref{lem:new}. Now \eqref{lil11} follows upon letting $\delta$ tend to $0+$.
\end{proof}

\section{Proof of Theorem \ref{thm:Ginibre}}\label{sec:Ginibre}

In view of \eqref{eq:distr} it suffices to prove that $$\limsup_{t\to\infty}\frac{N(t)-t}{t^{1/4}(\log t)^{1/2}}=\frac{1}{\pi^{1/4}}\quad\text{{\rm a.s.}}$$ Recalling \eqref{eq:var11} and that $\me N(t)=t$ for $t>0$ this is equivalent to
\begin{equation}\label{eq:gin}
\limsup_{t\to\infty}\frac{N(t)-\me N(t)}{(2{\rm Var}\,N(t)\log {\rm Var}\,N(t))^{1/2}}=1\quad\text{{\rm a.s.}}
\end{equation}

Since $N(t)$ is the sum of independent indicators, we intend to apply Theorem \ref{thm:main}. While condition (A4) holds trivially, conditions (A1), (A2) and (A3), with $\mu=2$, are secured by \eqref{eq:var11}. Since the function $b$ ($b(t)=t$) is strictly increasing and continuous, condition (A5) holds by a sufficient condition given in Remark \ref{suff}. The first part of condition (B1) is secured by continuity of $t\mapsto {\rm Var}\,N(t)=\sum_{k\geq 1}\mmp\{\Gamma_k\leq t\}\mmp\{\Gamma_k>t\}$.

Checking condition (B2.1) requires some preparation. Recall that a random variable $Y_{\nu,\,t}$ has a discrete Bessel distribution with parameters $\nu>-1$ and $t>0$ if $$\P\{Y_{\nu,\,t}=n\} = \frac 1{I_\nu(t) n! \Gamma(n+\nu+1)} \bigg(\frac t2\bigg)^{2n+\nu}, \quad n\in\mn_0,$$ where $\Gamma$ is the Euler gamma function and $I_\nu$ is the modified Bessel function of the first kind given by
$$I_\nu(t) =  \sum_{n\geq 0} \frac 1{n! \Gamma(n+\nu+1)} \bigg(\frac t2\bigg)^{2n+\nu}, \quad \nu > -1, \ t>0.$$ We shall need a central limit theorem for $Y_{\nu,\,t}$, properly normalized and centered.
\begin{lemma}\label{lem:cltY}
The variables $2t^{-1/2}(Y_{\nu,\,t}-t/2)$ converge in distribution as $t\to\infty$ to a random variable with the standard normal distribution.   	
\end{lemma}
\begin{rem}
The centering and the normalization used in Lemma \ref{lem:cltY} can be replaced with $\me Y_{\nu,\,t}$ and $(\var Y_{\nu,\,t})^{1/2}$, respectively. The equalities
$$\me Y_{\nu,\,t}=tR_\nu(t)/2\quad\text{and}\quad \var Y_{\nu,\,t}=t^2 R_\nu(t)(R_{\nu+1}(t)-R_\nu(t))/4+tR_\nu(t)/2,\quad \nu>-1,\ t>0,$$ with $R_\nu(t):=I_{\nu+1}(t)/I_\nu(t)$, can be checked directly or found in formula (2.1) of \cite{Lin+Kalbfleisch:2000}. Using these in combination with
\begin{equation}\label{eq:corr}
\lim_{t\to\infty}t(1-R_\nu(t))=\nu+1/2
\end{equation}
we infer $$\me Y_{\nu,\,t}= t/2-(\nu+1/2)/2+o(1)\quad\text{and}\quad \var Y_{\nu,\,t}~\sim~ t/4, \quad t\to\infty$$ thereby justifying the claim.

Relation \eqref{eq:corr} can be derived with some efforts from formula (9.13.7) on p.~167 in \cite{Spain+Smith:1970}, which provides a series representation of the modified Bessel function $I_\nu$. We note in passing that formula (A.3) in \cite{Lin+Kalbfleisch:2000} which states that $\lim_{t\to\infty}t(1-R_\nu(t))=2\nu+1$ seems to be incorrect.
\end{rem}
\begin{proof}
Observe that $Y_{\nu,\,t}$ has finite exponential moments of all orders: for $p\in\mr$,
\begin{equation}\label{eq:wr7}
\begin{split}
\E \eee^{pY_{\nu,\,t}} &= \sum_{n\geq 0} \frac{\eee^{pn}}{I_\nu(t)n! \Gamma(n+\nu+1)}\bigg(\frac t2\bigg)^{2n+\nu}\\
& =  \frac{e^{-\nu p/2}}{I_\nu(t)} \sum_{n\geq 0} \frac{1}{n! \Gamma(n+\nu+1)}\Big(\frac{\eee^{p/2} t}2\Big)^{2n+\nu}
= \frac{\eee^{-\nu p/2} I_\nu(\eee^{p/2}t)}{I_\nu(t)}<\infty.
\end{split}
\end{equation}
It is known (see, for instance, formula (9.13.7) on p.167 in \cite{Spain+Smith:1970}) that, for $\nu>-1$,
\begin{equation}\label{eq:equivI}
I_\nu(t)~\sim~ \frac{\eee^t}{\sqrt{2\pi t}}, \quad t\to\infty.
\end{equation}

To prove the lemma, it suffices to show convergence of moment generating functions. To this end, write, for $p\in\mr$,
\begin{multline*}
	\me \exp\Big(2p t^{-1/2} Y_{\nu,\,t}-p t^{1/2}\Big)=\exp(-p(t^{1/2}+\nu t^{-1/2}))\frac{I_\nu(\eee^{p t^{-1/2}}t)}{I_\nu(t)}~\sim~ \exp(-pt^{1/2}+\eee^{p t^{-1/2}}t-t)\\\sim\exp(-pt^{1/2}+(pt^{-1/2}+p^2t^{-1}/2)t)=\exp(p^2/2)=\me\exp(p\,{\rm Normal} (0,1)), \quad t\to\infty
\end{multline*}
having utilized \eqref{eq:equivI} for the first asymptotic equivalence and Taylor's expansion for the second.
\end{proof}

Now we state a lemma designed to check that relation \eqref{eq:onevar} (the second part of condition (B2.1)) holds.  	
\begin{lemma}\label{lem:B1}
Fix any $\varsigma\in(0,1)$. Then there exists the unique $x>0$ such that
$${\rm Var}\,\Big(\sum_{k =\lfloor t-xt^{1/2}\rfloor}^{\lfloor t+xt^{1/2}\rfloor}\1_{\{\Gamma_k\le t\}} \Big)~ \sim~(1-\varsigma) {\rm Var}\, N(t), \quad t\to \infty.$$
\end{lemma}
\begin{proof}
Given $x>0$ put $c(t,x):=\lfloor t-xt^{1/2}\rfloor-1$ for large $t$ ensuring that $c(t,x)\geq 2$.

We first note that
\begin{equation*}
{\rm Var}\,\Big( \sum_{k=1}^{c(t,x)}\1_{\{\Gamma_k\le t\}}\Big)
= \sum_{k=1}^{c(t,x)} \P\{\Gamma_k\le t\}\P\{\Gamma_k > t\}.
\end{equation*}
Using $$\P\{\Gamma_k \le t\}= 1- \eee^{-t}\sum_{j=0}^{k-1}\frac{t^j}{j!},\quad t\geq 0,$$ we obtain
\begin{align*}
{\rm Var}\,\Big( \sum_{k=1}^{c(t,x)}\1_{\{\Gamma_k\le t\}} \Big)
&= \sum_{k=1}^{c(t,x)}  \sum_{j=0}^{k-1}\eee^{-t} \frac{t^j}{j!}\sum_{i\geq k} \eee^{-t}\frac{t^i}{i!}\\
&= \eee^{-2t} \sum_{j=0}^{c(t,x)-1} \sum_{k=j+1}^{c(t,x)} \sum_{i\geq k}\frac{t^{i+j}}{i!j!} \\
&= \eee^{-2t} \sum_{j=0}^{c(t,x)- 1 } \sum_{i\geq j+1} \sum_{k=j+1}^{\min(i,\,c(t,x))}  \frac{t^{i+j}}{i!j!} \\
&= \eee^{-2t} \sum_{j=0}^{c(t,x)- 1 } \sum_{i=j+1}^{c(t,x)} (i-j) \frac{t^{i+j}}{i!j!}
+ \eee^{-2t} \sum_{j=0}^{c(t,x)- 1 } \sum_{i\geq c(t,x)+1} (
c(t,x)-j) \frac{t^{i+j}}{i!j!}\\
& =:
A(t,x)+ B(t,x).
\end{align*}
	
We proceed by analyzing $A$:
\begin{multline*}
A(t,x)
= \eee^{-2t} \sum_{j=0}^{c(t,x)- 1 }\frac{t^j}{j!}\sum_{i=j+1}^{c(t,x)}\frac{t^{i}}{(i-1)!}
- \eee^{-2t}\sum_{j=1}^{c(t,x)-1}\frac{t^j}{(j-1)!}\sum_{i=j+1}^{c(t,x)}\frac{t^{i}}{i!}\\=  \eee^{-2t} \sum_{j=0}^{
c(t,x)- 1}\frac{t^j}{j!}\cdot t \cdot \sum_{i=j}^{
c(t,x)-1}\frac{t^{i}}{i!}- \eee^{-2t}\sum_{j=0}^{
c(t,x)-2} \frac{t^j}{j!} \cdot t \cdot  \sum_{i=j+2}^{
c(t,x)}  \frac{t^{i}}{i!}\\=\eee^{-2t} \frac{t^{
c(t,x)-1}}{(
c(t,x)-1)!} \cdot t \cdot  \frac{t^{
c(t,x)-1}}{(
c(t,x)-1)!}+
\eee^{-2t} \sum_{j=0}^{c(t,x)-2} \frac{t^{2j+1}}{(j!)^2}+\eee^{-2t} \sum_{j=0}^{c(t,x)-2} \frac{t^{2j+2}}{j!(j+1)!}\\- \eee^{-2t}\sum_{j=0}^{c(t,x)-2} \frac{t^{j+1}}{j!} \frac{t^{c(t,x)}}{(c(t,x))!}=: A_1(t,x)+A_2(t,x)+A_3(t,x)-A_4(t,x).
\end{multline*}
By Stirling's formula, the definition of $c$ and Taylor's expansion of $y\mapsto -\log (1-y)$,
\begin{multline}\label{eq:stirling}
\eee^{-t}\frac{t^{c(t,x)-1}}{(c(t,x)-1)!}~\sim~ \frac{\exp(-t+(t-x t^{1/2})\log t-(t-xt^{1/2})\log(t-x t^{1/2})+t-xt^{1/2})}{(2\pi t)^{1/2}}\\
=\frac{\exp(-xt^{1/2}-(t-x t^{1/2})\log(1-xt^{-1/2}))}{(2\pi t)^{1/2}}~\sim~ \frac{\eee^{-x^2/2}}{(2\pi t)^{1/2}},\quad t\to\infty.
\end{multline}
With this at hand, we infer $$A_1(t,x)=O(1)=o(\var N(t)),\quad t\to\infty.$$

As a preparation for what follows, we note that, by Lemma \ref{lem:cltY},
\begin{equation*}
\mmp\{Y_{0,\,2t}\le c(t,x)-2\}=\mmp\big\{ (2t^{-1})^{1/2}(Y_{0,\,2t}-t) \le - 2^{1/2}x+o(1)\big\}~\to~ \Phi(-2^{1/2}x),\quad t\to\infty,
\end{equation*}
where $\Phi$ is the distribution function of ${\rm Normal}(0,1)$. Using this in combination with \eqref{eq:equivI} we obtain
$$A_2(t,x) =\eee^{-2t} t I_0(2t) \,\mmp\{Y_{0,\,2t}\le
c(t,x)-2\}~\sim~ \frac{\Phi(-2^{1/2}x)}{2\pi^{1/2}}t^{1/2},\quad t\to\infty$$ and
$$A_3(t,x)=\eee^{-2t} t I_1(2t) \,\mmp\{Y_{1,\,2t}\le
c(t,x)-2\}~\sim~\frac{\Phi(-2^{1/2}x)}{2\pi^{1/2}}t^{1/2},\quad t\to\infty.$$
By the central limit theorem for standard random walks,
\begin{equation}\label{eq:gammad}
\mmp\{\Gamma_{
c(t,x)}>t\}=\mmp\Big\{\frac{\Gamma_{
c(t,x)}-
c(t,x)}{(
c(t,x))^{1/2}}>\frac{t-
c(t,x)}{(
c(t,x))^{1/2}}\Big\}~\to~ \mmp\{{\rm Normal}(0,1)>x\}=\Phi(-x),\quad t\to\infty.
\end{equation}
This together with \eqref{eq:stirling} yields
$$A_4(t,x)
=t\,\mmp\{\Gamma_{
c(t,x)-1}>t\}\eee^{-t}\frac{t^{
c(t,x)}}{(
c(t,x))!}~\sim~ \frac{\eee^{-x^2/2}\Phi(-x)}{(2\pi)^{1/2}}t^{1/2},\quad t\to\infty.$$

Left with analyzing $B$ we first obtain the following decomposition:
\begin{multline*}
B(t,x)=
 c(t,x)\eee^{-2t} \sum_{j=0}^{
 c(t,x)- 1} \frac{t^j}{j!}\sum_{i\geq
 c(t,x)+1} \frac{t^i}{i!} -t \eee^{-2t} \sum_{j=0}^{
 c(t,x)- 2} \frac{t^j}{j!} \sum_{i\geq
 c(t,x)+1} \frac{t^i}{i!}\\=
 c(t,x)\eee^{-t} \frac{t^{
 c(t,x)-1}}{(
 c(t,x)-1)!}\sum_{i\geq
 c(t,x)+1} \eee^{-t}\frac{t^i}{i!}-(t-
 c(t,x))\sum_{j=0}^{
 c(t,x)- 2}\eee^{-t}\frac{t^j}{j!} \sum_{i\geq
 c(t,x)+1}\eee^{-t}\frac{t^i}{i!}\\=
  c(t,x)\eee^{-t} \frac{t^{
  c(t,x)-1}}{(
  c(t,x)-1)!}\mmp\{\Gamma_{
  c(t,x)+1}\leq t\}-(t-
  c(t,x))\mmp\{\Gamma_{
  c(t,x)-1}>t\}\mmp\{\Gamma_{
  c(t,x)+1}\leq t\}\\=:B_1(t,x)
  -B_2(t,x).
\end{multline*}
By the same argument as used for the analysis of $A_4$
we infer $$B_1(t,x)
~\sim~\frac{\eee^{-x^2/2}\Phi(x)}{(2\pi)^{1/2}}t^{1/2},\quad t\to\infty.$$
Finally, by \eqref{eq:gammad}, $$B_2(t,x)
~\sim~ x\Phi(-x)\Phi(x)t^{1/2},\quad t\to\infty.$$

Combining all the fragments together we arrive at
$${\rm Var}\,\Big( \sum_{k=1}^{
c(t,x)}\1_{\{\Gamma_k\le t\}}\Big)~\sim~ t^{1/2} \Big(\frac{\Phi(-2^{1/2}x)}{\pi^{1/2}}-\frac{\eee^{-x^2/2}\Phi(-x)}{(2\pi)^{1/2}}+\frac{\eee^{-x^2/2}\Phi(x)}{(2\pi)^{1/2}}-x\Phi(-x)\Phi(x)\Big),\quad t\to\infty.$$
Arguing similarly we also conclude that, with $d(t,x)
:=\lfloor t+xt^{1/2}\rfloor$, as $t\to\infty$, $$
{\rm Var}\,\Big(\sum_{k=1}^{
d(t,x)}\1_{\{\Gamma_k\le t\}}\Big)~\sim~ t^{1/2}\Big(\frac{\Phi(2^{1/2}x)}{\pi^{1/2}}-\frac{\eee^{-x^2/2}\Phi(x)}{(2\pi)^{1/2}}+\frac{\eee^{-x^2/2}\Phi(-x)}{(2\pi)^{1/2}}+x\Phi(-x)\Phi(x)\Big)
$$ and thereupon
\begin{multline*}
{\rm Var}\,\Big( \sum_{k =c(t,x)+1}^{
d(t,x)}\1_{\{\Gamma_k\le t\}} \Big)\\
\sim~(t/\pi)^{1/2} \big(\Phi(2^{1/2}x)-\Phi(-2^{1/2}x)-2^{1/2}\eee^{-x^2/2}\Phi(x)+2^{1/2}\eee^{-x^2/2}\Phi(-x)+2\pi^{1/2} x\Phi(-x)\Phi(x)\big).
\end{multline*}

In view of \eqref{eq:var11} it remains to show that given $\varsigma\in(0,1)$ there exists the unique $x>0$ such that
$$f(x):=\Phi(2^{1/2}x)-\Phi(-2^{1/2}x)-2^{1/2}\eee^{-x^2/2}\Phi(x)+2^{1/2}\eee^{-x^2/2}\Phi(-x)+2\pi^{1/2}x\Phi(-x)\Phi(x)=1-\varsigma.$$
Since $f^\prime(x)=2\pi^{1/2}\Phi(-x)\Phi(x)$, the function $f$ is increasing on $[0,\infty)$. Also, it is continuous with $f(0)=0$ and $\lim_{x\to\infty}f(x)=1$. The stated properties of $f$ justify the claim.
\end{proof}

Fix any $\varsigma\in (0,1)$ and put $R_\varsigma(t,x)
:=\{k\in\mn: c(t,x)<k\leq d(t,x)\}$ with $c(t,x)=\lfloor t-xt^{1/2}\rfloor-1$ and $d(t,x)=\lfloor t+xt^{1/2}\rfloor$. Here, $x$ is the unique positive number, for which the asymptotic relation of Lemma \ref{lem:B1} holds. In particular, this verifies \eqref{eq:onevar}, which is one part of condition (B2.1).

Further,
$$c(\tau_{n+1}{, x})+1-d(\tau_n, x)\ge \tau_{n+1}-\tau_n-x(\tau_{n+1}^{1/2}+\tau_n^{1/2})-1=\big(\tau_{n+1}^{1/2}+\tau_n^{1/2}\big)\big(\tau_{n+1}^{1/2}-
\tau_n^{1/2}-x\big)-1.$$
By Proposition B.1 in \cite{Fenzl+Lambert:2021}, $\var N(t)=t\eee^{-2t}(I_0(2t)+I_1(2t))$ for $t>0$. Using formula (9.13.7) on p.~167 in \cite{Spain+Smith:1970} we infer $$\var N(t)=\Big(\frac{t}{\pi}\Big)^{1/2}-\frac{1}{16(\pi t)^{1/2}}+o(t^{-1/2}),\quad t\to\infty.$$ Recalling that $w_n(\gamma, 2)=n^{1+\gamma}$ we infer $\tau_n^{1/2}=(\tau_n(\gamma, 2))^{1/2}= \pi^{1/2} n^{1+\gamma}+O(1)$ as $n\to\infty$ and thereupon $$\tau_{n+1}^{1/2}-\tau_n^{1/2}=\pi^{1/2} (1+\gamma)n^{\gamma}+O(1),\quad n\to \infty.$$ As a consequence,
$$\lim_{n\to\infty} (c(\tau_{n+1}, x)+1-d(\tau_n{, x}))=+\infty,$$ whence $d(\tau_n,x)<c(\tau_{n+1}, x)+1$ for large enough $n$. This proves that there exists $n_0\in\mn$ such that the sets $R_\varsigma(\tau_{n_0},x)$, $R_\varsigma(\tau_{n_0+1},x),\ldots$ are disjoint, that is, the second part of condition (B2.1) holds true.

Now \eqref{eq:gin} is secured by Theorem \ref{thm:main}.

\section{Proofs related to Karlin's occupancy scheme}\label{sec:Karlin}

\subsection{Auxiliary results}
We shall need several preparatory results.
\begin{lemma}\label{lem:reg}
(a) Let a function $x$ be regularly varying at $\infty$ of negative index and $y$ any positive function satisfying $\lim_{t\to\infty}t^{-1}y(t)=\infty$. Then $$\lim_{t\to\infty}\frac{x(y(t))}{x(t)}=0.$$

\noindent (b) Let a function $x$ be regularly varying at $\infty$ of positive index and $y$ any positive function satisfying $\lim_{t\to\infty}t^{-1}y(t)=0$. Then $$\lim_{t\to\infty}\frac{x(y(t))}{x(t)}=0.$$
\end{lemma}
\begin{proof}
(a) By Theorem 1.5.3 \cite{Bingham+Goldie+Teugels:1989}, there exists a nonincreasing function $x_1$ such that $x(t)\sim x_1(t)$ as $t\to\infty$. Also, for each $c>0$, $y(t)\ge ct$ for large enough $t$ and thereupon $$0\le \limsup_{t\to\infty}\frac{x(y(t))}{x(t)}=\limsup_{t\to\infty}\frac{x_1(y(t))}{x_1(t)}\le \limsup_{t\to\infty}\frac{x_1(ct)}{x_1(t)}=c^\beta,$$ where $\beta$ is the (negative) index of regular variation. Sending $c\to\infty$ finishes the proof.

The proof of part (b) is analogous, hence omitted.
\end{proof}

\begin{lemma}\label{eq:thm21}
(a) Conditions \eqref{eq:deHaan} and \eqref{eq:slowly} entail
\begin{equation}\label{eq:rho}
\rho(t)~\sim~(\beta+1)^{-1}(\log t)^{\beta+1}l(\log t),\quad t\to\infty.
\end{equation}

\noindent (b) Conditions \eqref{eq:deHaan} and \eqref{eq:slowly2} entail
\begin{equation}\label{eq:rho2}
\rho(t)~\sim~(\sigma \lambda)^{-1}\exp(\sigma (\log t)^\lambda)(\log t)^{1-\lambda},\quad t\to\infty.
\end{equation}
\end{lemma}
\begin{proof}
(a) By Theorem 3.7.3 in \cite{Bingham+Goldie+Teugels:1989}, $$\rho(t)~\sim~\int_\eee^t y^{-1}\ell(y){\rm d}y~\sim~\int_\eee^t y^{-1}(\log y)^\beta l(\log y){\rm d}y=\int_1^{\log t}y^\beta l(y){\rm d}y,\quad t\to\infty.$$ Relation \eqref{eq:rho} now follows by an application of Karamata's theorem (Proposition 1.5.8 in \cite{Bingham+Goldie+Teugels:1989}).

\noindent (b) By the same reasoning as above $$\rho(t)~\sim~\int_\eee^t y^{-1}\ell(y){\rm d}y~\sim~\int_\eee^t y^{-1}\exp(\sigma(\log y)^\lambda){\rm d}y=\int_1^{\log t} \exp(\sigma y^\lambda){\rm d}y, \quad t\to\infty.$$ An application of L'H\^{o}pital's rule yields \eqref{eq:rho2}.
\end{proof}

In the case $\alpha=1$ treated by Theorem \ref{thm:Karlin1}, that is, $\rho(t)\sim tL(t)$ as $t\to\infty$, a slowly varying function $L$ cannot be arbitrary. Indeed, according to Lemma 3 in \cite{Karlin:1967}
\begin{equation}\label{eq:raterho}
\lim_{t\to\infty}t^{-1}\rho(t)=0
\end{equation}
and $\int_1^\infty y^{-2}\rho(y){\rm d}y\leq 1$. As a consequence, $\lim_{t\to\infty} L(t)=0$ and, for each $c>0$, $\int_c^\infty y^{-1}L(y){\rm d}y<\infty$. The latter entails that the function $\hat L(t)=\int_t^\infty y^{-1}L(y){\rm d}y$ is well-defined for large $t$. Obviously $\lim_{t\to\infty}\hat L(t)=0$. It is known (see, for instance, Proposition 1.5.9b in \cite{Bingham+Goldie+Teugels:1989}) that $\hat L$ is slowly varying at $\infty$ and that
\begin{equation}\label{eq:rate}
\lim_{t\to\infty}\frac{L(t)}{\hat L(t)}=0.
\end{equation}
\begin{lemma}\label{lem:alone}
If $\rho(t)\sim tL(t)$ as $t\to\infty$ for some slowly varying $L$, then $$t^{-1}\int_0^\infty y^{-2}\eee^{-1/y}\rho(ty){\rm d}y~\sim~\int_t^\infty y^{-1}L(y){\rm d}y,\quad t\to\infty.$$
\end{lemma}
\begin{proof}
It suffices to prove that, as $t\to\infty$, $$t^{-1}\int_0^1 y^{-2}\eee^{-1/y}\rho(ty){\rm d}y=o(\hat L(t))\quad\text{and}\quad t^{-1}\int_1^\infty y^{-2}(1-\eee^{-1/y})\rho(ty){\rm d}y=o(\hat L(t)).$$
The former follows from $$t^{-1}\int_0^1 y^{-2}\eee^{-1/y}\rho(ty){\rm d}y\leq t^{-1}\rho(t)\int_0^1 y^{-2}\eee^{-1/y}{\rm d}y~\sim~ \int_0^1 y^{-2}\eee^{-1/y}{\rm d}y\, L(t)=o(\hat L(t)).$$ Here, the first inequality is implied by monotonicity of $\rho$ and the last asymptotic relation is secured by \eqref{eq:rate}. The latter is justified as follows $$t^{-1}\int_1^\infty y^{-2}(1-\eee^{-1/y})\rho(ty){\rm d}y\leq t^{-1}\int_1^\infty y^{-3}\rho(ty){\rm d}y=t\int_t^\infty y^{-3}\rho(y){\rm d}y \sim L(t)=o(\hat L(t)).$$ The asymptotic equivalence is ensured by Proposition 1.5.10 in \cite{Bingham+Goldie+Teugels:1989} and formula \eqref{eq:rate} entails the last asymptotic relation.
\end{proof}

\begin{lemma}\label{lem:KarlB2}
For any function $\ell$ slowly varying at $\infty$ there exists a positive function $c$ satisfying $$\lim_{t\to\infty} \frac{\ell(c(t))}{\ell(t)}=1 \quad\text{and}\quad \lim_{t\to\infty} \frac{c(t)}{t}=\infty.$$
\end{lemma}
\begin{proof}
According to the representation theorem for slowly varying function (Theorem 1.3.1 in \cite{Bingham+Goldie+Teugels:1989}),
$$\ell(t)=A(t)\exp\Big(\int_a^t u^{-1} \varepsilon(u)\dd u\Big),$$ where $\lim_{t\to\infty} A(t)=A\in(0,\infty)$, $a>0$ and $\lim_{t\to\infty} \varepsilon(t)=0$. Therefore,
the first limit relation in the statement of the lemma is equivalent to
$$\lim_{t\to\infty} \int_t^{c(t)}u^{-1} \varepsilon(u)\dd u=0.$$
In view of $$\Big|\int_t^{c(t)}u^{-1} \varepsilon(u)\dd u\Big|\le \int_t^{c(t)}u^{-1} |\varepsilon(u)| \dd u\le \sup_{u\ge t}|\varepsilon(u)|\log (t^{-1}c(t)),$$ one may take $c(t)=t\exp\big((\sup_{u\ge t}|\varepsilon(u)|)^{-1/2}\big)$. Since $\lim_{t\to\infty} \sup_{u\ge t}|\varepsilon(u)|=0$, both claimed limit relations hold true.
\end{proof}

Recall that we denoted by $\Pi_{\ell,\,\infty}$ the subclass of the de Haan class $\Pi$ with the auxiliary function $\ell$ satisfying $\lim_{t\to\infty}\ell(t)=\infty$. Also, we recall that $K_j(t)$ and $K_j^\ast(t)$ denote the number of boxes at time $t$ in the Poissonized Karlin scheme containing at least $j$ balls and exactly $j$ balls, respectively. For $k\in\mn$ and $t\geq 0$, denote by $\pi_k(t)$ the number of balls in box $k$ at time $t$. By the thinning property of Poisson processes, the processes $(\pi_1(t))_{t\geq 0}$, $(\pi_2(t))_{t\geq 0},\ldots$ are independent and, for $k\in\mn$, $(\pi_k(t))_{t\geq 0}$ is a Poisson process with intensity $p_k$. Now we are ready to exhibit representations of $K_j(t)$ and $K_j^\ast(t)$ as the infinite sums of independent indicators
\begin{equation}\label{eq:sumindic}
K_j(t)=\sum_{k\geq 1}\1_{\{\pi_k(t)\geq j\}}\quad\text{and}\quad K^\ast_j(t)=\sum_{k\geq 1}\1_{\{\pi_k(t)=j\}} ,\quad j\geq 1,~t\geq 0.
\end{equation}

\subsection{Asymptotic behavior of $\me K_j(t)$ and ${\rm Var}\,K_j(t)$}

Passing to the proofs of Theorems \ref{thm:Karlin0}, \ref{thm:Karlin} and \ref{thm:Karlin1} we first show that the asymptotics of $\me K_j(t)$ and ${\rm Var}\,K_j(t)$ are as stated in these theorems. In addition, we find the asymptotics of $\me K_j^\ast(t)$.
\begin{lemma}\label{meanvar}
Assume that $\rho\in \Pi_{\ell,\,\infty}$. Then, for each $j\in\mn$,
\begin{equation}\label{eq:mean0}
\me K_j(t)~\sim~ \rho(t),\quad t\to\infty,
\end{equation}
\begin{equation}\label{eq:var0}
{\rm Var}\,K_j(t)~\sim~ \Big(\log 2- \sum_{k=1}^{j-1}\frac{(2k-1)!}{(k!)^2 2^{2k}}\Big)\ell(t),\quad t\to\infty
\end{equation}
and
\begin{equation}\label{eq:meK*0}
\me K^\ast_j(t)~\sim~ \frac{\ell(t)}{j},\quad t\to\infty.
\end{equation}

Assume that $\rho(t)\sim t^\alpha L(t)$ as $t\to\infty$ for some $\alpha\in(0,1]$ and some $L$ slowly varying at $\infty$. If $\alpha\in(0,1)$ and $j\in\mn$ or $\alpha=1$ and $j\geq 2$, then, as $t\to\infty$,
\begin{equation}\label{eq:momincreas}
\me K_j(t)~\sim~\frac{\Gamma(j-\alpha)}{(j-1)!}\rho(t),
\end{equation}
\begin{equation}\label{eq:var}
\lim_{t\to\infty}\frac{{\rm Var}\, K_j(t)}{\rho(t)}=\Big(\sum_{i=0}^{j-1}\frac{\Gamma(i+j-\alpha)}{i!(j-1)!2^{i+j-1-\alpha}}-\frac{\Gamma(j-\alpha)}{(j-1)!}\Big)>0
\end{equation}
and
\begin{equation}\label{eq:momtons}
\me K_j^\ast(t)~\sim~\frac{\alpha \Gamma(j-\alpha)}{j!}\rho(t),\quad t\to\infty.
\end{equation}

If $\alpha=1$, then
\begin{equation}\label{eq:varalone}
{\rm Var}\, K_1(t)~\sim~ \me K_1(t)~\sim~\me K^*_1(t)~\sim~ t\hat L(t),\quad t\to\infty.
\end{equation}
\end{lemma}
\begin{proof}
Assume that $\rho\in\Pi_{\ell,\,\infty}$. By Theorem 1 in \cite{Karlin:1967}, $\me K_1(t)\sim \rho(t)$ as $t\to\infty$. By Proposition 2.5 in \cite{Iksanov+Kotelnikova:2022}, \eqref{eq:meK*0} holds. Relation \eqref{eq:deHaan} entails $\lim_{t\to\infty}(\rho(t)/\ell(t))=\infty$. Hence, for $j\geq 2$, $\me K_j(t)=\me K_1(t)-\sum_{i=1}^{j-1}\me K_i^\ast(t)\sim \rho(t)$ as $t\to\infty$, that is, \eqref{eq:mean0} holds true. Formula \eqref{eq:var0} can be found in Corollary 2.8 of \cite{Iksanov+Kotelnikova:2022}.

By formula (23) in \cite{Karlin:1967}, \eqref{eq:momtons} holds whenever $\alpha\in(0,1)$ and $j\in\mn$ or $\alpha=1$ and $j\geq 2$. Assume that $\alpha\in (0,1)$. By Theorem 1 in \cite{Karlin:1967}, $\me K_1(t)\sim \Gamma(1-\alpha)t^\alpha L(t)$ as $t\to\infty$. Noting that $\me K_{j+1}(t)=\me K_j(t)-\me K_j^\ast(t)$ for $j\in\mn$ and $t\geq 0$, relation \eqref{eq:momincreas} follows with the help of mathematical induction.

Assume that $\alpha=1$ and $j\geq 2$. The previous argument does not apply and we have to resort to direct calculations. Write with the help of the first equality in \eqref{eq:sumindic}, for $t\ge 0$,
$$\me K_j(t)=\me \sum_{k\ge 1} \1_{\{\pi_k(t)\ge j\}}=\sum_{k\ge 1} \Big(1-\eee^{-p_kt}\sum_{i=0}^{j-1}\frac{(p_kt)^i}{i!}\Big)=\int_{(1,\,\infty)} \Big(1-\eee^{-t/x}\sum_{i=0}^{j-1}\frac{(t/x)^i}{i!}\Big){\rm d}\rho(x).$$ Using the inequality
\begin{equation}\label{eq:addineq}
 1-\eee^{-x}\sum_{i=0}^{j-1}\frac{x^i}{i!}\le \frac{x^j}{j!},\quad x\geq 0
\end{equation}
and \eqref{eq:raterho} we infer
\begin{equation}\label{eq:addineq1}
\limsup_{x\to\infty} \Big(1-\eee^{-t/x}\sum_{i=0}^{j-1}\frac{(t/x)^i}{i!}\Big) \rho(x) \le \lim_{x\to\infty} \frac{(t/x)^j}{j!} \rho(x) = 0.
\end{equation}
Hence, integrating by parts and changing the variable $t/x=y$ we obtain
\begin{equation}\label{eq:meK}
\me K_j(t) =\int_0^t \eee^{-y}\frac{y^{j-1}}{(j-1)!} \rho(t/y){\rm d} y,\quad t\geq 0.
\end{equation}
According to Potter's bound (Theorem 1.5.6 in \cite{Bingham+Goldie+Teugels:1989}), given $\delta\in (0,1)$ there exists $x_0>1$ such that, for $y\le t/x_0$ and $t\geq x_0$,
$$\frac{\rho(t/y)}{\rho(t)}\le 2\max(y^{-1-\delta},\, y^{-1+\delta}).$$ With this at hand, write, for $t\geq x_0$,
$$\frac{\me K_j(t)}{\rho(t)}=\int_0^t \eee^{-y}\frac{y^{j-1}}{(j-1)!} \frac{\rho(t/y)}{\rho(t)}{\rm d}y=\int_0^{t/x_0}\ldots +\int_{t/x_0}^t\ldots=:A(t)+B(t).$$ By monotonicity of $\rho$,
$$B(t)\le \frac{\rho(x_0)}{\rho(t)}\int_0^\infty \eee^{-y}\frac{y^{j-1}}{(j-1)!}{\rm d}y=\frac{\rho(x_0)}{\rho(t)}~\to~0,\quad t\to\infty.$$ Observe that $$\1_{[0,\,t/x_0]}(y)\eee^{-y} y^{j-1}\frac{\rho(t/y)}{\rho(t)}\leq \eee^{-y} y^{j-1} \max(y^{-1-\delta},\,y^{-1+\delta}),\quad y>0$$ and that the function on the right-hand side is integrable on $(0,\infty)$. Hence, by Lebesgue's dominated convergence theorem,
$$\lim_{t\to\infty} A(t)=\int_0^\infty \lim_{t\to\infty}\1_{[0,\,t/x_0]}(y)\eee^{-y}\frac{y^{j-1}}{(j-1)!} \frac{\rho(t/y)}{\rho(t)}{\rm d}y=\int_0^\infty \eee^{-y}\frac{y^{j-2}}{(j-1)!}{\rm d}y=\frac{1}{j-1}.$$ This proves \eqref{eq:momincreas} in the case $\alpha=1$ and $j\geq 2$.

Assume that $\alpha=1$. By Theorem 1 and formulae (23) and (26) in \cite{Karlin:1967}, $${\rm Var}\, K_1(t)~\sim~ \me K_1(t)~\sim~\me K_1^*(t)~\sim~ \int_0^\infty y^{-2}\eee^{-1/y}\rho(ty){\rm d}y,\quad t\to\infty.$$ This in combination with Lemma \ref{lem:alone} entails \eqref{eq:varalone}.

Assume that $\alpha\in (0,1)$ and $j\in\mn$ or $\alpha=1$ and $j\geq 2$. Arguing along the lines of the first part of the proof one can show that
\begin{multline}\label{eq:formvar}
\var K_j(t)={\rm Var}\,\sum_{k\ge 1}\1_{\{\pi_k(t)\geq j\}}=\sum_{k\ge 1} \eee^{-p_kt}\sum_{i=0}^{j-1}\frac{(p_kt)^i}{i!}\Big(1-\eee^{-p_kt}\sum_{i=0}^{j-1}\frac{(p_kt)^i}{i!}\Big)\\=2\int_0^t  \eee^{-2y}\sum_{i=0}^{j-1}\frac{y^{i+j-1}}{i!(j-1)!}\rho(t/y){\rm d}y-\int_0^t \eee^{-y}\frac{y^{j-1}}{(j-1)!}\rho(t/y){\rm d}y.
\end{multline}
Changing in the first integral the variable $z=2y$ and invoking equality \eqref{eq:meK} we obtain $$\var K_j(t)=\sum_{i=0}^{j-1}\me K_{i+j}(2t)\frac{\binom{i+j-1}{i}}{2^{i+j-1}}-\me K_j(t).$$
This in combination with the already proved relation \eqref{eq:momincreas} and the fact that $\rho$ is regularly varying of index $\alpha$ yields
\begin{multline}
\lim_{t\to\infty} \frac{\var K_j(t)}{\rho(t)}=\lim_{t\to\infty} \Big(\frac{\sum_{i=0}^{j-1}\me K_{i+j}(2t)\frac{\binom{i+j-1}{i}}{2^{i+j-1}}}{\rho(2t)}\frac{\rho(2t)}{\rho(t)}-\frac{\me K_j(t)}{\rho(t)}\Big)\\
=\sum_{i=0}^{j-1}\frac{\Gamma(i+j-\alpha)}{i!(j-1)!2^{i+j-1-\alpha}}-\frac{\Gamma(j-\alpha)}{(j-1)!}:=c_j.\label{eq:inter1}
\end{multline}

Plainly, the constants $c_j$ are nonnegative. Now we intend to show that $c_j>0$ for each $j\in\mn$ if $\alpha\in (0,1)$ and for each $j\geq 2$ if $\alpha=1$. In the latter case the constants $c_j$ take a simpler form $$c_j=\frac{(2j-3)!}{((j-1)!)^22^{2j-3}},\quad j\geq 2,$$ which immediately proves positivity. Indeed, according to Lemma 2.3 in \cite{Iksanov+Kotelnikova:2022}, with $a=b=1$ and $l=j-1$,
$$\sum_{i=0}^{j-2} \frac{\binom{i+j-2}{j-2}}{2^{i+j-2}}=1,\quad j\geq 2,$$ whence $$\sum_{i=0}^{j-1}\frac{\Gamma(i+j-1)}{i!(j-1)!2^{i+j-2}}-\frac{\Gamma(j-1)}{(j-1)!}=\frac{1}{j-1}\Big(\sum_{i=0}^{j-2} \frac{\binom{i+j-2}{j-2}}{2^{i+j-2}}+\frac{\binom{2j-3}{j-2}}{2^{2j-3}}-1\Big)=\frac{\binom{2j-3}{j-2}}{(j-1)2^{2j-3}}.$$ Assume now that $\alpha\in (0,1)$. Recalling a series representation of $\var K_j(t)$ (see \eqref{eq:formvar}) and using $$\eee^{-x}\frac{x^j}{j!}\le 1- \eee^{-x}\sum_{i=0}^{j-1}\frac{x^i}{i!},\quad x\geq 0,$$ we infer
$$\var K_j(t)\ge \sum_{k\ge 1} \eee^{-2p_kt}\frac{(p_kt)^j}{j!}=2^{-j}\me K^\ast_j(2t),\quad t\geq 0,\,j\in\mn.$$ Here, the equality follows from the second equality in \eqref{eq:sumindic}.
With this at hand, invoking now \eqref{eq:momtons} and \eqref{eq:inter1} we obtain $$c_j\ge 2^{\alpha-j}\frac{\alpha\Gamma(j-\alpha)}{j!}>0,\quad j\in\mn.$$
\end{proof}

\subsection{Proofs of Theorems \ref{thm:Karlin0}, \ref{thm:Karlin} and \ref{thm:Karlin1}}
Since $K_j(t)$ is the sum of independent indicators, Theorems \ref{thm:Karlin0}, \ref{thm:Karlin} and \ref{thm:Karlin1} will be proved by an application of Theorem \ref{thm:main}. As a preparation we prove a lemma designed to check conditions (B2.1) and (B2.2) of Theorem \ref{thm:main} in the present setting.
\begin{lemma}\label{lem:karl}
Assume that either $\rho\in \Pi_{\ell,\,\infty}$ or $\rho$ is regularly varying at $\infty$ of index $\alpha\in(0,1]$. If $\rho\in\Pi_{\ell,\,\infty}$ and $j\in\mn$ or $\rho$ is regularly varying of index $\alpha\in (0,1)$ and $j\in\mn$ or $\alpha=1$ and $j\geq 2$, then for any positive functions $c$ and $d$ satisfying $\lim_{t\to\infty} c(t)=\infty$, $\lim_{t\to\infty}(c(t)/t)=0$ and $\lim_{t\to\infty}(d(t)/t)=\infty$,
\begin{equation}\label{eq:cutvar}
{\rm Var}\Big(\sum_{k\ge 1} \1_{\{c(t)<1/p_k\leq d(t)\}}\1_{\{\pi_k(t)\geq j\}}\Big) ~\sim~ {\rm Var}\,K_j(t),\quad t\to\infty.
\end{equation}
If $\alpha=1$ and $j=1$, then for any fixed $\varsigma\in [0,1)$, any $\varsigma_1\in (0,1]$ and $\varsigma_2\in [0,1)$ satisfying $\varsigma_1-\varsigma_2=1-\varsigma$ and any positive functions $c$ and $d$ satisfying $\lim_{t\to\infty}(\hat L(c(t))/\hat L(t))=\varsigma_1$, $\lim_{t\to\infty}(c(t)/t)=\infty$ provided that $\varsigma_1=1$  and $\lim_{t\to\infty}(\hat L(d(t))/\hat L(t))=\varsigma_2$,
\begin{equation}\label{eq:cutvarvarsigma}
{\rm Var}\Big(\sum_{k\ge 1} \1_{\{c(t)<1/p_k\leq d(t)\}}\1_{\{\pi_k(t)\geq j\}}\Big) ~\sim~ (1-\varsigma) {\rm Var}\,K_j(t),\quad t\to\infty.
\end{equation}
\end{lemma}
\begin{rem}\label{rem:exist}
If $\varsigma_1=1$, the existence of $c$ appearing in the second part of Lemma \ref{lem:karl} is secured by Lemma \ref{lem:KarlB2}. If $\varsigma_1<1$, one may take $c(t):=\hat{L}^{-1}(\varsigma_1\hat{L}(t))$ for large $t$, where $\hat{L}^{-1}$ is the inverse function of $\hat L$ which exists because $\hat{L}$ is decreasing and continuous for large $t$. The same argument applies also to $d$ if $\varsigma_2>0$. If $\varsigma_2=0$, the existence of $d$ is clear.
\end{rem}
\begin{proof}
Assume first that $\rho$ is regularly varying of index $\alpha\in (0,1)$ and $j\in\mn$ or $\alpha=1$ and $j\geq 2$. In view of \eqref{eq:var} it is enough to show that
\begin{equation}\label{eq:a}
{\rm Var}\,\Big(\sum_{k\ge 1} \1_{\{1/p_k> d(t)\}}\1_{\{\pi_k(t)\geq j\}}\Big)= o(\rho(t)), \quad t\to\infty
\end{equation}
and
\begin{equation}\label{eq:b}
{\rm Var}\,\Big(\sum_{k\ge 1} \1_{\{1/p_k\leq c(t)\}}\1_{\{\pi_k(t)\geq j\}}\Big)=o(\rho(t)), \quad t\to\infty.
\end{equation}

Recalling the first equality in \eqref{eq:sumindic} we start by writing
\begin{multline*}
{\rm Var}\,\Big(\sum_{k\ge 1} \1_{\{1/p_k> d(t)\}}\1_{\{\pi_k(t)\geq j\}}\Big)=\sum_{k\ge 1} \1_{\{1/p_k> d(t)\}}\eee^{-p_kt}\sum_{i=0}^{j-1}\frac{(p_kt)^i}{i!}\Big(1-\eee^{-p_kt}\sum_{i=0}^{j-1}\frac{(p_kt)^i}{i!}\Big)\\=\int_{(d(t),\,\infty)}\eee^{-t/x}\sum_{i=0}^{j-1}\frac{(t/x)^i}{i!}\Big(1-\eee^{-t/x}
\sum_{i=0}^{j-1}\frac{(t/x)^i}{i!}\Big){\rm d}\rho(x).
\end{multline*}
An integration by parts in combination with \eqref{eq:addineq1} yields
\begin{multline}
{\rm Var}\,\Big(\sum_{k\ge 1} \1_{\{1/p_k> d(t)\}}\1_{\{\pi_k(t)\geq j\}}\Big) =-\eee^{-t/d(t)}\sum_{i=0}^{j-1}\frac{(t/d(t))^i}{i!}\Big(1-\eee^{-t/d(t)}\sum_{i=0}^{j-1}\frac{(t/d(t))^i}{i!}\Big) \rho(d(t))\\+\int_0^{t/d(t)} \rho(t/y) \eee^{-y}\frac{y^{j-1}}{(j-1)!}\Big(2\eee^{-y}\sum_{i=0}^{j-1}\frac{y^i}{i!}-1\Big){\rm d}y=:-A_1(t)+A_2(t).\label{eq:cutvarA}
\end{multline}
In view of \eqref{eq:addineq}, $$0\leq \frac{A_1(t)}{\rho(t)}\leq \Big(\frac{t}{d(t)}\Big)^j \frac{\rho(d(t))}{\rho(t)}~\sim~\Big(\frac{L(d(t))}{(d(t))^{j-\alpha}}\Big)/\Big(\frac{L(t)}{t^{j-\alpha}}\Big),\quad t\to\infty.$$ By Lemma \ref{lem:reg}(a), the right-hand side converges to $0$ as $t\to\infty$ because $j-\alpha>0$ and $\lim_{t\to\infty}t^{-1}d(t)=\infty$.

Now we intend to show that $\lim_{t\to\infty} (A_2(t)/\rho(t))=0$. As a preparation, observe that, for each $y>0$, $$\lim_{t\to\infty}\1_{(0,\,t/d(t))}(y)\frac{\rho(t/y)}{\rho(t)} \eee^{-y}\frac{y^{j-1}}{(j-1)!}\Big(2\eee^{-y}\sum_{i=0}^{j-1}\frac{y^i}{i!}-1\Big)=0.$$ It remains to find an integrable majorant for the function under the limit. Let $y^*$ be the unique root of the equation $2\eee^{-y}\sum_{i=0}^{j-1}\frac{y^i}{i!}=1$. According to Potter's bound (Theorem 1.5.6 in \cite{Bingham+Goldie+Teugels:1989}), given $\delta>0$ there exists $x_0>1$ such that, for $y\le \frac{t}{x_0}$ and $t\geq x_0$,
$$
\frac{\rho(t/y)}{\rho(t)}\le 2\max(y^{-\alpha-\delta},\, y^{-\alpha+\delta}).
$$
Increasing $t$ if needed we can ensure that, together with the previous inequality, we also have $t/d(t)\leq y^\ast\wedge 1$. With this at hand, choosing $\delta>0$ such that $\alpha+\delta<j$ we conclude that
$$y\mapsto \1_{(0,\,y^\ast\wedge 1)}(y) 2\eee^{-y}\frac{y^{j-\alpha-\delta-1}}{(j-1)!}\Big(2\eee^{-y}\sum_{i=0}^{j-1}\frac{y^i}{i!}-1\Big),\quad y>0$$
is an integrable majorant. By Lebesgue's dominated convergence theorem $A_2/\rho$ does indeed vanish. The proof of \eqref{eq:a} is complete.

Passing to the proof of \eqref{eq:b} we obtain
\begin{multline}
\var \sum_{k\ge 1} \1_{\{1/p_k\le c(t)\}}\1_{\{\pi_k(t)\geq j\}}\\= \eee^{-t/c(t)}\sum_{i=0}^{j-1}\frac{(t/c(t))^i}{i!}\Big(1-\eee^{-t/c(t)}\sum_{i=0}^{j-1}\frac{(t/c(t))^i}{i!}\Big) \rho(c(t))\\
-\int_{t/c(t)}^t \rho(t/y) \eee^{-y}\frac{y^{j-1}}{(j-1)!}\Big(1-2\eee^{-y}\sum_{i=0}^{j-1}\frac{y^i}{i!}\Big){\rm d}y=:B_1(t)-B_2(t)\label{eq:cutvarB}
\end{multline}
Since $\lim_{t\to\infty}(t/c(t))=\infty$, we infer $$\lim_{t\to\infty}\eee^{-t/c(t)}\sum_{i=0}^{j-1}\frac{(t/c(t))^i}{i!}\Big(1-\eee^{-t/c(t)}\sum_{i=0}^{j-1}\frac{(t/c(t))^i}{i!}\Big)=0,$$ whence $B_1(t)=o(\rho(c(t)))=o(\rho(t))$ as $t\to\infty$. We have used the fact that $\rho$ is nondecreasing for the last asymptotic relation. Using monotonicity of $\rho$ once again yields, for large $t$,
\begin{multline*}
0\leq B_2(t)\le \rho(c(t)) \int_{t/c(t)}^t \eee^{-y}\frac{y^{j-1}}{(j-1)!}\Big(1-2\eee^{-y}\sum_{i=0}^{j-1}\frac{y^i}{i!}\Big){\rm d} y = o(\rho(c(t)))=o(\rho(t)), \quad t\to\infty,
\end{multline*}
and \eqref{eq:b} follows.

Consider now the case $\alpha=1$ and $j=1$. Since $\varsigma_1>\varsigma_2$ and $\hat{L}$ is nonincreasing, the assumptions of the lemma guarantee that, for large enough $t$, $c(t)<d(t)$. Thus, the sum $\sum_{k\geq 1}\1_{\{c(t)<1/p_k\leq d(t)\}}$ is nonzero for large $t$. Observe that
\begin{equation}\label{eq:aux}
\lim_{t\to\infty}(c(t)/t)=\lim_{t\to\infty}(d(t)/t)=\infty
\end{equation}
and particularly $\lim_{t\to\infty} c(t)=\lim_{t\to\infty}d(t)=\infty$. Indeed, assume on the contrary that $d(t_n)/t_n\leq c$ for a constant $c>0$ and large $n$, where $(t_n)_{n\in\mn}$ is some sequence of positive numbers diverging to $\infty$. Since the function $\hat L$ is nonincreasing we conclude that $\hat L(d(t_n))/\hat L(t_n)\geq \hat L(ct_n)/\hat L(t_n)\to 1$ as $n\to\infty$, a contradiction. The claim about $c$ follows analogously if $\varsigma_1<1$ and holds true by assumption if $\varsigma_1=1$.

For simplicity of presentation we assume that $\rho(t)=tL(t)$ for $t>0$ for some slowly varying $L$ rather than $\rho(t)\sim tL(t)$ as $t\to\infty$. In view of \eqref{eq:varalone} it is sufficient to prove that
\begin{equation}\label{eq:b1}
{\rm Var}\,\Big(\sum_{k\ge 1} \1_{\{1/p_k> c(t)\}}\1_{\{\pi_k(t)\geq 1\}}\Big)~\sim~ \varsigma_1 t\hat L(t), \quad t\to\infty.
\end{equation}
and
\begin{equation}\label{eq:a1}
{\rm Var}\,\Big(\sum_{k\ge 1} \1_{\{1/p_k> d(t)\}}\1_{\{\pi_k(t)\geq 1\}}\Big)~\sim~ \varsigma_2 t\hat L(t), \quad t\to\infty.
\end{equation}
The right-hand side of \eqref{eq:a1} reads $=o(t\hat L(t))$ if $\varsigma_2=0$. We shall prove \eqref{eq:b1} and \eqref{eq:a1} simultaneously. Let $(h, \varsigma^\ast)$ denote either $(c, \varsigma_1)$ or $(d, \varsigma_2)$. Using the $A_1$ and $A_2$ defined in \eqref{eq:cutvarA}, putting $j=1$ and replacing $d$ with $h$ we obtain $$0\leq \frac{A_1(t)}{t\hat L(t)}=(\eee^{-t/h(t)}-\eee^{-2t/h(t)})\frac{\rho(h(t))}{t\hat L(t)}\leq \frac{L(h(t))}{\hat L(t)}=\frac{L(h(t))}{\hat L(h(t))}\frac{\hat L(h(t))}{\hat L(t)}\to 0,\quad t\to\infty.$$ Here, while the first factor on the right-hand side converges to $0$ according to \eqref{eq:rate}, the second converges to $\varsigma^\ast$. Further, for $t$ large enough to ensure $t/h(t)\leq \log 2$ (this is possible in view of \eqref{eq:aux}), $$0\leq \frac{A_2(t)}{t\hat L(t)}=\int_0^{t/h(t)} \frac{\rho(t/y)}{t\hat L(t)} \eee^{-y}(2\eee^{-y}-1){\rm d}y\leq \frac{1}{\hat L(t)}\int_0^{t/h(t)}y^{-1}L(t/y){\rm d}y=\frac{\hat L(h(t))}{\hat L(t)}.$$ Given $\varepsilon\in (0,1)$ there exists $\delta>0$ such that, whenever $t/h(t)\leq \delta$,
$$\frac{A_2(t)}{t\hat L(t)}=\int_0^{t/h(t)} \frac{\rho(t/y)}{t\hat L(t)} \eee^{-y}(2\eee^{-y}-1){\rm d}y\geq (1-\varepsilon)\int_0^{t/h(t)} \frac{\rho(t/y)}{t\hat L(t)}{\rm d}y=(1-\varepsilon)\frac{\hat L(h(t))}{\hat L(t)}.$$ Sending first $t\to\infty$ and then $\varepsilon\to 0+$ we obtain \eqref{eq:b1} and \eqref{eq:a1}.

Assume now that $\rho\in\Pi_{\ell,\,\infty}$ and $j\in\mn$. In view of \eqref{eq:var0} it is enough to show that
\begin{equation}\label{eq:a121}
{\rm Var}\,\Big(\sum_{k\ge 1} \1_{\{1/p_k> d(t)\}}\1_{\{\pi_k(t)\geq j\}}\Big)= o(\ell(t)), \quad t\to\infty
\end{equation}
and
\begin{equation}\label{eq:b121}
{\rm Var}\,\Big(\sum_{k\ge 1} \1_{\{1/p_k\leq c(t)\}}\1_{\{\pi_k(t)\geq j\}}\Big)=o(\ell(t)), \quad t\to\infty.
\end{equation}

We intend to represent the right-hand side of \eqref{eq:cutvarA} in a form which is more suitable for the present proof. Let $\xi_j$ and $\xi_j^\prime$ be independent random variables having a gamma distribution with parameters $j$ and $1$, that is,
\begin{equation}\label{eq:pxi}
\mmp\{\xi_j\leq x\}=1-\eee^{-x}\sum_{i=0}^{j-1}\frac{x^i}{i!}=\int_0^x \eee^{-y}\frac{y^{j-1}}{(j-1)!}{\rm d}y,\quad x\geq 0.
\end{equation}
Then
\begin{multline*}
\mmp\{\min(\xi_j, \xi_j^\prime)\leq x\}=1-\eee^{-2x}\Big(\sum_{i=0}^{j-1}\frac{x^i}{i!}\Big)^2=2\int_0^x \eee^{-2y}\frac{y^{j-1}}{(j-1)!}\sum_{i=0}^{j-1}\frac{y^i}{i!}{\rm d}y\\=\int_0^{2x} \eee^{-y}\frac{(y/2)^{j-1}}{(j-1)!}\sum_{i=0}^{j-1}\frac{(y/2)^i}{i!}{\rm d}y,\quad x\geq 0.
\end{multline*}
This proves that
\begin{multline*}
-\eee^{-t/d(t)}\sum_{i=0}^{j-1}\frac{(t/d(t))^i}{i!}\Big(1-\eee^{-t/d(t)}\sum_{i=0}^{j-1}\frac{(t/d(t))^i}{i!}\Big) =\int_0^{t/d(t)} \eee^{-y}\frac{y^{j-1}}{(j-1)!}{\rm d}y\\-\int_0^{2t/d(t)} \eee^{-y}\frac{(y/2)^{j-1}}{(j-1)!}\sum_{i=0}^{j-1}\frac{(y/2)^i}{i!}{\rm d}y.
\end{multline*}
Using this we obtain an alternative representation of \eqref{eq:cutvarA} that we were aimed at:
\begin{multline}\label{eq:A}
{\rm Var}\, \sum_{k\ge 1} \1_{\{1/p_k>d(t)\}}\1_{\{\pi_k(t)\geq j\}}=\int_0^{t/d(t)} (\rho(2t/y)-\rho(t/y))\eee^{-y}\frac{y^{j-1}}{(j-1)!} \dd y\\+\int_{t/d(t)}^{2t/d(t)} (\rho(2t/y)-\rho(d(t)))\eee^{-y}\frac{y^{j-1}}{(j-1)!} \dd y\\+\int_0^{2t/d(t)} (\rho(2t/y)-\rho(d(t))) \eee^{-y}\frac{y^{j-1}}{(j-1)!}\Big(2^{1-j} \sum_{i=0}^{j-1}\frac{(y/2)^i}{i!}-1\Big)\dd y=:C_1(t)+C_2(t)+C_3(t).
\end{multline}

Since $\rho\in \Pi_\ell$,
there exists $x_1>0$ such that, for all $y\le t/x_1$ and large $t$,
\begin{equation}\label{eq:assump}
\frac{\rho(2t/y)-\rho(t/y)}{\ell(t/y)}\le \log2+1.
\end{equation}
According to Potter's bound (Theorem 1.5.6 in \cite{Bingham+Goldie+Teugels:1989}), given $\delta\in (0,1)$ there exists $x_0>1$ such that, for $y\le t/x_0$ and large $t$,
\begin{equation}\label{eq:potter}
\ell(t/y)/\ell(t) \le 2\max(y^{-\delta}, y^{\delta}).
\end{equation}
For $t$ so large that $d(t)\geq x^\ast:=\max(x_0, x_1)$ and $t/d(t)\leq 1$, we infer
\begin{equation*}
\frac{C_1(t)}{\ell(t)}\leq 2(\log 2+1)\int_0^{t/d(t)} y^{-\delta}\eee^{-y}\frac{y^{j-1}}{(j-1)!} \dd y \to 0,\quad t\to\infty.
\end{equation*}

Recalling that $\rho$ is a nondecreasing function we conclude that,
\begin{multline*}
\frac{C_2(t)}{\ell(t)}\le \frac{\rho(2d(t))-\rho(d(t))}{\ell(d(t))}\frac{\ell(d(t))}{\ell(t)}\int_{t/d(t)}^{2t/d(t)} \eee^{-y}\frac{y^{j-1}}{(j-1)!}\dd y\\\leq \eee^{-(j-1)}\frac{(j-1)^{j-1}}{(j-1)!}\frac{\rho(2d(t))-\rho(d(t))}{\ell(d(t))}\frac{\ell(d(t))/d(t)}{\ell(t)/t}\to 0,\quad t\to\infty.
\end{multline*}
Here, while the last factor converges to $0$ by an application of Lemma \ref{lem:reg}(a) to $f$ defined by $f(t)=t^{-1}\ell(t)$, the penultimate factor converges to $\log 2$ because $\rho\in \Pi_l$. Finally, $C_3(t)\leq 0$ for large $t$ satisfying $2t/d(t)\leq 1$, as a consequence of
\begin{equation*}
2^{1-j}\sum_{i=0}^{j-1} \frac{(y/2)^i}{i!}\le 2^{1-j} \sum_{i=0}^{j-1} 2^{-i}\le 2^{2-j}\le 1,\quad y\in [0,1],\ j\geq 2
\end{equation*}
and, for $j=1$,
\begin{equation*}
2^{1-j}\sum_{i=0}^{j-1} \frac{(y/2)^i}{i!}=1.
\end{equation*}
The proof of \eqref{eq:a121} is complete.

Mimicking the argument leading to \eqref{eq:A} we obtain an alternative representation of formula \eqref{eq:cutvarB}:
\begin{multline*}
{\rm Var}\,\Big(\sum_{k\ge 1} \1_{\{1/p_k\leq c(t)\}}\1_{\{\pi_k(t)\geq j\}}\Big)=\int_{t/c(t)}^{2t/c(t)} (\rho(c(t))-\rho(t/y)) \eee^{-y}\frac{y^{j-1}}{(j-1)!}\dd y\\+\int_{2t/c(t)}^{2t}(\rho(2t/y)-\rho(t/y)) \eee^{-y}\frac{y^{j-1}}{(j-1)!}\dd y+\int_{2t/c(t)}^{2t} (\rho(2t/y)-\rho(c(t)))\eee^{-y}\frac{y^{j-1}}{(j-1)!}\Big(2^{1-j} \sum_{i=0}^{j-1}\frac{(y/2)^i}{i!}-1 \Big)\dd y\\+\eee^{-2t}\Big(\sum_{i=0}^{j-1}\frac{(2t)^i}{i!}-\Big(\sum_{i=0}^{j-1}\frac{t^i}{i!}\Big)^2\Big)\rho(c(t))=:\sum_{k=1}^4 D_k(t).
\end{multline*}

Using the fact that $\rho$ is nondecreasing we obtain
\begin{multline*}
\frac{D_1(t)}{\ell(t)}=\int_{t/c(t)}^{2t/c(t)} \frac{\rho(c(t))-\rho(t/y)}{\ell(c(t))}\frac{\ell(c(t))}{\ell(t)} \eee^{-y}\frac{y^{j-1}}{(j-1)!}\dd y\\\le
 \frac{\rho(c(t))-\rho(c(t)/2)}{\ell(c(t))}\frac{c(t)\ell(c(t))}{t\ell(t)}\frac{t}{c(t)}\mmp\{\xi_j>t/c(t)\}\to 0,\quad t\to\infty.
\end{multline*}
Here, the first factor on the right-hand side converges to $\log 2$ because $\rho\in \Pi_{\ell,\,\infty}$, the second factor vanishes by Lemma \ref{lem:reg}(b) with $f$ given by $f(t)=t\ell(t)$, and the remainder vanishes because the tail of a gamma distribution decays exponentially fast, see \eqref{eq:pxi}.

Further, for $t$ so large that $2t/c(t)\geq 1$ and $c(t)\geq x^\ast$ with $x^\ast$ defined right after \eqref{eq:potter},
$$\frac{D_2(t)}{\ell(t)}=\int_{2t/c(t)}^{2t}\frac{\rho(2t/y)-\rho(t/y)}{\ell(t/y)}\frac{\ell(t/y)}{\ell(t)}\eee^{-y}\frac{y^{j-1}}{(j-1)!}\dd y=\int_{2t/c(t)}^{t/x^\ast}\ldots+\int_{t/x^\ast}^{2t}\ldots=:D_{21}(t)+D_{22}(t).$$ An application of \eqref{eq:assump} and \eqref{eq:potter} yields
$$D_{21}(t)\le 2\big(\log 2+1\big)\int_{2t/c(t)}^{t/x^*} y^{\delta}\eee^{-y}\frac{y^{j-1}}{(j-1)!} \dd y\to 0,\quad t\to\infty.$$ Since $\rho$ is nondecreasing,
$$D_{22}(t)\le \frac{\rho(2x^*)}{\ell(t)}\int_{t/x^*}^\infty \eee^{-y} \frac{y^{j-1}}{(j-1)!} \dd y\to 0,\quad t\to\infty.$$

Observe now that $\rho(2t/y)-\rho(c(t))\leq 0$ under the integral in $D_3$ and that $2^{1-j}\sum_{i=0}^{j-1}\frac{(y/2)^i}{i!}\ge 1$ for $y>0$ if $j=1$ and for $y\geq 2^j$ if $j\geq 2$. This demonstrates that $D_3(t)\leq 0$ for large $t$ satisfying $2t/c(t)\geq 2^j$. Finally, $D_4(t)\leq 0$ for all $t>0$ as a consequence of $$\sum_{i=0}^{j-1}\frac{(2t)^i}{i!}-\Big(\sum_{i=0}^{j-1}\frac{t^i}{i!}\Big)^2=-\sum_{i=j}^{2(j-1)}\frac{t^i}{i!}\sum_{k=i-(j-1)}^{j-1}\binom{i}{k}\leq 0,\quad t>0.$$
The proof of \eqref{eq:b121} is complete.
\end{proof}

We shall check that, under the assumptions of Theorems \ref{thm:Karlin0}, \ref{thm:Karlin} and \ref{thm:Karlin1} (for the latter provided that either $j\geq 2$ or $j=1$ and
\eqref{eq:exotic} holds), all the conditions of Theorem \ref{thm:main} are satisfied. Also, we shall show that if under the assumptions of Theorem \ref{thm:Karlin1} relation \eqref{eq:exotic} fails to hold in the case $j=1$, then all the conditions of Proposition \ref{lilhalf1} are satisfied.

\noindent {\sc Condition (A1)} is secured by \eqref{eq:var0}, \eqref{eq:var} and \eqref{eq:varalone}, respectively.

\noindent {\sc Condition (A2)}. Invoking \eqref{eq:var0} in combination \eqref{eq:slowly} or \eqref{eq:slowly2} secures (A2) under the assumptions of Theorem \ref{thm:Karlin0}. By Theorem 1.5.3 in \cite{Bingham+Goldie+Teugels:1989}, formulae \eqref{eq:var} and \eqref{eq:varalone} ensure (A2) under the assumptions of Theorems \ref{thm:Karlin} and \ref{thm:Karlin1}, respectively.

\noindent {\sc Condition (A4)} holds trivially.

\noindent {\sc Conditions (A5) and (B1)}. Since the function $b$ given by $$b(t)=\me K_j(t)=\sum_{k\geq 1}(1-\eee^{-p_kt}\sum_{i=0}^{j-1}((p_kt)^i/(i!))),\quad t\geq 0$$ is strictly increasing and continuous, condition (A5) holds by a sufficient condition given in Remark \ref{suff}. Since the function $a$ given by $$a(t)=\var K_j(t)=\sum_{k\geq 1}\eee^{-p_kt}\sum_{i=0}^{j-1}\frac{(p_kt)^i}{i!}\Big(1-\eee^{-p_kt}\sum_{i=0}^{j-1}\frac{(p_kt)^i}{i!}\Big),\quad t\geq 0$$ is continuous, the first part of condition (B1) also holds.

\noindent {\sc Condition (A3)}. If in the setting of Theorem \ref{thm:Karlin0} relation \eqref{eq:slowly} prevails, then condition (A3) holds, with $\mu=1/\beta+1$, according to \eqref{eq:mean0}, \eqref{eq:var0} and \eqref{eq:rho}. If in the setting of Theorem \ref{thm:Karlin0} relation \eqref{eq:slowly2} prevails, then condition (A3) holds, with $\mu=1$ and $q=1/\lambda-1$, according to \eqref{eq:mean0}, \eqref{eq:var0} and \eqref{eq:rho2}. Condition (A3), with $\mu=1$ and $q=0$, is secured by \eqref{eq:momincreas} and \eqref{eq:var} in the setting of Theorem \ref{thm:Karlin}; by \eqref{eq:varalone} if $j=1$ and by \eqref{eq:momincreas} and \eqref{eq:var} if $j\ge 2$ in the setting of Theorem \ref{thm:Karlin1}.

Assume that in the setting of Theorem \ref{thm:Karlin1} condition \eqref{eq:exotic} fails to hold. Then Proposition \ref{lilhalf1} ensures \eqref{eq:LILkar200}. It will be explained in Remark \ref{rem:counterex} that, under failure of \eqref{eq:exotic}, relation \eqref{eq:LILkar200} is the best one can hope for as far as an application of Theorem \ref{thm:main} is concerned.

Now we shall prove that condition (B2.1)
holds in the case $j=1$ under the assumptions of Theorem \ref{thm:Karlin1} when \eqref{eq:exotic} prevails, and that condition (B2.2) holds under the assumptions of Theorems \ref{thm:Karlin0}, \ref{thm:Karlin} and \ref{thm:Karlin1} (the latter with $j\geq 2$).

\noindent {\sc Condition (B2.1)}. Assume that $\rho$ is regularly varying at $\infty$ of index $\alpha=1$ and $j=1$. For each $\varsigma\in (0,1)$, define the set $R_\varsigma(t)$ appearing in condition (B2.1)
by $$R_\varsigma(t)=\{k\in\mn: c(t)<1/p_k\leq d(t)\},\quad t\geq 1$$ with any positive functions $c$ and $d$ satisfying
$\lim_{t\to\infty}(\hat L(c(t))/\hat L(t))=1$, $\lim_{t\to\infty}(c(t)/t)=\infty$ and $\lim_{t\to\infty}(\hat L(d(t))/\hat L(t))=\varsigma$ (according to Remark \ref{rem:exist}, such functions $c$ and $d$ do exist). Then,~\eqref{eq:cutvarvarsigma} entails \eqref{eq:onevar}.

Recall that condition (A3) holds with $\mu=1$ and $q=0$, so that $w_n(\gamma, 1)=\exp(n^{1+\gamma})$. In view of \eqref{eq:varalone}, $\exp(n^{1+\gamma})/2\leq \tau_n\leq 2\exp(n^{1+\gamma})$ for large $n$. This together with \eqref{eq:exotic}, slow variation and monotonicity of $\hat L$ proves $\lim_{n\to\infty}(\hat L(\tau_{n+1})/\hat L(\tau_n))=0$. We intend to show that $d(\tau_n)<c(\tau_{n+1})$ for large $n$. Since $\hat L$ is nonincreasing, it suffices to check that, for large $n$,
\begin{equation}\label{eq:intermed}
\frac{\hat L(d(\tau_n))}{\hat L(\tau_n)}>\frac{\hat L(c(\tau_{n+1}))}{\hat L(\tau_{n+1})}\frac{\hat L(\tau_{n+1})}{\hat L(\tau_n)}.
\end{equation}
The latter holds true, for the left-hand side converges as $n\to\infty$ to $\varsigma$, whereas the right-hand side converges to $1\times 0=0$. Thus, there exists $n_0\in\mn$ such that the sets $R_\varsigma(\tau_{n_0})$, $R_\varsigma(\tau_{n_0+1}),\ldots$ are disjoint, and condition (B2.1)does indeed holds.

\noindent {\sc Condition (B2.2)}. Assume that either \eqref{eq:slowly} or \eqref{eq:slowly2} holds true (so that $\rho\in\Pi_{\ell,\,\infty}$) and $j\in\mn$, or $\rho$ is regularly varying at $\infty$ of index $\alpha\in (0,1)$ and $j\in\mn$, or $\rho$ is regularly varying of index $\alpha=1$ and $j\geq 2$. Put $c(t):=t/\log t$ and $d(t):=t\log t$ and define the set $R_0(t)$ appearing in condition (B2.2) by $$R_0(t)=\{k\in\mn: c(t)<1/p_k\leq d(t)\},\quad t\geq 1.$$ With this choice relation \eqref{eq:one} holds according to \eqref{eq:cutvar}.

If in the setting of Theorem \ref{thm:Karlin0} relation \eqref{eq:slowly} prevails then $\mu=1/\beta+1>1$, whence $w_n(\gamma, \mu)=n^{(1+\gamma)\beta}$. In view of \eqref{eq:var0} and \eqref{eq:slowly} $\tau_n$ is equal asymptotically to $\exp(n^{1+\gamma})$ times terms of smaller orders. If in the setting of Theorem \ref{thm:Karlin0} relation \eqref{eq:slowly2} prevails then $\mu=1$ and $q=1/\lambda-1>0$, whence $w_n(\gamma, 1)=\exp(n^{(1+\gamma)\lambda})$. In view of \eqref{eq:var0} and \eqref{eq:slowly2} $\tau_n$ is equal asymptotically to $\exp(\sigma^{-1/\lambda} n^{1+\gamma})$ times terms of smaller orders. In the settings of Theorems \ref{thm:Karlin} and \ref{thm:Karlin1} (the latter with $j\ge 2$) $\mu=1$ and $q=0$, so that $w_n(\gamma, 1)=\exp(n^{1+\gamma})$. According to \eqref{eq:var}, $\tau_n$ is equal asymptotically to $\exp(\alpha^{-1} n^{1+\gamma})$ times terms of smaller orders. Thus, in all the considered cases while $\lim_{n\to\infty}(\tau_{n+1}/\tau_n)=\infty$ superexponentially fast, $\log \tau_n$ exhibits a polynomial growth. As a consequence, the inequality $d(\tau_n)<c(\tau_{n+1})$ holds for large enough $n$, that is, there exists $n_0\in\mn$ such that the sets $R_0(\tau_{n_0})$, $R_0(\tau_{n_0+1}),\ldots$ are disjoint. We have shown that condition (B2.2)
holds true unless $\alpha=1$ and $j=1$.

The proofs of Theorems \ref{thm:Karlin0}, \ref{thm:Karlin} and \ref{thm:Karlin1} are complete.

\begin{rem}\label{rem:counterex}
Assume that $\rho(t)\sim tL(t)$ as $t\to\infty$ for some $L$ slowly varying at $\infty$ and that, for some $\gamma>0$ small enough,
\begin{equation}\label{eq:nonexot}
\liminf_{n\to\infty}\frac{\hat L(\exp(n+1)^{1+\gamma})}{\hat L(\exp(n^{1+\gamma}))}>0,
\end{equation}
that is, condition \eqref{eq:exotic} fails to hold. Then the second parts of conditions (B2.1) and (B2.2), with appropriately chosen $R_\varsigma(t)$ and $R_0(t)$, hold, whereas the first parts of conditions (B2.1) and (B2.2) fail to hold.

For each $\varsigma\in [0,1)$, define $$R_\varsigma(t)=\{k\in\mn: c(t)<1/p_k\leq d(t)\},\quad t\geq 1$$ with any positive functions $c$ and $d$ satisfying
$\lim_{t\to\infty}(\hat L(c(t))/\hat L(t))=\varsigma_1$, $\lim_{t\to\infty}(c(t)/t)=\infty$ provided that $\varsigma_1=1$ and $\lim_{t\to\infty}(\hat L(d(t))/\hat L(t))=\varsigma_2$, where $\varsigma_1\in (0,1]$ and $\varsigma_2\in [0,1)$ are any fixed numbers such that $\varsigma_1-\varsigma_2=1-\varsigma$. If $\varsigma=0$, then $\varsigma_1=1$ and $\varsigma_2=0$ and \eqref{eq:cutvarvarsigma}
entails \eqref{eq:one}. If $\varsigma\in (0,1)$, then \eqref{eq:cutvarvarsigma} secures \eqref{eq:onevar}.

Arguing as in the previous proof one can check that \eqref{eq:nonexot} guarantees that $$\liminf_{n\to\infty}\frac{\hat L(\tau_{n+1})}{\hat L(\tau_n)}>0.$$ We claim that inequality \eqref{eq:intermed} cannot hold for all $n$ large enough and neither can $d(\tau_n)<c(\tau_{n+1})$ because $\hat L$ is nonincreasing. This is obvious if $\varsigma=0$, for the limit of the left-hand side of \eqref{eq:intermed} is then $0$, whereas the limit inferior of the right-hand side is positive. If $\varsigma\in (0,1)$ is close to $0$, then $\varsigma_1$ is close to $1$ and $\varsigma_2$ is close to $0$. Hence, \eqref{eq:intermed} fails to hold for large $n$ in this case, too.
\end{rem}

\subsection{Proofs of Proposition \ref{equivvar} and Theorem \ref{thm:depoiss}}

We start with an auxiliary result which is an important ingredient of our proof
of Proposition \ref{equivvar}.
\begin{lemma}\label{cutsum}
Assume that either $\rho\in \Pi_{\ell,\,\infty}$ or $\rho$ is regularly varying at $\infty$ of index $\alpha\in(0,1]$.
Then for $l, j\ge 2$ and a positive constant $C_{\alpha,\,l,\,j}$,
\begin{equation*}
\lim_{n\to\infty}\frac{\sum_{i\ge 1}\binom{n}{l}p_i^l(1-p_i)^{n}}{\var K_j(n)}=C_{\alpha,\,l,\,j}.
\end{equation*}
\end{lemma}
\begin{proof}
Using \eqref{eq:var0} and \eqref{eq:meK*0} or \eqref{eq:var} and \eqref{eq:momtons} we infer, for $l\ge 2$ and a positive constant $C_{\alpha,\,l,\,j}$,
$$\lim_{n\to\infty} \frac{\me K_{l}^*(n+l)}{\var K_{j}(n)}=\lim_{n\to\infty} \frac{\me K_{l}^*(n+l)}{\var K_{j}(n+l)}=C_{\alpha,\,l,\,j}.$$
Here, the first equality is a consequence of $\lim_{n\to\infty}(\var K_{j}(n+l)/\var K_j(n))=1$, which is justified by the regular variation of $t\mapsto\var K_j(t)$.
According to Lemma 1 in \cite{Gnedin+Hansen+Pitman:2007},
for $l\in\mn$,
\begin{equation*}
\lim_{n\to\infty}|\me \mathcal{K}_l^*(n)-\me K_l^*(n)|=0
\end{equation*}
(the asymptotic relation holds true for arbitrary discrete probability distribution $(p_k)_{k\in\mn}$, not necessarily satisfying the regular variation assumption).
This entails $$\sum_{i\ge 1}\binom{n}{l}p_i^l(1-p_i)^{n}=\me \mathcal{K}_{l}^*(n+l) \binom{n}{l}\Big/ \binom{n+l}{l}~\sim~\me K_{l}^*(n+l),\quad n\to\infty.$$ The proof of Lemma \ref{cutsum} is complete.
\end{proof}

\begin{proof}[Proof of Proposition \ref{equivvar}]
If $j=1$, then, according to \eqref{eq:var0} and the assumption $\lim_{t\to\infty}\ell(t)=\infty$ or \eqref{eq:var}, or \eqref{eq:varalone}, $\lim_{n\to\infty}\var K_1(n)=\infty$. With this at hand, relation
\eqref{eq:ratio} follows from Lemma 4 in \cite{Gnedin+Hansen+Pitman:2007}, as has already been mentioned.

From now on we are concerned with the case $j\geq 2$ (however, some formulae that follow are valid for $j\in\mn$). Then, by \eqref{eq:var0} or \eqref{eq:var},
\begin{equation}\label{eq:zero_var}
\lim_{t\to\infty}\frac{\var K_j(t)}{t}=0.
\end{equation}
For $k,j,n\in\mn$, denote by $A_k(j,n)$ the event $\{\text{the box}~ k~\text{contains at least}~j~ \text{balls out of}~ n\}$. Then $$
\var \mathcal{K}_j(n)=\sum_{k\ge 1}\mmp(A_k(j,n))(1-\mmp(A_k(j,n)))+\sum_{i\neq k}(\mmp(A_i(j,n)\cap A_k(j,n))-\mmp(A_i(j,n))\mmp(A_k(j,n))).$$ We shall prove that
\begin{equation}\label{eq:variance}
\lim_{n\to\infty} \frac{\sum_{k\ge 1}\mmp(A_k(j,n))(1-\mmp(A_k(j,n)))-\var K_j(n)}{\var K_j(n)}=0
\end{equation}
and
\begin{equation}\label{eq:12}
\lim_{n\to\infty}\frac{\sum_{i\neq k}(\mmp(A_i(j,n)\cap A_k(j,n))-\mmp(A_i(j,n))\mmp(A_k(j,n)))}{\var K_j(n)}=0.
\end{equation}

\noindent {\sc Proof of \eqref{eq:variance}}. Note that $$\mmp(A_k(j,n))=1-\sum_{i=0}^{j-1}\binom{n}{i}p_k^i(1-p_k)^{n-i}$$ and $$\var K_j(n)=\sum_{k\ge 1}\sum_{i=0}^{j-1}\eee^{-p_kn}\frac{(p_kn)^i}{i!}\Big(1-\sum_{i=0}^{j-1}\eee^{-p_kn}\frac{(p_kn)^i}{i!}\Big).$$ Hence, the numerator in \eqref{eq:variance} is equal to
\begin{multline*}
\sum_{k\ge 1}\Big(\sum_{i=0}^{j-1}\Big(\binom{n}{i}p_k^i(1-p_k)^{n-i}-\eee^{-p_kn}\frac{(p_kn)^i}{i!}\Big)\\-\sum_{i=0}^{j-1}\sum_{m=0}^{j-1}\Big(\binom{n}{i}\binom{n}{m}p_k^{i+m}(1-p_k)^{2n-i-m}-\eee^{-2p_kn}\frac{(p_kn)^{i+m}}{i!m!}\Big)\Big).
\end{multline*}
The penultimate inequality in the proof of Lemma 2.13 in \cite{Iksanov+Kotelnikova:2022} states that, for large enough $n$ and all $i\in\mn_0$, $i\le n$,
\begin{equation}\label{eq:ineq1}
	-B_ip_k\le \binom{n}{i}p_k^i(1-p_k)^{n-i}-\eee^{-p_kn}\frac{(p_kn)^i}{i!}\le A_ip_k
\end{equation}
for some positive constants $A_i$ and $B_i$. Therefore,
$$
\sum_{k\ge 1} \sum_{i=0}^{j-1}\Big|\binom{n}{i}p_k^i(1-p_k)^{n-i}-\eee^{-p_kn}\frac{(p_kn)^i}{i!}\Big|\le \sum_{i=0}^{j-1} \max(A_i,B_i)=o(\var K_j(n)),\quad n\to\infty
$$
because $\lim_{n\to\infty}\var K_j(n)=\infty$. Further, write
\begin{multline*}
\eee^{-2p_kn}\frac{(p_kn)^{i+m}}{i!m!}-\binom{n}{i}\binom{n}{m}p_k^{i+m}(1-p_k)^{2n-i-m}\\=\binom{i+m}{i}2^{-(i+m)}\Big(\eee^{-2p_kn}\frac{(2p_kn)^{i+m}}{(i+m)!}-2^{i+m}\frac{(n!)^2}{(n-i)!(n-m)!(i+m)!}p_k^{i+m}(1-p_k)^{2n-i-m}
\Big)\\=\binom{i+m}{i}2^{-(i+m)}\Big(\eee^{-2p_kn}\frac{(2p_kn)^{i+m}}{(i+m)!}-\binom{2n}{i+m}p_k^{i+m}(1-p_k)^{2n-i-m}\Big)\\+\binom{i+m}{i}2^{-(i+m)}
\Big(\binom{2n}{i+m}-2^{i+m}\frac{(n!)^2}{(n-i)!(n-m)!(i+m)!}\Big)p_k^{i+m}(1-p_k)^{2n-i-m}\\=:D_{i,m,k}(n)+E_{i,m,k}(n).
\end{multline*}
By \eqref{eq:ineq1},
$$
\sum_{k\ge 1}\sum_{i=0}^{j-1}\sum_{m=0}^{j-1}|D_{i,m, k}(n)|\le \sum_{i=0}^{j-1}\sum_{m=0}^{j-1}\binom{i+m}{i}2^{-(i+m)}\max(A_{i+m},B_{i+m})=o(\var K_j(n)),\quad n\to\infty.
$$
Observe now that $E_{0,0, k}(n)=E_{0,1, k}(n)=E_{1,0, k}(n)=0$. For $i$ and $m$ such that $i+m\ge 2$, $\binom{2n}{i+m}-2^{i+m}\frac{(n!)^2}{(n-i)!(n-m)!(i+m)!}$ is a polynomial of order at most $i+m$. Since the coefficient in front of $n^{i+m}$ is $0$, we infer
$$\Big|\binom{2n}{i+m}-2^{i+m}\frac{(n!)^2}{(n-i)!(n-m)!(i+m)!}\Big|\le C_{i,m} n^{i+m-1}$$
for a positive constant $C_{i,m}$. This in combination with Lemma \ref{cutsum} (put there $l=i+m$ and replace $n$ by $2n-i-m$) and \eqref{eq:zero_var} yields, for $i,m\leq j-1$ such that $i+m\geq 2$,
$$
\sum_{k\ge 1}|E_{i,m, k}(n)|\le C_{i,m}\binom{i+m}{i}2^{-(i+m)}\sum_{k\ge 1}n^{i+m-1}p_k^{i+m}(1-p_k)^{2n-i-m}~\sim~ {\rm const} \frac{\var K_j(n)}{n}\to 0
$$
as $n\to\infty$. Hence, $\lim_{n\to\infty}\sum_{k\ge 1}\sum_{i=0}^{j-1}\sum_{m=0}^{j-1}|E_{i,m, k}(n)|=0$. The proof of \eqref{eq:variance} is complete.

\noindent {\sc Proof of \eqref{eq:12}}. We start by writing
\begin{multline*}
\mmp(A_i(j,n)\cap A_k(j,n))-\mmp(A_i(j,n))\mmp(A_k(j,n))\\=\mmp((A_i(j,n))^c\cap (A_k(j,n))^c)-\mmp((A_i(j,n))^c)\mmp((A_k(j,n))^c)
\\=\sum_{a=0}^{j-1}\sum_{b=0}^{j-1}\Big(\binom{n}{a}\binom{n-a}{b}p_i^ap_k^b(1-p_i-p_k)^{n-a-b}-\binom{n}{a}\binom{n}{b}p_i^ap_k^b(1-p_i)^{n-a}(1-p_k)^{n-b}\Big)\\
=:\sum_{a=0}^{j-1}\sum_{b=0}^{j-1}C_{a,b}(i,k,n),
\end{multline*}
where $X^c$ denotes the complement of a set $X$. Further, we obtain
\begin{multline*}
C_{a,b}(i,k,n)=\binom{n}{a}\binom{n-a}{b}p_i^ap_k^b\Big((1-p_i-p_k)^{n-a-b}-(1-p_i)^{n-a}(1-p_k)^{n-b}\Big)\\-\binom{n}{a}\Big(\binom{n}{b}-\binom{n-a}{b}\Big)p_i^ap_k^b(1-p_i)^{n-a}(1-p_k)^{n-b}=:
C^{(1)}_{a,b}(i,k,n)+C^{(2)}_{a,b}(i,k,n).
\end{multline*}
To analyze $C^{(1)}_{a,b}$ we argue as on p.~152 in \cite{Gnedin+Hansen+Pitman:2007}. Using an expansion $$(x-y)^m=x^m-mx^{m-1}y+O(m^2x^{m-2}y^2),\quad m\to\infty,$$ which holds for positive $x$ and $y$, $x>y$, with $x=(1-p_i)(1-p_k)$, $y=p_ip_k$ and $m=n-a-b$ for $0\le a,b\le j-1$ yields
\begin{multline*}
(1-p_i-p_k)^{n-a-b}=(1-p_i)^{n-a-b}(1-p_k)^{n-a-b}-(n-a-b)p_ip_k (1-p_i)^{n-a-b-1}(1-p_k)^{n-a-b-1}\\+O(n^2p_i^2 p_k^2(1-p_i)^{n-a-b-2}(1-p_k)^{n-a-b-2}),\quad n\to\infty.
\end{multline*}
Hence,
\begin{multline*}
C^{(1)}_{a,b}(i,k,n)=\binom{n}{a}\binom{n-a}{b}p_i^ap_k^b\Big((1-p_i)^{n-a-b}(1-p_k)^{n-a-b}\Big(1-(1-p_i)^{b}(1-p_k)^{a}
\Big)\\-(n-a-b)p_ip_k(1-p_i)^{n-a-b-1}(1-p_k)^{n-a-b-1}+O(n^2p_i^2p_k^2(1-p_i)^{n-a-b-2}(1-p_k)^{n-a-b-2})\Big)\\=: F_{a,b}(i,k,n)+G_{a,b}(i,k,n)+H_{a,b}(i,k,n).
\end{multline*}
Now we investigate the contributions of the summands to $\sum_{i\neq k}\sum_{a=0}^{j-1}\sum_{b=0}^{j-1}C_{a,b}(i,k,n)$.

\noindent {\sc Analysis of $H_{a,b}$}:
\begin{multline*}
\sum_{i\ne k}\binom{n}{a}\binom{n-a}{b}n^2p_i^{a+2}p_k^{b+2}(1-p_i)^{n-a-b-2}(1-p_k)^{n-a-b-2}\\\le \sum_{i\ge 1}n\binom{n}{a}p_i^{a+2}(1-p_i)^{n-a-b-2}\sum_{k\ge 1}n\binom{n-a}{b}p_k^{b+2}(1-p_k)^{n-a-b-2}\\~\sim~ {\rm const}\,\frac{\big(\var K_j(n)\big)^2}{n^2}~\to~ 0,\quad n\to\infty.
\end{multline*}
Here, ${\rm const}$ denotes a positive constant, the asymptotic equivalence follows from Lemma \ref{cutsum}, and the last limit relation is justified by \eqref{eq:zero_var}.

\noindent {\sc Analysis of $F_{a,b}$}. Assume first that $a\ge 2$ and $b\ge 2$. An application of Bernoulli's inequality $1-(1-x)^m\leq mx$, which holds for $m\in\mn$ and $x\in [0,1]$, yields
\begin{multline*}
F_{a,b}(i,k,n):= \binom{n}{a}\binom{n-a}{b}p_i^ap_k^b(1-p_i)^{n-a-b}(1-p_k)^{n-a-b}\Big(1-(1-p_i)^{b}(1-p_k)^{a}\Big)\\\le
b\binom{n}{a}\binom{n-a}{b}p_i^{a+1}p_k^b(1-p_i)^{n-a-b}(1-p_k)^{n-a-b}+a\binom{n}{a}\binom{n-a}{b}p_i^ap_k^{b+1}(1-p_i)^{n-a-b}(1-p_k)^{n-a-b}\\=:F^{(1)}_{a,b}(i,k,n)+F^{(2)}_{a,b}(i,k,n).
\end{multline*}
By Lemma \ref{cutsum} and \eqref{eq:zero_var},
\begin{multline*}
0\le\sum_{i\neq k}F^{(1)}_{a,b}(i,k,n)\le b\sum_{i\ge 1} \binom{n}{a}p_i^{a+1}(1-p_i)^{n-a-b}\sum_{k\ge1}\binom{n-a}{b}p_k^b(1-p_k)^{n-a-b}\\~\sim~ \text{\rm const}\, \frac{\var K_j(n)}{n} \var K_j(n)=o(\var K_j(n)),\quad n\to\infty.
\end{multline*}
The argument for $F^{(2)}_{a,b}(i,k,n)$ is analogous.

Consider now $F_{a,b}(i,k,n)$ with $a\le 1$ or $b\le 1$. Notice that $F_{0,0}(i,k,n)=0$. Further, $F_{0,1}(i,k,n)
=F_{1,0}(i,k,n)=np_ip_k(1-p_i)^{n-1}(1-p_k)^{n-1}$, whence $F_{0,1}(i,k,n)+G_{0,0}(i,k,n)=0$ and $F_{0,1}(i,k,n)+C^{(2)}_{1,1}(i,k,n)=0$.

For $b\ge2$, using Pascal's rule we obtain
\begin{multline*}
F_{0,b}(i,k,n)+C^{(2)}_{1,b}(i,k,n)\\=\binom{n}{b}p_k^b(1-p_i)^{n-b}(1-p_k)^{n-b}\Big(1-(1-p_i)^{b}\Big)-n\Big(\binom{n}{b}
-\binom{n-1}{b}\Big)p_ip_k^b(1-p_i)^{n-1}(1-p_k)^{n-b}\\=\binom{n}{b}p_k^b(1-p_i)^{n-b}(1-p_k)^{n-b}\Big(1-(1-p_i)^{b}
-bp_i(1-p_i)^{b-1}\Big).
\end{multline*}
Since, for $x\in[0,1]$ and $n\in\mn$, $(1-x)^n+nx(1-x)^{n-1}\le 1$, we conclude that the last expression is nonnegative. Invoking Bernoulli's inequality twice and then Lemma \ref{cutsum} and \eqref{eq:zero_var} we infer
\begin{multline*}
	0\le \sum_{i\ne k} (F_{0,b}(i,k,n)+C^{(2)}_{1,b}(i,k,n))\le \sum_{i\ne k} b\binom{n}{b}p_ip_k^b(1-p_i)^{n-b}(1-p_k)^{n-b}\Big(1-(1-p_i)^{b-1}\Big)\\\le \sum_{i\ne k} b(b-1)\binom{n}{b}p_i^2p_k^b(1-p_i)^{n-b}(1-p_k)^{n-b}\le b(b-1) \sum_{i\ge 1} p_i^2(1-p_i)^{n-b}\sum_{k\ge 1}\binom{n}{b}p_k^b(1-p_k)^{n-b}\\~\sim~ {\rm const}\, \frac{(\var K_j(n))^2}{n^2}\to 0,\quad n\to\infty.
\end{multline*}
Analogously, for $a\ge 2$,
$$\lim_{n\to\infty} \sum_{i\ne k} (F_{a,0}(i,k,n)+C^{(2)}_{a,1}(i,k,n))=0.$$

Consider now $F_{1,1}(i,k,n)=n(n-1)p_ip_k(1-p_i)^{n-2}(1-p_k)^{n-2}(p_i+p_k-p_ip_k)$. The first summand of $F_{1,1}$ satisfies $n(n-1)p_ip_k(1-p_i)^{n-2}(1-p_k)^{n-2}(p_i+p_k)+G_{0,1}(i,k,n)+G_{1,0}(i,k,n)=0$. By Lemma \ref{cutsum} and \eqref{eq:zero_var}, the second summand of $F_{1,1}$ satisfies
\begin{multline*}
	\sum_{i\ne k}n(n-1)p_i^2p_k^2(1-p_i)^{n-2}(1-p_k)^{n-2}\le \sum_{i\ge 1} np_i^2(1-p_i)^{n-2} \sum_{k\ge 1} (n-1)p_k^2(1-p_k)^{n-2}\\~\sim~ {\rm const}\, \frac{(\var K_j(n))^2}{n^2}\to0,\quad n\to\infty.
\end{multline*}
For $b\ge 2$,
\begin{multline*}
F_{1,b}(i,k,n)+G_{0,b}(i,k,n)=	n\binom{n-1}{b}p_ip_k^b(1-p_i)^{n-1-b}(1-p_k)^{n-1-b}\Big(1-(1-p_i)^{b}(1-p_k)\Big)\\-(n-b)\binom{n}{b}p_ip_k^{b+1}
(1-p_i)^{n-1-b}(1-p_k)^{n-1-b}\\=n\binom{n-1}{b}p_ip_k^b(1-p_i)^{n-1-b}(1-p_k)^{n-1-b}\Big(1-(1-p_i)^{b}(1-p_k)-p_k\Big).
\end{multline*}
Since $(1-p_i)^{b}(1-p_k)+p_k\le 1$, the latter expression is nonnegative. By Bernoulli's inequality, Lemma \ref{cutsum} and \eqref{eq:zero_var},
\begin{multline*}
0\le \sum_{i\ne k}(F_{1,b}(i,k,n)+G_{0,b}(i,k,n))\le \sum_{i\ne k} bn\binom{n-1}{b}p_i^2p_k^b(1-p_i)^{n-1-b}(1-p_k)^{n-1-b}\\\le b\sum_{i\ge 1}np_i^2(1-p_i)^{n-1-b}\sum_{k\ge 1}\binom{n-1}{b}p_k^b(1-p_k)^{n-1-b}\\~\sim~ {\rm const}\, \frac{\var K_j(n)}{n} \var K_j(n)=o(\var K_j(n)),\quad n\to\infty.
\end{multline*}
Analogously, for $a\ge 2$,
$$\sum_{i\ne k} (F_{a,1}(i,k,n)+G_{a,0}(i,k,n))=o(\var K_j(n)),\quad n\to\infty.$$

\noindent {\sc Analysis of $G_{a,b}$ (cont)}. We have not investigated $G_{a,b}(i,k,n)$ with $a,b\geq 1$ so far. For such $a,b$, by Lemma \ref{cutsum} and \eqref{eq:zero_var},
\begin{multline*}
\sum_{i\ne k}|G_{a,b}(i,k,n)|=\sum_{i\ne k}(n-a-b)\binom{n}{a}\binom{n-a}{b}p_i^{a+1}p_k^{b+1}(1-p_i)^{n-a-b-1}(1-p_k)^{n-a-b-1}\\\le \sum_{i\ge 1}(n-a-b)\binom{n}{a}p_i^{a+1}
(1-p_i)^{n-a-b-1}\sum_{k\ge 1} \binom{n-a}{b}p_k^{b+1}(1-p_k)^{n-a-b-1}\\~\sim~ {\rm const}\, \var K_j(n) \frac{\var K_j(n)}{n}=o(\var K_j(n)),\quad n\to\infty.
\end{multline*}

\noindent {\sc Analysis of $C^{(2)}_{a,b}$ (cont)}. Observe that $C^{(2)}_{a,b}(i,k,n)=0$ whenever $a=0$ or $b=0$. Thus, we are left with analyzing $C^{(2)}_{a,b}(i,k,n)=0$ for $a,b\geq 2$. Notice that $\binom{n}{b}-\binom{n-a}{b}=O(n^{b-1})$ as $n\to\infty$. Hence, for some constant $C>0$,
\begin{multline*}
\sum_{i\ne k} |C^{(2)}_{a,b}(i,k,n)|
=\sum_{i\ne k}\binom{n}{a}\Big(\binom{n}{b}-\binom{n-a}{b}\Big)p_i^ap_k^b(1-p_i)^{n-a}(1-p_k)^{n-b}\\\le \sum_{i\ge 1}\binom{n}{a}p_i^a(1-p_i)^{n-a}C\sum_{k\ge 1}n^{b-1}p_k^b(1-p_k)^{n-b}~\sim~{\rm const}\, \var K_j(n) \frac{\var K_j(n)}{n}=o(\var K_j(n))
\end{multline*}
as $n\to\infty$. We have used Lemma \ref{cutsum} and \eqref{eq:zero_var} for the asymptotic relations.

Combining all the fragments together we arrive at \eqref{eq:12}. The proof of Proposition \ref{equivvar} is complete.
\end{proof}

With Proposition \ref{equivvar} at hand, we are ready to prove the LIL stated in Theorem \ref{thm:depoiss}.
\begin{proof}[Proof of Theorem \ref{thm:depoiss}.]
We divide the proof into several steps.

\noindent {\sc Step 1}. Under the present assumptions, for $j\in\mn$, $|\me\mathcal{K}_j(n)-\me K_j(n)|\to 0$ and $\var \mathcal{K}_j(n)\sim \var K_j(n)$ as $n\to\infty$ by \eqref{eq:means} and formula \eqref{eq:ratio} of Proposition \ref{equivvar}.

\noindent {\sc Step 2}. Recall that, for $j,n\in\mn$, $\mathcal{K}_j(n)=K_{j}(S_n)$. We are going to prove that
$$\lim_{n\to\infty}\frac{K_{j}(S_n)-K_j(n)}{(\var K_j(n)f(\var K_j(n)))^{1/2}}=0\quad\text{a.s.},$$
where $f(t)=\log t$ under \eqref{eq:deHaan} and \eqref{eq:slowly} and $f(t)=\log\log t$ under the other assumptions of Theorem \ref{thm:depoiss}. This is equivalent to showing that, for all $\varepsilon>0$ and large enough $n_0\in\mn$, $\sum_{n\ge n_0}\1_{A_n(\varepsilon)}<\infty$ a.s., where $$A_n=A_n(\varepsilon):=\Big\{\frac{|K_{j}(S_n)-K_j(n)|}{(\var K_j(n)f(\var K_j(n)))^{1/2}}>\varepsilon\Big\}$$ (the value of $n_0$ is chosen large enough to ensure that $n_0\geq 3$ and $f(\var K_j(n))>0$ for all $n\geq n_0$). For any $\delta>0$ and integer $n\geq 3$, consider the event
$$B_n=B_n(\delta):=\Big\{\frac{|S_n-n|}{(2n\log\log n)^{1/2}}\le 1+\delta\Big\}.$$
Write $$\1_{A_n}=\1_{A_n\cap B_n}+\1_{A_n\cap B_n^c} \le \1_{A_n\cap B_n}+\1_{B_n^c}\quad \text{a.s.},$$
where, as usual, $B_n^c$ is the complement of $B_n$.
According to the LIL for the standard random walk $(S_n)_{n\in\mn}$,
$\sum_{n\ge 3}\1_{B_n^c}<\infty$ a.s. Put $d_n:=(1+\delta)\sqrt{2n\log\log n}$ for $n\geq 3$. We are left with proving that $\sum_{n\ge n_0}\1_{A_n\cap B_n}<\infty$ a.s.
To this end, it is enough to check that $\sum_{n\ge n_0}\mmp(A_n\cap B_n)<\infty$. Observe that
\begin{multline*}
\sum_{n\ge n_0}\mmp(A_n\cap B_n)\le \sum_{n\ge n_0}\mmp\Big\{\frac{K_{j}(n+d_n)-K_j(n)}{(\var K_j(n)f(\var K_j(n)))^{1/2}}>\varepsilon\Big\}\\+\sum_{n\ge n_0}\mmp\Big\{\frac{K_{j}(n)-K_{j}(n-d_n)}{(\var K_j(n)f(\var K_j(n)))^{1/2}}>\varepsilon\Big\}=:I(\varepsilon)+J(\varepsilon).
\end{multline*}
We shall only show that $I(\varepsilon)<\infty$, the proof for $J(\varepsilon)$ being analogous. For $j\in\mn$ and $t\ge 0$, put
$b_j(t):=\me K_j(t)$ and $a_j(t):=\var K_j(t)$. Since, for each $t\geq 0$, $$\sum_{k\ge 1} \eee^{-p_kt}\frac{(p_kt)^{j-1}}{(j-1)!}p_k\leq \sum_{k\ge 1}p_k=1,$$ the series $\sum_{k\ge 1} \eee^{-p_kt}\frac{(p_kt)^{j-1}}{(j-1)!}p_k$ converges uniformly in $t\geq 0$. This fact enables us to differentiate the series termwise:
$$
(b_j(t))^\prime=\sum_{k\ge 1}\Big(1-\eee^{-p_kt}\sum_{i=0}^{j-1}\frac{(p_kt)^i}{i!}\Big)^\prime=\sum_{k\ge 1} \eee^{-p_kt}\frac{(p_kt)^{j-1}}{(j-1)!}p_k=\frac{j\me K^*_j(t)}{t}.
$$
By the mean value theorem for differentiable functions, for $n\geq 3$ and some $\theta_n\in[n,n+d_n]$, $$b_j(n+d_n)-b_j(n)=d_nb_j^\prime(\theta_n).$$
Under the present assumptions, the function $t\mapsto\me K^*_j(t)$, hence also $b_j^\prime$, is regularly varying at $\infty$, see \eqref{eq:meK*0}, \eqref{eq:momtons} and \eqref{eq:varalone}. Since $\lim_{n\to\infty}(\theta_n/n)=1$ we conclude that $b_j^\prime(\theta_n)\sim b_j^\prime(n)$ as $n\to\infty$ by the uniform convergence theorem for regularly varying functions (Theorem 1.5.2 in \cite{Bingham+Goldie+Teugels:1989}) and thereupon
\begin{equation}\label{eq:asymp_a}
b_j(n+d_n)-b_j(n)\sim d_n \frac{j\me K^*_j(n)}{n},\quad n\to\infty.
\end{equation}

Under the assumptions of Theorem \ref{thm:Karlin0}, we infer $$\frac{d_n}{(a_j(n)f(a_j(n)))^{1/2}} \frac{j\me K^*_j(n)}{n}~\sim~ {\rm const}\, \Big(\frac{\ell(n)}{f(\ell(n))}\Big)^{1/2}\Big(\frac{\log\log n}{n}\Big)^{1/2}~\to~ 0 ,\quad n\to\infty$$
having utilized \eqref{eq:var0} and \eqref{eq:meK*0}. Under the assumptions of Theorem \ref{thm:Karlin}, according to \eqref{eq:var} and \eqref{eq:momtons},
\begin{multline}\label{eq:dep2}
\frac{d_n}{(a_j(n)f(a_j(n)))^{1/2}}\frac{j\me K^*_j(n)}{n}=\frac{d_n}{(a_j(n)\log\log a_j(n))^{1/2}}\frac{j\me K^*_j(n)}{n}\\~\sim~ {\rm const}\,n^{(\alpha-1)/2}(L(n))^{1/2}~\to~ 0,\quad n\to\infty
\end{multline}
because $\alpha\in (0,1)$. Under the assumptions of Theorem \ref{thm:Karlin1}, relation \eqref{eq:dep2} still holds whenever $j\geq 2$. Indeed, formula \eqref{eq:raterho} then ensures that $\lim_{n\to\infty} L(n)=0$. For $j=1$, we obtain by virtue of \eqref{eq:varalone} that, as $n\to\infty$,
$$\frac{d_n}{(a_j(n)f(a_j(n)))^{1/2}}\frac{j\me K^*_j(n)}{n}=\frac{d_n}{(a_j(n)\log\log a_j(n))^{1/2}}\frac{j\me K^*_j(n)}{n}~\sim~ {\rm const }\, (\hat{L}(n))^{1/2}~\to~ 0
$$ because $\lim_{t\to\infty} \hat{L}(t)=0$ (see the paragraph preceding Lemma \ref{lem:alone}).

Thus, we have proved that, under the present assumptions, $$\lim_{n\to\infty}\frac{b_j(n+d_n)-b_j(n)}{(a_j(n)f(a_j(n)))^{1/2}}=0.$$
Now we explain that the inequalities $I(\varepsilon)<\infty$ for all $\varepsilon>0$ are secured by the inequalities $$K(\varepsilon):=\sum_{n\ge n_0}\mmp\Big\{\frac{K_{j}(n+d_n)-K_j(n)-b_j(n+d_n)+b_j(n)}{(a_j(n)f(a_j(n)))^{1/2}}>\varepsilon\Big\}<\infty$$
for all $\varepsilon>0$. Indeed, there exists a positive integer $M$ such that $\1_{\big\{\frac{b_j(n+d_n)-b_j(n)}{(a_j(n)f(a_j(n)))^{1/2}}>\varepsilon\big\}}=0$ for $n\geq M+1$. Hence, for all $\varepsilon>0$,
\begin{multline*}
I(2\varepsilon)=\sum_{n\ge n_0} \mmp\Big\{\frac{K_j(n+d_n)-K_j(n)}{(a_j(n)f(a_j(n)))^{1/2}}>2\varepsilon, \frac{b_j(n+d_n)-b_j(n)}{(a_j(n)f(a_j(n)))^{1/2}}>\varepsilon\Big\}\\+\sum_{n\ge n_0} \mmp\Big\{\frac{K_{j}(n+d_n)-K_j(n)}{(a_j(n)f(a_j(n)))^{1/2}}>2\varepsilon, \frac{b_j(n+d_n)-b_j(n)}{(a_j(n)f(a_j(n)))^{1/2}}\le\varepsilon\Big\}\\\le \sum_{n\ge n_0}\1_{\big\{ \frac{b_j(n+d_n)-b_j(n)}{(a_j(n)f(a_j(n)))^{1/2}}>\varepsilon\big\}}\\+\sum_{n\ge n_0} \mmp\Big\{\frac{K_{j}(n+d_n)-K_j(n)}{(a_j(n)f(a_j(n)))^{1/2}}>\varepsilon+ \frac{b_j(n+d_n)-b_j(n)}{(a_j(n)f(a_j(n)))^{1/2}}\Big\}\leq M+K(\varepsilon).
\end{multline*}
To prove that $K(\varepsilon)<\infty$, note that
\begin{multline*}
	X_j^*(n):=K_{j}(n+d_n)-K_j(n)-b_j(n+d_n)+b_j(n)\\=\sum_{k\ge 1}\Big(\1_{\{\pi_k(n)<j,\,\pi_k(n+d_n)\ge j\}}-\me\1_{\{\pi_k(n)<j,\,\pi_k(n+d_n)\ge j\}}\Big).
\end{multline*}
An application of \eqref{eq:ss1} yields, for $\theta \in\mr$,
\begin{multline}\label{eq:depois}
	\me \exp\Big(\frac{\theta X_j^*(n)}{(a_j(n)f( a_j(n)))^{1/2}}\Big)\le \exp\Big(\frac{\theta^2\var X_j^*(n)}{2a_j(n)f( a_j(n))}\eee^{\frac{|\theta|}{(a_j(n)f( a_j(n)))^{1/2}}}\Big)\\\le \exp\Big(\frac{\theta^2(b_j(n+d_n)-b_j(n))}{2a_j(n)f( a_j(n))}\eee^{\frac{|\theta|}{(a_j(n)f( a_j(n)))^{1/2}}}\Big).
\end{multline}
Here, we have used $\var  X_j^*(n)\le b_j(n+d_n)-b_j(n)$ for $n\geq 3$.

Under the assumptions of Theorem \ref{thm:Karlin0}, pick $\theta:=(\log n)^{1+c/2}$ with $c:=\beta/2$ in the setting of \eqref{eq:slowly} and $c:=\lambda/2$ in the setting of \eqref{eq:slowly2}. By Markov's inequality and \eqref{eq:depois}, $$K(\varepsilon)\le \sum_{n\ge n_0}\eee^{-\varepsilon (\log n)^{1+c/2}}\exp\Big(\frac{(\log n)^{2+c}(b_j(n+d_n)-b_j(n))}{2a_j(n)f( a_j(n))}\eee^{\frac{(\log n)^{1+c/2}}{(a_j(n)f( a_j(n)))^{1/2}}}\Big).$$
Recalling \eqref{eq:var0}, \eqref{eq:meK*0} and \eqref{eq:asymp_a}, for any $\Delta\in(0,1/2)$, $$\frac{(\log n)^{2+c}(b_j(n+d_n)-b_j(n))}{2a_j(n)f( a_j(n))}~\sim~ {\rm const}\, \frac{(\log n)^{2+c}(\log\log n)^{1/2}}{n^{1/2}f(\ell(n))}=o(n^{-1/2+\Delta}),\quad n\to\infty.$$ Under \eqref{eq:slowly} $\exp\big(\frac{(\log n)^{1+c/2}}{(a_j(n)f( a_j(n)))^{1/2}}\big)$ diverges to $\infty$ as $n\to\infty$ more slowly than any power, whereas under \eqref{eq:slowly2} it converges to $1$. Hence, there exists a positive integer $N$ such that $$K(\varepsilon)\le N+\eee\sum_{n\ge N} \eee^{-\varepsilon (\log n)^{1+c/2}}<\infty.$$

Under the assumptions of Theorems \ref{thm:Karlin} or \ref{thm:Karlin1}, pick any $\beta\in(0,\alpha/4)$ and put $\theta:=n^{\alpha/4-\beta}$. Appealing to Markov's inequality and \eqref{eq:depois} once again we conclude that
$$
K(\varepsilon)\le \sum_{n\ge n_0}\eee^{-\varepsilon n^{\alpha/4-\beta}}\exp\Big(\frac{n^{\alpha/2-2\beta}(b_j(n+d_n)-b_j(n))}{2a_j(n)f( a_j(n))}\eee^{\frac{n^{\alpha/4-\beta}}{(a_j(n)f(a_j(n)))^{1/2}}}\Big).
$$
Recalling \eqref{eq:var}, \eqref{eq:momtons}, \eqref{eq:varalone} and \eqref{eq:asymp_a} we obtain
$$
\frac{n^{\alpha/2-2\beta}(b_j(n+d_n)-b_j(n))}{a_j(n)f(a_j(n))}\sim\text{\rm const} \frac{n^{(\alpha-1)/2-2\beta}(\log\log n)^{1/2}}{f(n)}~\to~ 0,\quad n\to\infty,
$$
and
$$
\frac{n^{\alpha/4-\beta}}{(a_j(n)f(a_j(n)))^{1/2}}~\sim~ n^{-\alpha/4-\beta}(L^\ast(n)f(n))^{-1/2}~\to~ 0,\quad n\to\infty,
$$
where $L^\ast=\hat L$ if $\alpha=1$ and $j=1$ and $L^\ast=L$ otherwise. Hence, there exists a positive integer $N_1$ such that
$$K(\varepsilon)\le N_1+\eee\sum_{n\ge N_1} \eee^{-\varepsilon n^{\alpha/4-\beta}}<\infty.$$

\noindent {\sc Step 3.} LILs \eqref{eq:LILkar}, \eqref{eq:LILkar1} and \eqref{eq:LILkar2} particularly imply that
$$
\limsup_{n\to\infty}\frac{K_j(n)-\me K_j(n)}{(2\var K_j(n)f(\var K_j(n)))^{1/2}}\le C\quad\text{a.s.},
$$
where the constant $C$ depends on a setting and is equal to the right-hand side of \eqref{eq:LILkar}, \eqref{eq:LILkar1} or \eqref{eq:LILkar2}, respectively.
This taken together with the conclusions of Steps 1 and 2 enables us to conclude that
$$\limsup_{n\to\infty}\frac{\mathcal{K}_j(n)-\me \mathcal{K}_j(n)}{(2\var \mathcal{K}_j(n)f(\var \mathcal{K}_j(n)))^{1/2}}\le C\quad \text{a.s.}$$

It is shown in Lemma \ref{lem:new} that, for any $\delta>0$ and the deterministic sequence $(\tau_n)$ defined in \eqref{eq:tau},
$$\limsup_{n\to\infty}\frac{K_{j}(\tau_n)-\me K_{j}(\tau_n)}{C(2\var K_{j}(\tau_n)f(\var K_j(\tau_n)))^{1/2}}\ge 1-\delta\quad\text{a.s.}$$
The same proof may be used to check that the latter limit relation holds true with $\lfloor \tau_n \rfloor$ replacing $\tau_n$. Combining the resulting inequality with the conclusions of Steps 1 and 2 we obtain
\begin{multline*}
\limsup_{n\to\infty}\frac{\mathcal{K}_{j}(n)-\me \mathcal{K}_{j}(n)}{C(2\var \mathcal{K}_{j}(n)f(\var \mathcal{K}_j(n)))^{1/2}}\geq \limsup_{n\to\infty}\frac{\mathcal{K}_{j}(\lfloor \tau_n\rfloor)-\me \mathcal{K}_{j}(\lfloor \tau_n\rfloor)}{C(2\var \mathcal{K}_{j}(\lfloor \tau_n\rfloor)f(\var \mathcal{K}_j(\lfloor \tau_n\rfloor)))^{1/2}}\\ \ge 1-\delta\quad \text{a.s.}
\end{multline*}
Since the left-hand side does not depend on $\delta$, we infer $$\limsup_{n\to\infty}\frac{\mathcal{K}_{j}(n)-\me \mathcal{K}_{j}(n)}{(2\var \mathcal{K}_{j}(n)f(\var \mathcal{K}_j(n)))^{1/2}}\geq C\quad\text{a.s.}$$
\end{proof}

\noindent {\bf Acknowledgement}.
The authors thank the two anonymous referees for useful comments which helped to improve the presentation, and also Lutz Mattner for a pointer to \cite{Volkova:1996}. DB was supported by the National Science Center, Poland (Opus, grant number 2020/39/B/ST1/00209). The research of AI was supported by UC Berkeley Economics/Haas in the framework of the U4U program. VK acknowledges the support by the National Science Center, Poland under the scholarship programme of the Polish National Science Center for students and researchers from Ukraine without a PhD degree, carried out within the framework of the Basic Research Programme under the 3rd edition of the EEA and Norway Grants 2014-2021.

\end{document}